\documentclass[11pt, twoside]{article}

\usepackage[utf8]{inputenc}
\usepackage{url, amsmath,enumerate,fancyhdr,amssymb, amsthm, pstricks, epsf, lscape, ifpdf, graphicx, float, mathtools}

\DeclarePairedDelimiter{\ceil}{\lceil}{\rceil}

\DeclareMathOperator*{\argmin}{argmin}

\usepackage{blindtext}
\usepackage{relsize}
\usepackage[dvipsnames]{xcolor}

\colorlet{myblue}{blue!70!black} % tweak 70→60/80 to taste

\usepackage{algorithm}
\usepackage{algpseudocode}
\usepackage{bbm}
\usepackage{soul}

\usepackage{geometry}
\usepackage{setspace}

\numberwithin{equation}{section}

\usepackage{hyperref}
 \hypersetup{colorlinks,
 citecolor=blue,
 filecolor=blue,
 linkcolor=black,
 urlcolor=blue}

 \usepackage[title]{appendix}

\newcommand{\norm}[1]{\left\lVert#1\right\rVert}

\mathchardef\mhyphen="2D

\newtheorem{theorem}{Theorem}[section]
\newtheorem{lemma}[theorem]{Lemma}

\theoremstyle{definition}

\newtheorem{proposition}[theorem]{Proposition}
\newtheorem{corollary}[theorem]{Corollary}

\newtheorem{remark}[theorem]{Remark}

\usepackage{booktabs,tabularx,threeparttable}
\usepackage{amssymb}
\usepackage{hyperref}

\newcolumntype{Y}{>{\raggedright\arraybackslash}X}
\newcolumntype{Z}{>{\centering\arraybackslash}X}

\newcommand{\FOSP}{\(\epsilon\)-FOSP}
\newcommand{\SOSP}{\(\epsilon\)-SOSP}

\newcolumntype{P}[1]{>{\centering\arraybackslash}p{#1}}

\newcommand{\IterRow}[3]{#1 & #2 & #3\\}

\makeatletter
\DeclareMathOperator*{\infd}{inf\vphantom{\operator@font p}} % inf with descender
\DeclareMathOperator*{\supx}{su\smash{\operator@font p}} % sup without descender
\makeatother

\usepackage[numbers]{natbib}
\usepackage{fancyhdr}

\begin{document}

\title{\textbf{A Sequential Cubic Programming Method with Second-Order Complexity
Guarantees for Equality Constrained Optimization}} 
\author{Nikos Dimou$^{1,}\footnote{Email: dimou@unc.edu.}$\; and Michael J. O'Neill$^{1,}\footnote{Email: mikeoneill@unc.edu.}$}

\date{$^1$Department of Statistics and Operations Research \\ University of North Carolina at Chapel Hill}

\maketitle

\pagestyle{myheadings}
\pagestyle{plain}

\begin{abstract}
We develop a new method for equality constrained optimization problems based on a sequential cubic programming framework. Each iteration utilizes a step decomposition based on the Jacobian of the constraints into a normal and a tangential component, the latter of which is found by solving a subproblem involving cubic regularization. The method incorporates second-order correction steps as necessary to ensure global convergence to second-order stationary points as well as local quadratic convergence. In addition, we show that the algorithm is the first to obtain worst case complexity guarantees on the order of $\mathcal{O}(\epsilon_g^{-3/2})$ for the gradient of the Lagrangian, $\mathcal{O}(\epsilon_H^{-3})$ in terms of second-order stationarity, and $\mathcal{O}(\epsilon_c^{-1})$ in terms of the constraint violation. These  are the best known complexity guarantees of any method for this class of problems.
\end{abstract}

\section{Introduction}\label{sec:sec1}

In this paper, we consider the equality constrained optimization problem
\begin{equation}\label{eq:1}
    \min\limits_{x\in\mathbb{R}^{n}} f(x)\;\;\text{s.t.}\;\;c(x)=0,
\end{equation}
where $f:\mathbb{R}^{n}\to\mathbb{R}$ and $c:\mathbb{R}^{n}\to\mathbb{R}^{m}$ are smooth, non-convex functions. Problems of this form arise naturally in a number of areas, such as resource allocation, optimal control and PDE-constrained optimization \cite{bertsekas1998network, betts2010practical, hinze2008optimization}. For a general non-convex optimization problem it is well-known that finding a local minimum is an intractable problem. Therefore, existing algorithms typically target approximate first-order stationarity instead of global optimality. While the notion of stationarity that appears in the literature for this class of problems varies, perhaps the most standard one is defined in terms of the Lagrangian function: A point $x\in\mathbb{R}^{n}$ is said to be a $(\epsilon_g,\epsilon_c)$-first order stationary point (FOSP) for (\ref{eq:1}) if there exists $\lambda \in \mathbb{R}^{m}$ such that 
\begin{equation}\label{eq:fosp}
    \norm{\nabla f(x) + \nabla c(x) \lambda}\leq \epsilon_g\;\;\text{and}\;\;\norm{c(x)}_1\leq \epsilon_c\;.
\end{equation}

However, the prevalence of saddle points in non-convex optimization limits guarantees of convergence to global minima. A simple way of overcoming this issue is to search for second-order approximate solutions: We say that a point $x\in\mathbb{R}^{n}$ is a $(\epsilon_g,\epsilon_c,\epsilon_H)$-second order stationary point (SOSP) for (\ref{eq:1}) if there exists $\lambda\in\mathbb{R}^{m}$ so that (\ref{eq:fosp}) holds and
\begin{equation}\label{eq:sosp}
    d^\top \left(\nabla^2 f(x) +\sum_{i=1}^{m}\lambda_i \nabla^2 c_i(x) \right)d\geq -\epsilon_H \norm{d}^2
\end{equation}
for all $d\in\mathbb{R}^{n}$ such that $\nabla c(x)^\top d=0$. Under reasonable assumptions and certain constraint qualifications, the target points (\ref{eq:fosp})-(\ref{eq:sosp}) are able to fully capture the notion of approximate second-order KKT points \cite[Chapter 12]{nocedal2006numerical}.\par

Recent work emphasizes iteration complexity results, providing practical performance guarantees that go beyond mere convergence, by evaluating the number of steps required to reach a target tolerance. In the unconstrained setting, the optimal iteration complexity bound to $(\epsilon_g,\epsilon_H)$-SOSPs (when one ignores the constraints in \eqref{eq:fosp}-\eqref{eq:sosp}) is of order $\mathcal{O}\left(\max\left\{\epsilon_g^{-3/2},\epsilon_H^{-3}\right\}\right)$, under the assumption that the objective function is twice Lipschitz differentiable. This result is due to the fundamental work of Nesterov and Polyak \cite{nesterov2006cubic} who proposed a cubic regularization method, which was later extended by Cartis et al. \cite{cartis2011adaptive} with the introduction of the ARC algorithm.\par

Soon after their development for unconstrained optimization, cubic regularization methods were extended to various constrained settings, including the equality constrained case. Specifically, Cartis et al. \cite{cartis2013evaluation} developed a two-phase method, in which a first phase seeks an approximate feasible solution, and a second phase seeks optimality while retaining approximate feasibility. Although this approach tends to be undesirable in practice, it achieves optimal complexity bounds of order $\mathcal{O}\left(\epsilon^{-3/2}\right)$ with respect to the gradient of the Lagrangian. Similar complexity bounds were also attained in a two-phase trust funnel algorithm developed in \cite{curtis2018complexity}. This was inspired by the TRACE algorithm for unconstrained optimization \cite{curtis2017trust}, as well as by sequential quadratic programming methods (SQP), the latter of which are well-known for their superior practical performance compared to other single-phase methods in non-linear optimization \cite{nocedal2006numerical}. In fact, a number of SQP algorithms have been designed to tackle (\ref{eq:1}) that offer both global convergence \cite{berahas2021sequential, berahas2024modified, curtis2024sequential} and good complexity guarantees \cite{berahas2025sequential, curtis2024worst}.\par

\subsection{Contributions}\label{sec:sec1.1}

Aside from their practical performance differences, a striking imbalance characterizes the  theoretical guarantees established by single-phase and two-phase schemes: While two-phase methods enjoy optimal complexity guarantees with respect to the gradient of the Langrangian, the state-of-the-art SQP methods demonstrate improved constraint violation bounds, none of which have yet been established for higher-order stationarity. Although numerous second-order approaches have been developed for equality constrained optimization problems (see Table \ref{tab:table}), there yet does not exist an algorithm that simultaneously achieves optimal complexity bounds with respect to the gradient and Hessian of the Lagrangian, and ensures sufficiently good complexity guarantees with respect to the constraint violation. In more concrete terms, the question of whether the worst-case complexity guarantees established for the unconstrained optimization setting carry over to the equality constrained one when second-order derivative information is available, remains open.

The aim of this paper is to give a definitive answer to this question, by developing a new method for solving equality constrained optimization problems. More specifically, we design a sequential cubic programming method (SCP) that converges to a $(\epsilon_g,\epsilon_c)$-FOSP in at most $\mathcal{O}\left(\max\left\{\epsilon_g^{-2},\epsilon_c^{-1}\right\}\right)$ iterations when only first-order derivative information is available, and to a $(\epsilon_g,\epsilon_c,\epsilon_H)$-SOSP in at most $\mathcal{O}\left(\max\left\{\epsilon_g^{-3/2},\epsilon_c^{-1},\epsilon_H^{-3}\right\}\right)$ steps when the objective and constraints are twice Lipschitz differentiable. An immediate consequence of our result is that the worst-case number of steps needed for the first bound of (\ref{eq:fosp}) and (\ref{eq:sosp}) to be satisfied, coincides with that of the unconstrained case \cite{cartis2011adaptive2, nesterov2006cubic}, which is known to be sharp \cite{cartis2012complexity}. Importantly, the order $\mathcal{O}(\epsilon_c^{-1})$ in terms of the constraint violation is a significant improvement of the previously known bound $\mathcal{O}(\epsilon_c^{-2})$ for second-order constrained optimization \cite{goyens2024computing}, and matches the current benchmark for this class of problems in terms of first-order criticality \cite{berahas2025sequential, curtis2024worst}. In other words, we show that the worst-case number of steps needed to approximately minimize a smooth non-convex function is not altered when smooth non-convex equality constraints are also considered, unless one strongly focuses on feasibility $\left(\text{i.e.,\;}\epsilon_c\leq \min\left\{\epsilon_g^{3/2},\epsilon_H^3\right\}\right)$\footnote{Although one has the freedom of enforcing this tolerance relation, it is not preferred theoretically or in practice; most related works typically target approximate criticality with $\epsilon_c=\epsilon_g=\epsilon_H$ or $\epsilon_c=\epsilon_g=\epsilon_H^2$ (see Table \ref{tab:table} and Appendix \ref{sec:AppA}).}.\par

Our approach is based on a sequential cubic programming framework, which relies on classical works in both sequential programming and cubic regularization settings. In particular, in every iteration we formulate and solve a subproblem based on a local linearization of the constraints and a local cubic approximation of the Lagrangian function. Every iterate involves a step decomposition into a normal and a tangential component, the former of which is chosen to satisfy the linearized feasibility constraints in an inexact manner, while the tangential step is chosen to be an approximate
minimizer of an unconstrained cubic problem. Cubic regularization subproblems have been extensively studied in the relevant literature and have been shown to be efficiently solvable, thus rendering our method computationally attractive. A significant characteristic of the proposed method is the lack of dependence on initialization conditions or penalty parameters, the latter of which often lead to key restrictions in the implementation of two-phase and augmented Lagrangian methods, respectively (see Appendix \ref{sec:AppA}).

Apart from ARC and SQP frameworks, the proposed algorithm is also inspired by trust-region methods. In the early work of Byrd et al. \cite{byrd1987trust}, it was realized that both the objective and constraint violation might dramatically increase when iterates lie around saddle points, thereby prohibiting global convergence. This phenomenon, also known as the \textit{Maratos effect}, can be addressed by introducing second-order correction steps. Besides effectively adopting this methodology as to ensure nice global convergence properties, we modify the criterion of \cite{byrd1987trust} for invoking such correction, by utilizing the known connection between trust region radius and cubic regularization parameters. The latter mechanism also ensures local quadratic convergence properties, a feature uncommon in the surrounding literature\footnote{The only relevant work we are aware of that both ensures complexity bounds and claims local quadratic convergence to second-order points in equality constrained optimization is that of Goyens et al. \cite{goyens2024computing}.}. It is worth mentioning that, although the method developed by Byrd et al. \cite{byrd1987trust} includes several ingredients necessary for both global and fast convergence warranties to second-order stationary points, a direct iteration complexity analysis of their algorithm would likely result in worse bounds than the ones presented here, due to standard trust-region methods being suboptimal \cite{cartis2010complexity}.

\subsection{Related work}\label{sec:sec1.2}

In the last decade, a plethora of works have been actively developed for solving equality constrained optimization problems. A line of cubic regularization works by Cartis, Gould and Toint \cite{cartis2011evaluation, cartis2014complexity} showed iteration complexity bounds to first-order critical points that match those of steepest descent. A two-phase method by the same authors \cite{cartis2019optimality} extended these theoretical results to higher-order stationary points. A trust region SQP method was recently proposed by Fang et al. \cite{fang2025high}, which established high-probability iterations complexity bounds for identifying first- and second-order stationarity. Similar trust-region SQP methods have previously appeared in both the deterministic \cite{fletcher2002global} and the stochastic \cite{fang2024fully} settings, offering global convergence guarantees. Aside from the standard quadratic and cubic methods, a number of works have exploited smoothness properties of the augmented Lagrangian function. Improving on the complexity bounds of Xie and Wright \cite{xie2021complexity}, He et al. \cite{he2023newton} suggested a Newton-CG augmented Lagrangian algorithm and gave a range of iteration complexity guarantees with dependence on a generalized linear independence constraint qualification condition. An augmented Lagrangian algorithm was introduced by Grapiglia and Yuan \cite{grapiglia2021complexity}. Their complexity bounds boil down to the ones presented here, under the additional assumption that the constraints are linear.  Bai and Mei \cite{bai2018analysis} represents an early work in complexity for SQP methods which established fast local linear convergence properties via a first-order SQP method over the same manifold framework. Further, a variation of the two-phase method of \cite{cartis2013evaluation} has been extended to higher orders by Birgin et al. \cite{birgin2016evaluation}. Lastly, a linearized quadratic penalty method was developed by Bourkhissi and Necoara \cite{bourkhissi2025complexity} with convergence and complexity warranties to first-order critical points. We summarize the iteration complexity bounds of the preceding works in Table \ref{tab:table}, and provide additional details in Appendix \ref{sec:AppA}.\par

\begin{table}[t]
  \centering
  \small
  \begin{threeparttable}
    
    \begin{tabularx}{\textwidth}{P{4.8cm}    P{5.4cm}    P{5.0cm}}
      \toprule
      \textbf{Paper} & \textbf{\FOSP\ complexity} & \textbf{\SOSP\ complexity} \\
      \midrule
      
      \IterRow{\textbf{This work}}%
              {\(\mathcal{O}\big(\max\left\{\epsilon_g^{-2},\epsilon_c^{-1}\right\}\big)\)}%
              {\(\mathcal{O}\Big(\max\left\{\epsilon_g^{-3/2},\epsilon_c^{-1},\epsilon_H^{-3}\right\}\Big)\)}
      \IterRow{\cite{goyens2024computing}$^{\dag,\ddag}$}%
              {\(\mathcal{O}\big(\max\left\{\epsilon_g^{-2},\epsilon_c^{-2}\right\}\big)\)}%
              {\(\mathcal{O}\big(\max\left\{\epsilon_g^{-2},\epsilon_c^{-2},\epsilon_H^{-3}\right\}\big)\)}
      \IterRow{\cite{berahas2025sequential, curtis2024worst}}%
              {\(\mathcal{O}\big(\max\left\{\epsilon_g^{-2},\epsilon_c^{-1}\right\}\big)\)}%
              {—}
              \IterRow{\cite{cartis2013evaluation, curtis2018complexity}$^\dag$}%
              {\(\mathcal{O}\big(\epsilon_P^{-1/2}\epsilon_D^{-3/2}\big)\) (H)}%
              {—}
      \IterRow{\cite{fang2025high}}%
              {\(\mathcal{O}\big(\epsilon^{-2}\big)\) (H)}%
              {\(\mathcal{O}\big(\epsilon^{-3}\big)\)}
      \IterRow{\cite{cartis2019optimality}$^{\dag,\ddag}$}%
              {\(\mathcal{O}\big(\epsilon_P^{-1}\epsilon_D^{-2}\big)\)}%
              {\(\mathcal{O}\big(\max\{\epsilon_P^{-1},\epsilon_P^{-2}\epsilon_D^{-3}\}\big)\)}
      \IterRow{\cite{cartis2011evaluation, cartis2014complexity}$^{\dag,\ddag}$}%
              {\(\mathcal{O}\big(\epsilon^{-2}\big)\)}
              {—}
      \IterRow{\cite{he2023newton}}%
              {—}%
              {\(\widetilde{\mathcal{O}}\big(\epsilon_g^{-2}\max\left\{\epsilon_g^{-2}\epsilon_H,\epsilon_H^{-3}\right\}\big)\)}
               \IterRow{\cite{bai2018analysis}$^{\dag,\ddag}$}%
              {\(\mathcal{O}\big(\epsilon^{-4}\big)\)}
              {—}
              \IterRow{\cite{bourkhissi2025complexity}}%
              {\(\mathcal{O}\big(\epsilon^{-5/2}\big)\)}%
              {—}
    \IterRow{\cite{xie2021complexity}}%
              {\(\mathcal{O}\big(\epsilon^{-11/2}\big)\) (H)}%
              {\(\mathcal{O}\big(\epsilon^{-7}\big)\)}
      \IterRow{\cite{grapiglia2021complexity}$^{\ddag}$}%
              {\(\mathcal{O}\big(\epsilon^{-2/(\alpha-1)}\big)\), \(\alpha>1\)}%
              {—}    \IterRow{\cite{birgin2016evaluation}$^{\dag,\ddag}$}%
              {\(\mathcal{O}\big(\epsilon_P\,\epsilon_D^{-3/2}\min\{\epsilon_D,\epsilon_P\}^{-3/2}\big)\) (H) }%
              {—}
      \bottomrule
    \end{tabularx}
    \caption{\small{Best known iteration complexity bounds to first and second-order stationary points for equality constrained optimization problems. The bounds given in terms of the unscripted tolerance $\epsilon>0$ typically mean that $\epsilon=\epsilon_g=\epsilon_c=\epsilon_H$, with the exception of \cite{cartis2011evaluation, cartis2014complexity}. (H) Second-order derivative information (Lipschitz continuity of the Hessian of the objective and/or constraints) is used for first-order stationarity. $(\dag)$ The stationary points targeted are different to (\ref{eq:fosp})-(\ref{eq:sosp}). $(\ddag)$ The constrained optimization setting considered is more generic than \eqref{eq:1}. We refer the reader to Appendix \ref{sec:AppA} for more details about the different types of target points and optimization settings of these works.}}
    \label{tab:table}
  \end{threeparttable}
\end{table}

The method that gives the currently best iteration complexity bounds for second-order stationary points is that of Goyens et al. \cite{goyens2024computing}. There, the authors utilize Fletcher's augmented Lagrangian function to tackle non-convex equality constrained optimization problems defined over manifolds. The demonstrated complexity bounds match those of the unconstrained case when only first-order information is available, but fall short of the optimal bound with respect to the gradient of the Lagrangian due to the fact that ordinary gradient steps are used for stationarity, despite access to second-order derivatives. Moreover, the bound on the constraint violation is less sharp than the one acquired here, and is expected to persist even if their underlying subproblem is implemented via a cubic regularization or trust-region method. This is due to the fact that employing Fletcher's augmented Lagrangian as the merit function naturally leads to the same order of complexity in terms of $\epsilon_g$ and $\epsilon_c$, unlike when the $\ell_1$ penalty function is used as a merit function, as it is in the work. Additionally, Fletcher's augmented Lagrangian function is considered computationally expensive, reducing practical application \cite[Section 15.4]{nocedal2006numerical}. Furthermore, the Hessian of the augmented Lagrangian computed in \cite{goyens2024computing} requires third-order derivative information of $f$ and $c$, thus making Fletcher's function less attractive to non-smooth merit functions, such as the $\ell_1$-merit function adopted here, which only require second-order derivatives.

The work that is algorithmically closest to ours is the recent one by Pei et al. \cite{pei2023sequential}, which was the first to consider a sequential programming framework based on a cubic regularization subproblem. However, a number of differences can be identified between the two works, which can in turn classify this paper as a strengthened version of the former. First, the authors of  \cite{pei2023sequential} established global convergence properties to first-order stationary points under stronger positive semidefinitiness assumptions. Importantly, second-order stationarity, iteration complexity bounds and local convergence properties were not established. Moreover, the analysis of \cite{pei2023sequential} required null-space basis as opposed to just a projection matrix, which may cause fill-in for large scale problems (see Section \ref{sec:sec2}). Further, although the normal and tangential components of the two works originate from the same subproblems, the steps considered here are solved in an inexact manner (Appendix \ref{sec:AppC}), yielding a more computationally favorable algorithm. Lastly, a clear difference can be spotted in the update formulas for the merit and cubic parameters. More specifically, our work provides a stronger predicted reduction at each iteration and avoids pathological algorithmic defects, such as allowing the cubic parameter to vanish.

Finally, an increasing number of works rises in the literature of more general constrained optimization problems. First, (\ref{eq:1}) has been studied in frameworks with additional inequality constraints, as one can already see from Table \ref{tab:table}. To be exact, non-convex inequality constrained problems have been studied via augmented Lagrangian \cite{grapiglia2021complexity}, SQP \cite{curtis2024sequential}, and two-phase \cite{birgin2016evaluation} methods. Moreover, optimization problems with constraints of the form ``$x\in C$", where $C$ represents a closed convex set, have been recently studied and good complexity guarantees have been established. A Frank-Wolfe algorithm was developed in \cite{nouiehed2018convergence} that converges to $(\epsilon_g,\epsilon_H)$-SOSPs in at most $\mathcal{O}\big(\max\left\{\epsilon_g^{-2},\epsilon_H^{-3}\right\}\big)$ iterations. A ghost-penalty methodology was proposed in \cite{facchinei2021ghost} which reaches a scaled KKT first-order critical point in either $\mathcal{O}\big(\epsilon^{-4}\big)$ or $\mathcal{O}\big(\epsilon^{-2}\big)$ total steps, depending on various (feasibility) assumptions. Better complexity results for similar critical points of higher-orders were given in \cite{Cartis2019}, under a modified framework of the ARC algorithm \cite{cartis2011adaptive}. The fundamental problem of escaping saddle points, which is related to the Maratos effect, has been studied in this convex-constrained setting by Mokhtari et al. \cite{mokhtari2018escaping}, and good complexity guarantees to second-order criticality have been given.

\subsection{Organization}\label{sec:sec1.3}

The rest of the paper is organized as follows. In Section \ref{sec:sec2} we formally introduce the sequential cubic optimization algorithm and emphasize on computational approximations of all steps considered. Numerical methods for computing the inexact trial step are described in Appendix \ref{sec:AppC}. In Section \ref{sec:sec3} we analyze global convergence qualities and iteration complexity bounds to first-order stationary points. Section \ref{sec:sec4} includes the analysis to second-order points, when Lipschitz continuity of the Hessian of the Lagrangian is provided. The main convergence and complexity results of this paper appear in this section. In Section \ref{sec:sec5} we show that our algorithms enjoys local quadratic convergence properties. Finally, in Section \ref{sec:sec6} we briefly discuss future directions and set some open questions related to our work.

\subsection{Notation}\label{sec:sec1.4}
 
Unless otherwise indicated, $\norm{\cdot}$ denotes the Euclidean
$\ell_2$-norm for vectors and matrices, and $\norm{\cdot}_1$ denotes the $\ell_1$-norm for vectors. Let $\mathcal{R}(\cdot)$ and $\mathcal{N}(\cdot)$ denote the range and null spaces, respectively. We define $\mathcal{X}$ to be an open convex set containing all iterate points $\{x_k\}$ (including trial steps as well as second-order correction steps) of the proposed algorithm. For every iteration $k\geq 0$, let $f_k := f(x_k)$, $g_k := \nabla f(x_k)$ and $\nabla^2 f_k:=\nabla^2f(x_k)$ be the value, gradient and Hessian of the objective function,
respectively. Further,
$A_k := \nabla c(x_k)^\top\in\mathbb{R}^{m\times n}$ denotes the Jacobian matrix of the constraints, and $a^i_k$ represents its $i$-th column. We write the Hessian matrix of the $i$-th constraint function by $\nabla^2c_k^i:=\nabla^2c^i(x_k)$. We define the Lagrangian function $\mathcal{L}(x_k,\lambda_k):=f_k+\lambda_k^\top c_k$ at iteration $k\geq 0$, where $\lambda_k\in\mathbb{R}^{m}$ are the dual variables. In addition, given a matrix $B$, $\lambda_{\min}(B)$ denotes its minimum eigenvalue, and $\sigma_{\min}(B)$ its lowest singular value. For functions $h_1:\mathbb{R}\to\mathbb{R}$, $h_2:\mathbb{R}\to[0,\infty)$, we write
$h_1(\cdot)=\mathcal{O}\big(h_2(\cdot)\big)$ in order to indicate that
$\lvert h_1(\cdot)\rvert \leq C h_2(\cdot)$ for some $C>0$. Lastly, we denote the cardinality of a set $\mathcal{A}$ by $|\mathcal{A}|$.

%\nd{Add more stuff here.}

\section{Sequential Cubic Programming Method}\label{sec:sec2}

We develop an iterative algorithm which requires computing a trial step in each iteration. In order to compute this trial step, we follow the approach originally proposed in \cite{pei2023sequential}. There, a cubic sequential method is introduced and global convergence to first-order critical points is shown. However, second-order stationarity was not considered, and complexity theory was out of the scope of the latter. Here, we modify their method in order to achieve global and local convergence to higher order critical points and guarantee the best known iteration complexity bounds for problems of the form (\ref{eq:1}).\par

At a high level, the proposed method approximately minimizes, at each iteration, a local cubic model of the Lagrangian subject to a relaxed linearization of the constraints:

\vspace{-1cm}\begin{center}
\begin{equation*}{(SCP_k)\;\;\;\;\;\;\;\;}
\begin{array}{ll}
    \min\limits_{d\in\mathbb{R}^n}m_k(d):=\displaystyle f_k+g_k^\top d+\frac{1}{2}d^\top H_k d +\frac{ \sigma_k}{3}\norm{d}^3\\
    \hspace{0.08cm}\text{s.t.}\;\; A_k d+\beta_k c_k=0
    \end{array}
\end{equation*}
\end{center}
Here, $\sigma_k>0$ represents the cubic regularization parameter that characterizes cubic regularization methods \cite{cartis2011evaluation}. The scalar $\beta_k\in(0,1]$ appears in trust-region methods (e.g., \cite{byrd1987trust}) and acts as a feasibility-control parameter. This relaxation will allows us to work towards simultaneously reducing both the objective function and the constraint violation. The symmetric matrix $H_k$ represents the Hessian of the Lagrangian at iteration $k$, defined by
\begin{equation}\label{eq:Hessian}
   H_k:=\nabla^2 f_k+\sum_{i=1}^{m}\lambda_{k}^{i}\nabla^2 c_k^{i}. 
\end{equation}
Although approximations are plausible in a similar manner to \cite{cartis2011adaptive, cartis2011adaptive2}, we choose to work with the exact Hessian to simplify the exposition; we defer the development and analysis of such approximations to future work.  \par
The $(SCP_k)$ subproblem represents a special case of constrained cubic regularization optimization problems, which are in general hard to solve. However, it can be solved via a step decomposition technique. In particular, we follow a decomposition that originates from the theory of sequential quadratic programming \cite[Chapter 18]{nocedal2006numerical}. Namely, the step $d_k$ is decomposed into two components, as follows:
\begin{equation}\label{eq:decomposition}
    d_k=v_k+u_k .
\end{equation}
Here, $v_k$ is referred to as the \textit{normal} step which lies in $\mathcal{R}(A_k^\top)$, and $u_k$ represents the \textit{tangential} step, which belongs to $\mathcal{N}(A_k)$. A standard choice for the normal step in SQP methods is based on exact feasibility. That is, for the $(SCP_k)$, one has (see \cite{pei2023sequential})
\begin{equation}\label{eq:normalstep1}
    v_k=\beta_k v_k^c,
\end{equation}
where
\begin{equation}\label{eq:normalstep2}
    v_k^c=-A_k^\top(A_k A_k^\top)^{-1}c_k.
\end{equation}
To guarantee that $v_k^c$ is well-defined at every iteration, we need an assumption that is commonly known as \textit{linear independence constraint qualification} (LICQ). The latter implies that $A_k$ has full row rank for all $k\geq 0$, so that the inverse of $A_k A_k^\top$ is always well-defined. The LICQ assumption is given explicitly, among others, at the beginning of Section \ref{sec:sec3}.\par

One of the main goals of this paper is to prove the desired convergence and complexity results by not requiring exact solutions for any of the subproblems considered. For that reason, instead of choosing $v_k$ as in (\ref{eq:normalstep1})-(\ref{eq:normalstep2}), we solve the corresponding least-squares problem inexactly. In fact, the normal step $v_k$ is chosen as in (\ref{eq:normalstep1}), where $v_k^c$ lies in $\mathcal{R}(A_k^\top)$ and satisfies\footnote{We require the $\ell$-1 norm on the left-hand side of this inequality in order to avoid factors of $\sqrt{m}$ appearing in the analysis due to equivalence of norms. This condition could be modified to the $\ell$-2 norm and our analysis would still hold, albeit with these additional factors.}
\begin{equation}\label{eq:normalstep3}
   \norm{A_k v_k^c+c_k}_1\leq r_v \min\left\{\norm{c_k}_1,\norm{v_k^c}^3\right\},
\end{equation}
for $r_v\in[0,1-\tau)$, where $\tau\in(0,1)$ is some (user-chosen) constant. Of course, if $c_k=0$, then we simply set $v_k=v_k^c=0$. Although the upper bound on the residual seems to be strongly correlated with the step $v_k^c$, there are several numerical methods for verifying \eqref{eq:normalstep3} without computing an exact optimal solution of the least squares problem. For illustrative purposes, in Appendix \ref{sec:AppC}, we show that the standard LSQR method (which is suitable when the Jacobian is large and sparse) will always terminate with a point $v_k^c$ satisfying \eqref{eq:normalstep3}, and we explicitly establish the minimum number of iterations needed to do so.\par

In addition, given $\norm{c_k}_1\neq 0$, we choose the feasibility parameter $\beta_k$ in a way similar to \cite{byrd1987trust}, namely
\begin{equation}\label{eq:beta}
    \beta_k\in\left[\min\left\{1,\frac{\theta}{\norm{v_k^c}\sqrt{\sigma_k}}\right\},\min\left\{1,\frac{1}{\norm{v_k^c}\sqrt{\sigma_k}}\right\}
  \right],
\end{equation}
for some $\theta \in (0,1]$. This is well-defined, as $\|v_k^c\|\neq 0$ whenever $\norm{c_k}_1\neq 0$, due to \eqref{eq:normalstep3}. The parameter $\beta_k$ controls the length of the trial step in the range space of the Jacobian, meaning that it ensures $\norm{v_k}\leq \frac{1}{\sqrt{\sigma_k}}$ at every iteration. In the theory of cubic regularization methods, it is oftentimes convenient to think of the quantity $\frac{1}{\sqrt{\sigma_k}}$ as a trust-region radius, since it fundamentally plays the same role. This is also the case for our analysis, as we implicitly restrict the search for a trial step to a bounded region governed by the cubic (penalty) parameter.

Assuming that the normal step approximately satisfies the constraints of the problem $(SCP_k)$, and restricting the cubic regularization to the null space of the Jacobian only, $(SCP_k)$ reduces to an unconstrained cubic regularization subproblem of the form
\begin{equation} \label{eq:mku}
    \min\limits_{u\in\mathbb{R}^n}m_k^U(u):=f_k+ \widetilde{g}_k ^\top u+\frac{1}{2}u^\top \widetilde{H}_k  u +\frac{ \sigma_k}{3}\norm{P_k u}^3,
\end{equation}

\noindent where $P_k$ represents the orthogonal projection matrix of $\mathcal{N}(A_k)$, and $\widetilde{g}_k:=P_k^\top (g_k+ H_k v_k)$, $\widetilde{H}_k:=P^\top_k H_k P_k$. We assume that the computed global minimizer of \eqref{eq:mku} lies in the null space of the Jacobian. This assumption is always satisfied whenever a Krylov subspace method is applied, a number of which have already been developed for this class of problems; see, for instance, \cite{bellavia2025regularized, carmon2018analysis, cartis2011adaptive, jia2022solving}. It should be highlighted that the LICQ assumption allows one to also work with an orthonormal basis $Z_k$ of $\mathcal{N}(A_k)$ instead of the projection matrix $P_k$ (and thus with the usual tangential step $Z_k u_k$ \cite{nocedal2006numerical}) in \eqref{eq:decomposition}, to form another equivalent null space reduction of $(SCP_k)$. In fact, the replacement of $P_k$ with $Z_k$ can be done effectively in our analysis with minimal changes in the proofs (see Lemma \ref{lem:dualvariables}). The question of choosing  one over the other, reduces to a matter of dimensionality: If the Jacobian is large and sparse, then it is cost-efficient to work with products of the form $P_k y$ via a Krylov subspace method (without forming the projection matrix explicitly), thus avoiding the fill-in issues associated with constructing an orthonormal basis. This advantage is stronger when the null space has high dimension, in which case $P_k$ equals the identity matrix plus a low-rank correction (see Appendix \ref{sec:AppB}). On the other hand, if the null space of $A_k$ has small dimension, it is usually preferable to directly compute an orthonormal basis $Z_k$ (e.g., via QR factorization), as it can facilitate the process of finding an approximate solution to \eqref{eq:mku}: the latter process relies on computing the leftmost eigenvalue of the reduced Hessian $Z_k^\top H_k Z_k$ (see \hyperlink{or3}{(OR3)} below), which would then be a small dense matrix.\par

While the cubic regularization subproblem can be solved efficiently by the aforementioned Krylov methods, or others \cite{birgin2019newton, gould2010solving, lieder2020solving}, we only require approximate solutions. In particular, we assume the existence of an Oracle that produces sufficiently good reduction in the model of (\ref{eq:mku}) at every iteration. We define the predicted reduction of the model in the null space of $A_k$,
\begin{equation}
    \Delta m_k^U(u_k):=m_k^U(0)-m_k^U(u_k),
\end{equation}
for every candidate solution $u_k$. We consider the following:

\noindent\underline{\textit{SCP Oracle}}: At iteration $k\geq0$, given the quantities $\sigma_k,v_k,f_k,g_k,H_k$, return an (in)exact solution $u_k\in\mathcal{R}(P_k)$ of (\ref{eq:mku}) satisfying:
\hypertarget{or1}{}\hypertarget{or2}{}\hypertarget{or3}{}

    \noindent\hyperlink{or1}{(OR1)} $\Delta m_k^U(u_k)\geq \Delta m_k^U(u_k^c)$, where $u_k^c:=-\alpha_k^*\widetilde{g_k}$ for $\alpha_k^*:=\argmin_{\alpha\geq 0} m_k^U(-\alpha \widetilde{g_k})$.
    
    \noindent\hyperlink{or2}{(OR2)}  $\norm{\nabla m_k^U(u_k)}\leq\delta \sigma_k \norm{u_k}^2$, where $\delta\in(0,1/6)$.
    
     \noindent\hyperlink{or3}{(OR3)} $\lambda_{\min}(P_k^\top H_k P_k)\geq -\sigma_k\norm{u_k}$.

These conditions are standard in related works in both constrained and unconstrained optimization that rely on solutions of cubic subproblems \cite{cartis2011adaptive, cartis2013evaluation}. \hyperlink{or1}{(OR1)} corresponds to the minimum reduction required for convergence to first-order stationarity, which is provided by the Cauchy step. As $\widetilde{P}_k^\top  \widetilde{g}_k=\widetilde{g}_k$, the Cauchy coefficient $\alpha_k^C$ has the closed form solution \cite{carmon2019gradient}
\begin{equation}\label{eq:cauchy}
    \alpha_k^C=\frac{2}{R_k+\sqrt{R_k^2+4\sigma_k \norm{\widetilde{g}_k}}},\;\; \text{where}\;\; R_k:=\frac{\widetilde{g}_k^\top \widetilde{H}_k \widetilde{g}_k}{\norm{\widetilde{g}_k}^2}.
\end{equation}
Condition \hyperlink{or2}{(OR2)} demands a (usually) stricter sufficient decrease of the model compared to \hyperlink{or1}{(OR1)}, when one is interested in higher-order critical points. This differs slightly from the reduction criterion that appears in the early works of Cartis et al. \cite{cartis2011adaptive, cartis2013evaluation}. If we were to rigorously transfer their second criterion into our setting, we would acquire the alternative 
\hypertarget{or2prime}{}

    \noindent\hyperlink{or2prime}{(OR2')}  $\nabla m_k^U(u_k)^\top u_k=0$ and $\norm{\nabla m_k^U(u_k)}\leq \kappa_{\theta}\min\{1,\norm{u_k}\}\norm{\widetilde{g_k}}$ for some $\kappa_{\theta}\in(0,1)$.

We find ``$\nabla m_k^U(u_k)^\top u_k=0$" a strong requirement, thus we choose to resort to a KKT residual error of order $\norm{u_k}^2$ \hyperlink{or2}{(OR2)}, in the hope that this will be easier to satisfy in practice. It should be noted that \hyperlink{or2prime}{(OR2')} also works for our analysis (see Remark \ref{rem:criterion}).\par

Lastly, \hyperlink{or3}{(OR3)} translates to an implicit utilization of negative curvature, when such curvature exists, as required for convergence to second-order stationarity. Note that the third requirement also encourages the use of the projection matrix instead of an orthonormal basis, as the computation of the minimum eigenvalue of $\widetilde{H}_k$ can be done efficiently with the Lanczos process \cite{golub2000large}.

A natural question is the following: Can we guarantee that a point satisfying \hyperlink{or1}{(OR1)}-\hyperlink{or3}{(OR3)} always exists? The next lemma shows that this is the case.

%\nd{Note to Mike: Global or local optimal solution?}

\begin{lemma}\label{lem:uglobal}
    Suppose $u^*_k$ is a global minimizer of \eqref{eq:mku}. Then, $u_k^*$ satisfies \hyperlink{or1}{(OR1)}-\hyperlink{or3}{(OR3)}.
\end{lemma}

\begin{proof}
    We first observe that, if $u_k^*$ is an optimal solution, then its projection still minimizes \eqref{eq:mku}, due to the property $P_k=P_k^2$. Thus, we may assume without loss of generality that $u_k^*\in\mathcal{R}(P_k)$. As $\nabla m^U_k(u_k^*)=0$, the result follows from the fact that $\widetilde{g_k}\in\mathcal{N}(A_k)$ and \cite[Lemma 3.2]{cartis2011adaptive}.
\end{proof}

Lemma \ref{lem:uglobal} is, however, not informative. For large-scale or ill-conditioned problems, computing an exact global minimizer of the cubic model may be prohibitively expensive, and a Lanczos solver may require a high-dimensional Krylov subspace even to attain a highly accurate approximation, whereas any point $u_k$ satisfying (OR1)-(OR3) might be detected earlier on. Indeed, standard analyses of cubic regularization methods impose similar inexact model-minimization conditions and either expect or explicitly assume that the inner solver satisfies them before reaching a global minimizer \cite{cartis2011adaptive, cartis2011adaptive2, cartis2013evaluation}.\par

In the current work, we go a step further by showing that the classical Lanczos method will always terminate with a tangential step satisfying (OR1)-(OR3), and we precisely identify the minimum number of iterations needed for termination, under a standard geometric assumption that avoids the hard case (Proposition \ref{prop:lanczos}). When the hard case is present, a tangential step can still be recovered via a nested restarting procedure \cite{jia2022solving}.

\subsection{The Lagrangian multipliers}\label{sec:sec2.1}

The dual variables $\lambda_k\in\mathbb{R}^{m}$ are taken to be approximations of the least-squares estimators
\begin{equation}\label{eq:lag2}
    \lambda_k^*:=-(A_k A_k^\top)^{-1}A_k g_k,
\end{equation}
given that these are well-defined. Specifically, the Lagrangian multipliers satisfy
\begin{equation}\label{eq:lambdas}
    \norm{A_k\left(g_k+A_k^\top \lambda_k\right)}\leq r_{\lambda}\norm{v_k}
\end{equation}
for some $r_\lambda\geq0$. An important realization of this approximation arises at the limit of the SCP procedure: as the algorithm converges to feasibility, it drives the dual variables to the least-square estimators. As we will see in Section \ref{sec:sec5}, one of the main reasons why the least squares estimators are a suitable choice for the dual variables, is because they also enforce fast local quadratic convergence.

\subsection{Merit function}\label{sec:sec2.2}

To decide whether or not the step $d_k$ will be accepted at each iteration, we introduce the standard non-smooth $l_1$ merit function with parameter $\mu>0$,
\begin{equation}
    \varphi(x,\mu):=f(x)+\mu\norm{c(x)}_1.
\end{equation}
Further, we consider a local cubic model of the merit function at each iteration,
\begin{equation}
    q_k(d_k):=f_k+g_k^\top d_k +\frac{1}{2}d_k^\top H_k d_k+\frac{ \sigma_k}{3}\norm{d_k}^3+\mu_k\norm{c_k+A_k d_k}_1.
\end{equation}
We denote the decrease in the model by
\begin{equation}
    \Delta q_k:=q_k(0)-q_k(d_k).
\end{equation}
A unique feature of our method is how the merit parameter update is performed in every iteration. Particularly, at every iteration $k\geq 0$, we set
\begin{equation}\label{eq:merit1}
    \mu_k:= \begin{cases} 
      \nu \mu_k^c\; &\;\;\;\text{if}\;\;\mu_{k-1}<\mu_k^c,  \\
       & \\
      \mu_{k-1} &  \;\;\;\text{if}\;\;\mu_{k-1}\geq \mu_k^c,
   \end{cases}
\end{equation}
where $\nu>1$ is a fixed constant and
\begin{equation}\label{eq:merit2}
    \mu_k^c= \begin{cases} 
      0  &\;\;\; \text{if}\;\;\norm{c_k}_1=0, \\
       & \\
      \displaystyle\frac{g_k^\top v_k+\frac{1}{2}v_k^\top H_k v_k+\frac{\sigma_k}{3}\left[\norm{d_k}^3-\norm{u_k}^3\right]}{(1-r_v-\tau)\beta_k \norm{c_k}_1} &  \;\;\;\text{if}\;\;\norm{c_k}_1\neq 0,
   \end{cases}
\end{equation}
where $r_v$ and $\tau$ are defined in \eqref{eq:normalstep3}.

This, perhaps unusual, update choice serves the following goal: the desired progress of the algorithm should explicitly depend on both the predicted reduction of the reduced cubic model and the constraint violation.

\begin{lemma}\label{lem:reduction}
    Let $\tau\in(0,1)$ and $u_k\in \mathbb{R}^{n}$ satisfy (OR1)-(OR3). Then, $\Delta q_k\geq \Delta m_k^U(u_k)+\tau\mu_k \beta_k \norm{c_k}_1$ for every $k\geq 0$.
\end{lemma}

\begin{proof}
By definition, we have
\begin{equation*}
    \Delta q_k = \mu_k \norm{c_k}_1- g_k^\top d_k -\frac{1}{2}d_k^\top H_k d_k -\frac{\sigma_k}{3}\norm{d_k}^3-\mu_k\norm{c_k+A_k d_k}_1
\end{equation*}
and
\begin{align*}
    \Delta m_k^U(u_k)&=-\widetilde{g}_k^\top u_k-\frac{1}{2}u_k^\top \widetilde{H}_k u_k-\frac{\sigma_k}{3}\norm{P_k u_k}^3 \\
    &=-(g_k+H_k v_k)^\top P_k u_k - \frac{1}{2}u_k^\top P_k^\top H_k P_k u_k-\frac{\sigma_k}{3}\norm{P_k u_k}^3. 
\end{align*}
Therefore, as $d_k=P_k u_k +v_k$, we get
\begin{equation}\label{eq:oops1}
\Delta q_k=\Delta m_k^U(u_k)-g_k^\top v_k-\frac{1}{2}v_k^\top H_k v_k-\frac{\sigma_k}{3}\left(\norm{d_k}^3-\norm{u_k}^3\right)+\mu_k\left(\norm{c_k}_1-\norm{c_k+A_k d_k}_1\right).
\end{equation}
Hence, it suffices to show that
\begin{equation}\label{eq:oops2}
    \mu_k\left(\norm{c_k}_1-\norm{c_k+A_k d_k}_1-\tau\beta_k\norm{c_k}_1\right)\geq g_k^\top v_k+\frac{1}{2}v_k^\top H_k v_k+\frac{\sigma_k}{3}\left(\norm{d_k}^3-\norm{u_k}^3\right).
\end{equation}
The approximation \eqref{eq:normalstep3} gives
\begin{align}
    \norm{c_k}_1-\norm{c_k+A_k d_k}_1 &=\norm{c_k}_1-\norm{\beta_k c_k+\beta_k A_k v_k^c +(1-\beta_k)c_k}_1 \nonumber \\
    &\geq \norm{c_k}_1-\beta_k \norm{c_k+A_k v_k^c}_1-(1-\beta_k)\norm{c_k}_1 \nonumber \\
    &\geq \norm{c_k}_1-\beta_k r_v \norm{c_k}_1-(1-\beta_k)\norm{c_k}_1 \nonumber \\
    &=\beta_k(1-r_v)\norm{c_k}_1. \label{eq:oops3}
\end{align}
If $c_k=0$, then \eqref{eq:oops1} reduces to $\Delta q_k=\Delta m_k^U(u_k)$ (recall that if $c_k=0$, we set $v_k=v_k^c=0$ and $\beta_k=1$ without loss of generality). Otherwise, the result follows from \eqref{eq:merit2}, \eqref{eq:oops2} and \eqref{eq:oops3}, as $\mu_k\geq \mu_k^c$. 
\end{proof}

\subsection{The Maratos effect and Negative Curvature}\label{sec:sec2.3}

An undesirable phenomenon called the ``Maratos effect" can arise when iterates lie around saddle points. In this case, steps which would lead to superlinear convergence are rejected by the merit function, causing a significant slowdown in algorithmic progress, due to the step increasing the constraint violation sharply. Similarly, in certain circumstances, steps along a direction of negative curvature of the Lagrangian can lead to an increase in both the objective and constraint violation and thus will never be accepted by the merit function, prohibiting global convergence (see \cite[Section 4]{byrd1987trust} for a concrete example). In order to avoid such cases, a common solution is to incorporate a second-order correction step \cite[Chapter 15]{nocedal2006numerical}, which is defined as
 \begin{equation}\label{eq:correction}
        w_k^*:=-A_k^\top(A_k A_k^\top)^{-1}c(x_k+d_k).
    \end{equation}
This additional step captures the curvature information of the constraints, thereby allowing the optimization procedure to move away from saddle points. Still, this particular choice is not a necessity, and one can utilize approximate solutions of the corresponding least square problem, as is the case with the trial step and dual variables considered. We require the correction steps to satisfy $w_k\in \mathcal{R}(A_k^\top)$ and
\begin{equation}\label{eq:w}
    \norm{A_k w_k+c(x_k+d_k)}\leq r_w\norm{d_k}^3,
\end{equation}
for some $r_w\geq 0$. This choice is practically convenient in the case where the trial step is relatively large. Of course, it can be computationally expensive to compute a second-order correction at every iteration. We therefore invoke such action only near ``almost" feasible points. In fact, we take a correction step when the actual reduction $\Delta q_k$ is insufficiently smaller than the predicted one of the cubic model, and $k\in\mathcal{K}$, where
\begin{equation}
    \mathcal{K}:=\left\{k\geq 0:\norm{v_k^c}\leq \frac{\zeta}{\sqrt{\sigma_k}}\right\},
\end{equation}
for $\zeta\in(0,\theta)$, where $\theta$ is defined in \eqref{eq:beta}. This criterion is inspired by that of Byrd et al. \cite{byrd1987trust}, but it is modified for our cubic model, directly utilizing the relationship between the cubic regularization parameter and the trust region radius.

%\nd{Note to Mike: One of the reviewers suggested to  briefly talk about how often corrections are triggered; refer to Byrd's work. I guess we can briefly say how often they accure}

\subsection{The SCP Algorithm}\label{sec:sec2.4}

We are now ready to present the sequential cubic programming algorithm (SCP). Before we come to its description, we define the ratio quantities that will be used to determine if sufficient reduction is achieved at each iteration:
\begin{subequations}
    \begin{align}
        \rho_k&:=\frac{\varphi(x_k,\mu_k)-\varphi(x_k+d_k,\mu_k)}{\Delta q_k}, \\
        \rho_k^{corr}&:=\frac{\varphi(x_k,\mu_k)-\varphi(x_k+d_k+w_k,\mu_k)}{\Delta q_k}.
    \end{align}
\end{subequations}

We say that an iteration $k$ is \textit{very successful} if the step taken produces a large reduction on the model, i.e., if $\rho_k>\eta_2$ or if $\rho_k^{corr}>\eta_2$ when the first ration test fails (i.e., $\rho_k<\eta_1$). Accordingly, we say that $k$ is just \textit{successful} if the computed step provides sufficient reduction for the algorithm to progress towards an optimal solution, i.e., when $\rho_k \in [\eta_1,\eta_2]$ or $\rho_k^{corr}\in [\eta_1,\eta_2]$ (given that $\rho_k<\eta_1
$). Otherwise, iteration $k$ is deemed \textit{unsuccessful}. We thus define the following partitioning sets:
\begin{subequations}
\begin{align*}
&\mathcal{S}:=\left\{k\geq 0:\text{iteration } k\text{\;is (very) successful}\right\};\\ &\mathcal{U}:=\left\{k\geq 0:\text{iteration } k\text{\;is  unsuccessful}\right\}.
\end{align*}
\end{subequations}

\begin{algorithm}[H]
\caption{(SCP)}
\label{alg:scp}
\begin{algorithmic}[]
\State \textbf{Input:} $x_0$, $\sigma_0\geq\sigma_{\min}>0$, $\mu_{-1}>0$, $0<\eta_1<\eta_2<1$, $\nu>1$, $\tau\in(0,1)$, $\theta\in(0,1)$, $\zeta\in(0,\theta)$, $\gamma_2>\gamma_1>1$, $\gamma_3\in(0,1]$, $\delta\in(0,1)$, $r_v\in[0,1-\tau)$, $r_\lambda\geq 0$, $r_w\geq 0$.
\State $k \gets 0$
\While {$x_k$ does not satisfy \eqref{eq:fosp}-\eqref{eq:sosp}}
    \State Compute the normal step $v_k$ so that it satisfies \eqref{eq:normalstep1}-\eqref{eq:normalstep3}.
    \State Compute the dual variables $\lambda_k$ so that they satisfy \eqref{eq:lambdas}.
    \State Compute the tangential step $u_k$ via the SCP Oracle.
    \State Compute $\mu_k^c$ and update $\mu_k$ as in (\ref{eq:merit1})-(\ref{eq:merit2}).
    \State Compute $d_k=u_k+v_k$.
    \If {$\rho_k\geq\eta_1$} (no second-order correction)
    \State Set $x_{k+1}=x_k+d_k$
    \ElsIf {\;$k\in\mathcal{K}$}  (second-order correction)
    \State Compute the correction vector $w_k$ so that it satisfies (\ref{eq:w}).
    \If{$\rho_k^{corr}\geq\eta_1$}
    \State Set $x_{k+1}=x_k+d_k+w_k$
    \EndIf
    \Else
    \State Set $x_{k+1}=x_k$ 
    \EndIf
    \State Update $\sigma_{k+1}\in \begin{cases} 
[\max\{\sigma_{\min},\gamma_3\sigma_k\},\sigma_k] & \text{if } \rho_k>\eta_2\;\text{or}\;``\rho_k<\eta_1\;\text{and}\;\rho_k^{corr}>\eta_2", \\
\{\sigma_k\} & \text{if } \rho_k\in[\eta_1,\eta_2]\;\text{or}\;`` \rho_k<\eta_1\;\text{and}\;\rho_k^{corr}\in[\eta_1,\eta_2]", \\
[\gamma_1\sigma_k,\gamma_2\sigma_k] & \text{otherwise}.
\end{cases}$
    \State $k \gets k + 1$
\EndWhile
\end{algorithmic}
\end{algorithm}

It is clear, by definition, that $\mathcal{S}^{c}=\mathcal{U}$. In addition, we introduce the following notation regarding the set of iterations in which a second-order correction steps is computed:
\begin{equation}\label{eq:cset}
    \mathcal{C}:=\left\{k \in \mathcal{K}: \text{a correction step is computed at iteration\;} k\right\}.
\end{equation}
We note that $\mathcal{C}$ does not have to be a subset of $\mathcal{S}$.

Following the structure of the original ARC algorithm \cite{cartis2011adaptive}, when an iteration is unsuccessful we increase the cubic parameter with the hope that the ``cubic penalization" of the objective function in (\ref{eq:mku}) will lead to an accepted step. On the other hand, if an iteration is very successful, we have the freedom of putting less mass on the cubic term of the subproblem, whereas if an iteration provides sufficient reduction, we keep the same cubic regularization parameter. The latter update differs from the standard ARC method \cite{cartis2011adaptive, cartis2011adaptive2}, where one usually has the choice of slightly increasing the cubic parameter at successful (but not very successful) iterations. Namely, if the ratio lies in $[\eta_1,\eta_2]$, then one sets $\sigma_{k+1}\in[\sigma_k,\gamma_1 \sigma_k]$. From a practical point of view, this update choice offers no significant advantage over simply setting $\sigma_{k+1}=\sigma_k$ (see the discussion at Sections 2 and 7 of \cite{cartis2011adaptive}). We therefore resort to keeping the same cubic parameter when sufficient progress towards a solution is made.

\section{First-order stationary points}\label{sec:sec3}

We start by providing convergence guarantees to first-order stationary points. We also prove worst-case complexity bounds that match the best known bounds for equality constrained problems \cite{berahas2025sequential, curtis2024worst}. Although our principal results concern second-order stationarity, a separate first-order analysis remains both theoretically and computationally relevant. Indeed, several related methods, for instance \cite{goyens2024computing}, first seek a first-order stationary point and subsequently assess second-order stationarity in a separate phase. This separation is computationally motivated, since identifying a first-order stationary point is generally less expensive. In the present framework, first-order guarantees also require substantially weaker algorithmic conditions: First, the matrix $H$ does not need to be equal to the Hessian of the Lagrangian for converence guarantees. In particular, it can be any symmetric bounded matrix. Second, the error bound \eqref{eq:normalstep3} can be relaxed to $\norm{A_k v_k^c+c_k}_1\leq r_v\norm{c_k}_1$, which is much easier to satisfy in practice (Lemma \ref{lem:lsqr}). Further, for the computation of the tangential step only the condition \hyperlink{or1}{(OR1)} of the \textit{Oracle} is required for first-order stationarity, which reduces to simply computing $-\alpha_k^C\widetilde{g}_k$, with $\alpha_k^C$ as in \eqref{eq:cauchy}. Lastly, it is crucial to highlight that we do not need to consider second-order correction steps. This is reasonable, as the main purpose of such correction approach is to avoid saddle points, the latter of which are embedded within the class of plausible first-order critical points. For that reason, when we refer to the set $\mathcal{S}$ of successful iterations of the form $x_{k+1}$, we will implicitly mean that $x_{k+1}=x_k+d_k$.\par

 We make the following assumptions throughout the paper:\hypertarget{f1}{}\hypertarget{f2}{}\hypertarget{f3}{}\hypertarget{f4}{}

%\nd{Note to Mike: We need to fix (F1)} \mo{Trying to figure out how to do it. The merit function depends on $\mu_k$, which makes it hard to assume it has bounded level sets, as the definition of $\mu_{\max}$ is directly dependent on this assumption, so I'm not sure how to state it. Maybe we can handwave some shit? Not sure.}

\textbf{\hyperlink{f1}{(F1)}}: \textit{The functions $f,c$ belong to $C^2$ over $\mathcal{X}$. Moreover, the functions $f,\;g,\;c_i,\;\nabla c_i$, for all $i \in \{1,\dots,m\},$ are bounded over $\mathcal{X}$ by constants $\kappa_f,\kappa_g,\kappa_{c_i},\kappa_{\nabla c_i}$, respectively. Further, they are Lipschitz continuous over $\mathcal{X}$ with Lipschitz constants $L_f,L_g,L_{c_i}$ and $L_{\nabla c_i}$, respectively.}\par
\textbf{\hyperlink{f2}{(F2)}}: \textit{There exists $\gamma_A>0$ such that $\sigma_{\min}(A(x))\geq \gamma_A$ for every $x\in\mathcal{X}$}.  \par
\textbf{\hyperlink{f3}{(F3)}}: \textit{$\norm{H}\leq \kappa_H$ everywhere on $\mathcal{X}$ for some $\kappa_H>0$.} \par
\textbf{\hyperlink{f4}{(F4)}}: \textit{The objective function is bounded below on the set of all successful iterations, namely $f(x)\geq f_{low}$ for every $x\in\mathcal{S}$, for some $f_{low}\in\mathbb{R}$.}\\

Under Assumption \hyperlink{f1}{(F1)}, we define $\kappa_c := \sum_{i=1}^m \kappa_{c_i}$, $\kappa_A := \sum_{i=1}^m \kappa_{\nabla c_i}$, $L_c := \sum_{i=1}^m L_{c_i}$, and $L_A := \sum_{i=1}^m L_{\nabla c_i}$. Assumption \hyperlink{f2}{(F2)} is equivalent to the well-known LICQ, and guarantees that all steps considered are well-defined. Although this is a relatively strong assumption in the literature of constrained optimization, it is a common assumption in the equality constrained literature for methods with worst-case complexity results \cite{curtis2024sequential,goyens2024computing,he2023newton}.

Besides nice local convergence properties (Section \ref{sec:sec5}), the choice of the least-square dual variables is also justified by the following auxiliary result, whose implicit use will become evident in what follows. We refer the reader to Appendix \ref{sec:AppB} for its proof.

\begin{lemma}\label{lem:dualvariables}
    Suppose that \hyperlink{f1}{(F1)}-\hyperlink{f2}{(F2)} hold, and consider the least-squares estimators $\lambda_k^*$ of \eqref{eq:lag2}. Further, let $P_k\in \mathbb{R}^{n\times n}$ be the orthogonal projection matrix and $Z_k\in \mathbb{R}^{(n-m)\times n}$ be the orthonormal basis of $\mathcal{N}(A_k)$, respectively. Then, we have 
    \begin{enumerate}[(i)]
        \item $\norm{\nabla_x\mathcal{L}(x_k,\lambda_k^*)}=\norm{P_k^\top g_k}=\norm{Z_k^\top g_k}$ for every $k\geq 0$.
        \item $\lambda_{\min}(P_k^\top H_k P_k)=\min\{\lambda_{\min}(Z_k^\top H_k Z_k),0\}$.
        \item $P(\cdot)^\top g(\cdot)$ is Lipschitz continuous over $\mathcal{X}$.
    \end{enumerate} 
\end{lemma}

We in fact have $\nabla_x\mathcal{L}(x_k,\lambda_k^*)=P_k^\top g_k$ (see \eqref{eq:b3}). As a consequence, one has the flexibility of working with any of the three quantities that appear in \emph{(i)}, depending on the context. We follow this strategy for the remainder of this paper. Further, the second equality in \emph{(i)} along with \emph{(ii)} indicate that we can effectively replace $P_k$ with $Z_k$ everywhere in the current work, if the latter is easier to compute (e.g., via QR factorization) or if it expedites the process of solving \eqref{eq:mku} at each iteration.

\subsection{First-order Convergence}\label{sec:sec3.1}

Our goal in this section is to prove the following theorem.

\begin{theorem}\label{th:firstorderconvergence}
    Suppose \hyperlink{f1}{(F1)}-\hyperlink{f4}{(F4)} hold. Then, $\lim\limits_{k\to\infty}\norm{\nabla_x\mathcal{L}(x_k,\lambda_k)}=0$ and $\lim\limits_{k\to\infty}\norm{c_k}_1=0$.
\end{theorem}

For the proof of Theorem \ref{th:firstorderconvergence} we will need a series of technical lemmas. We first give a lower bound on the predicted reduction with respect to the tangential step. It represents an improvement of the bound provided in \cite[Lemma 2.1]{cartis2011adaptive}.

\begin{lemma}\label{lem:deltagkbound}
    If $u_k$ satisfies \hyperlink{or1}{(OR1)}, then, for every $k\geq 0$, we have
    \begin{equation}\label{eq:deltagkbound}
        \Delta m_k^U(u_k)\geq\frac{3}{10}\norm{\widetilde{g}_k}\min\left\{\frac{\norm{\widetilde{g}_k}}{1+\norm{H_k}},\sqrt{\frac{\norm{\widetilde{g}_k}}{\sigma_k}} 
  \right\}.
    \end{equation}
\end{lemma}

\begin{proof}
    For every $\alpha\geq 0$, by the Cauchy-Schwarz inequality and orthogonality of $P_k$ we have that
    \begin{align}\label{eq:lemma1eq1}
        -\Delta m_k^U(u_k)&\leq -\Delta m_k^U(u_k^c)\nonumber\\
        &\leq -\Delta m_k^U(-\alpha \widetilde{g}_k)\nonumber\\
        &\leq -\alpha\norm{\widetilde{g}_k}^2+\frac{1}{2}\alpha^2\norm{\widetilde{g_k}}^2\norm{\widetilde{H}_k}+\alpha^3\frac{\sigma_k}{3}\norm{\widetilde{g}_k}^3\nonumber\\
        &\leq \alpha\norm{\widetilde{g}_k}^2\left( -1+\frac{1}{2}\alpha\norm{H_k}+\alpha^2\frac{\sigma_k}{3}\norm{\widetilde{g}_k} \right).
    \end{align}

We may assume without loss of generality that $\norm{\widetilde{g}_k}\neq 0$ (otherwise (\ref{eq:deltagkbound}) holds trivially). For $y\in(0,1]$ we define
\begin{equation}
    \alpha_k(y):=y \min\left\{ \frac{1}{1+\norm{H_k}},\frac{1}{\sqrt{\sigma_k \norm{\widetilde{g}_k}}}  \right\}.
\end{equation}
By plugging $\alpha_k(y)$ into (\ref{eq:lemma1eq1}) we deduce that 
\begin{align}
    \Delta m_k^U(u_k)&\geq \alpha_k(y) \norm{\widetilde{g}_k} \left( 1-\frac{y\norm{H_k}}{2(1+\norm{H_k})} - \frac{y}{\sigma_k \norm{\widetilde{g}_k}} \frac{\sigma_k}{3} \norm{\widetilde{g}_k}  \right) \nonumber \\
    &\geq\alpha_k(y)\norm{\widetilde{g}_k}^2\left(1-\frac{5y}{6}  \right) \nonumber \\
    &=y\left(1-\frac{5y}{6}\right)\norm{\widetilde{g}_k}\min\left\{\frac{\norm{\widetilde{g}_k}}{1+\norm{H_k}},\sqrt{\frac{\norm{\widetilde{g}_k}}{\sigma_k}}\right\}.\label{eq:review2}
\end{align}
The right hand side of \eqref{eq:review2} is maximized for $y=3/5$, and the result directly follows.
\end{proof}

\begin{remark}\label{rem:remark1}
It should be mentioned that, while we require a minimum reduction attained by the Cauchy point, one may simply compute the step $\widetilde{u}_k:=-\widetilde{\alpha}_k\widetilde{g}_k$ at each iteration, where $\widetilde{\alpha}_k:=\alpha_k(3/5)$ as above. This computational choice for the tangential step not only guarantees a positive predicted reduction of $m_k^U$, but also a positive reduction of the entire model: If $\norm{c_k}>0$, then $\Delta q_k>0$ by Lemma \ref{lem:reduction}, whereas if $\norm{c_k}=0$ and $\norm{\widetilde{g}_k}>0$, then $\Delta q_k>0$ by Lemma \ref{lem:deltagkbound}. Note that if $\norm{c_k}=\norm{\widetilde{g}_k}=0$, then $\norm{\nabla_x\mathcal{L}(x_k,\lambda_k)}=0$ by Lemma \ref{lem:dualvariables}, thus we are at a first-order stationary point.
\end{remark}

Next, we provide some useful upper bounds on the normal and tangential steps.

\begin{lemma}\label{lem:tangentialstepbound}
    If \hyperlink{f3}{(F3)} holds, then, for every $k\geq 0$,
    \begin{equation*}
        \norm{u_k}\leq 3\max\left\{ \frac{\kappa_H}{\sigma_k},\sqrt{\frac{\norm{\widetilde{g}_k}}{\sigma_k}}  \right\}.
    \end{equation*}
\end{lemma}

\begin{proof}
    The bound follows from Lemma \ref{lem:deltagkbound} and minor modifications of \cite[Lemma 2.2]{cartis2011adaptive}.
\end{proof}

\begin{lemma}\label{lem:normalstepbound}
    If  \hyperlink{f1}{(F1)}-\hyperlink{f2}{(F2)} hold, then for every $k\geq 0$ there exist positive constants $\kappa_v,\;\kappa_{vs}>0$ such that $\norm{v_k^c}\leq \kappa_v\norm{c_k}_1$, $\norm{v_k^c}^2\leq \kappa_{vs}\norm{c_k}_1$. The same upper bounds hold for $\norm{v_k}$ and $\norm{v_k}^2$, respectively.
\end{lemma}

\begin{proof}
  Due to assumptions \hyperlink{f1}{(F1)} and \hyperlink{f2}{(F2)}, for the first inequality we have
\begin{equation*}
    \norm{v_k^c}\leq \frac{\norm{A_k v_k^c}}{\sigma_{\min}(A_k)}\leq \frac{\norm{A_k v_k^c+c_k}}{\sigma_{\min}(A_k)}+\frac{\norm{c_k}}{\sigma_{\min}(A_k)}\leq \frac{r_v+1}{\sigma_{\min}(A_k)}\norm{c_k}_1\leq \frac{r_v+1}{\gamma_A}\norm{c_k}_1.
\end{equation*}
For the second inequality, we get
\begin{equation*}
    \norm{v_k^c}^2\leq \left(\frac{r_v+1}{\gamma_A}\right)^2\norm{c_k}^2\leq \left(\frac{r_v+1}{\gamma_A}\right)^2\kappa_c\norm{c_k}_1.
\end{equation*}
It is easy to see that the same bounds hold for $\norm{v_k}$, $\norm{v_k}^2$, since $\beta_k\in(0,1]$.
\end{proof}

Using the previous Lemmas we obtain the following bound on the trial step:

\begin{lemma}\label{lem:fullstepbound}
    If \hyperlink{f1}{(F1)}-\hyperlink{f2}{(F2)} hold, then $\displaystyle\norm{d_k}\leq\frac{\kappa_d}{\sqrt{\sigma_k}}$ for some $\kappa_d>0$.  
\end{lemma}

\begin{proof}
    From \eqref{eq:beta} we find $\norm{v_k}\leq\frac{1}{\sqrt{\sigma_k}}$. From the triangle inequality
    and Lemma \ref{lem:tangentialstepbound} we get
    \begin{equation*}
        \norm{d_k}\leq \norm{v_k}+\norm{u_k}\leq \left(1+3\max\left\{ \frac{\kappa_H}{\sqrt{\sigma_{\min}}},\sqrt{\kappa_{\widetilde{g}}} \right\}\right)\frac{1}{\sqrt{\sigma_k}},
    \end{equation*}
where $\kappa_{\widetilde{g}}=\kappa_g+\kappa_H\kappa_v\kappa_c$ by definition of $\widetilde{g}_k$, assumptions \hyperlink{f1}{(F1)}-\hyperlink{f2}{(F2)}, and Lemma \ref{lem:normalstepbound}.
\end{proof}

By the known relation between trust region methods and adaptive cubic methods that was previously mentioned, Lemma \ref{lem:fullstepbound} essentially implies that the trial step is always well-controlled and stays within a feasible region induced by the regularization parameter $\sigma_k$.

We now show that the sequence of merit penalty parameters cannot ``blow up".

\begin{lemma}\label{lem:meritparameterbound}
    If \hyperlink{f1}{(F1)}-\hyperlink{f3}{(F3)} hold, then there exists $\mu_{\max}>0$ such that $\mu_k\leq\mu_{\max}$ for every $k\geq 0$.
\end{lemma}

\begin{proof}
    Recall by the update of the penalty parameter that $\mu_k>\mu_{k-1}$ if and only if $\norm{c_k}_1\neq 0$ and $\mu_{k-1}<\mu_k^c$. Observe that, from the Cauchy-Schwartz and triangle inequalities we get
\begin{align*}
    \mu_k^c &= \frac{g_k^\top v_k+\frac{1}{2}v_k^\top H_k v_k+\frac{\sigma_k}{3}\left[ \norm{d_k}^3-\norm{u_k}^3 \right]}{(1-\tau-r_v)\beta_k \norm{c_k}_1} \\
    & \leq
    \underbrace{\frac{\norm{g_k}\norm{v_k}}{(1-\tau-r_v)\beta_k \norm{c_k}_1}}_{\mathrm{I}}
    +
    \underbrace{\frac{\norm{H_k}\norm{v_k}^2}{2(1-\tau-r_v)\beta_k \norm{c_k}_1}}_{\mathrm{II}}
    +
    \underbrace{\frac{\sigma_k\norm{u_k}^2\norm{v_k}}{3(1-\tau-r_v)\beta_k \norm{c_k}_1}}_{\mathrm{III}} \\
    &+
    \underbrace{\frac{\sigma_k\norm{u_k}\norm{v_k}^2}{3(1-\tau-r_v)\beta_k \norm{c_k}_1}}_{\mathrm{IV}}
    +
    \underbrace{\frac{\sigma_k\norm{v_k}^3}{3(1-\tau-r_v)\beta_k \norm{c_k}_1}}_{\mathrm{V}}.
\end{align*}

We now bound each term individually, using Lemmas \ref{lem:tangentialstepbound} and \ref{lem:normalstepbound}, and recalling that $\|v_k\| \leq 1/\sqrt{\sigma_k}$ and $\norm{v_k}=\beta_k \norm{v_k^c}$:
\begin{itemize}
    \item[] $\mathrm{I}=\displaystyle\frac{\norm{g_k}\norm{v_k}}{(1-\tau-r_v)\beta_k \norm{c_k}_1}\leq \frac{\kappa_g \kappa_v}{1-\tau-r_v}$.
    \item[] $\mathrm{II}=\displaystyle\frac{\norm{H_k}\norm{v_k}^2}{2(1-\tau-r_v)\beta_k \norm{c_k}_1}\leq \frac{\kappa_H\kappa_{vs}}{2(1-\tau-r_v)}$.
    \item[] $\mathrm{III}=\displaystyle\frac{\sigma_k\norm{u_k}^2\norm{v_k}}{3(1-\tau-r_v)\beta_k \norm{c_k}_1}\leq \frac{3\sigma_k\kappa_v}{(1-\tau-r_v)}\max\left\{ \frac{\kappa_H^2}{\sigma_k^2},\frac{\kappa_{\widetilde{g}}}{\sigma_k} \right\}\leq\frac{3\kappa_v}{(1-\tau-r_v)}\max\left\{ \frac{\kappa_H^2}{\sigma_{\min}},\kappa_{\widetilde{g}}
 \right\}$.
 \item[] $\mathrm{IV}=\displaystyle\frac{\sigma_k\norm{u_k}\norm{v_k}^2}{3(1-\tau-r_v)\beta_k \norm{c_k}_1}\leq \frac{\kappa_v\sqrt{\sigma_k}\norm{u_k}}{3(1-\tau-r_v)}\leq \frac{\kappa_v}{(1-\tau-r_v)}\max\left\{ \frac{\kappa_H}{\sqrt{\sigma_{\min}}},\sqrt{\kappa_{\widetilde{g}}} \right\}$.
 \item[] $\mathrm{V}=\displaystyle\frac{\sigma_k\norm{v_k}^3}{3(1-\tau-r_v)\beta_k \norm{c_k}_1}\leq \frac{\sigma_k \kappa_v \norm{v_k}^2}{3(1-\tau-r_v)}\leq \frac{\kappa_{v}}{3(1-\tau-r_v)}$.
\end{itemize}
Set
\begin{equation}
    \mu_{\max}:=\displaystyle\frac{3\kappa_g\kappa_v+\frac{3}{2}\kappa_H\kappa_{vs}+9\kappa_v\max\left\{ \frac{\kappa_H^2}{\sigma_{\min}},\kappa_{\widetilde{g}}
 \right\}+3\kappa_v\max\left\{ \frac{\kappa_H}{\sqrt{\sigma_{\min}}},\sqrt{\kappa_{\widetilde{g}}} \right\}+\kappa_{v}}{3(1-\tau-r_v)}.
\end{equation}
If $\mu_k>\mu_{k-1}$, then we must have $\mu_{k-1}<\mu_{\max}$. Hence,
if this inequality is not satisfied at some iteration $k_0$, then it remains unsatisfied for every $k\geq k_0$. The result follows by the construction of the SCP algorithm, since the merit parameter is increased by a factor of $\nu>1$.
\end{proof}

Next, a bound on the difference of the predicted and actual reduction of the model, is given, when second-order correction steps are not considered.

\begin{lemma}\label{lem:ratiobound}
    Suppose that \hyperlink{f1}{(F1)}-\hyperlink{f3}{(F3)} hold. Then $\Delta q_k-\left(\varphi(x_k,\mu_k)-\varphi(x_k+d_k,\mu_k)\right)\leq \kappa_0\norm{d_k}^2$ for every $k\geq 0$, where $\kappa_0:=\frac{1}{2}\left(L_g+\kappa_H+\mu_{\max}L_A\right)$.
\end{lemma}

\begin{proof}
First, note that $\Delta q_k=q_k(0)-q_k(d_k)=\varphi(x_k,\mu_k)-q_k(d_k)$. Therefore, the difference of the predicted and actual reduction is equal to $\varphi(x_k+d_k,\mu_k)-q_k(d_k)$. By expanding this difference and using the triangle inequality we find
\begin{align}\label{eq:aredpred}
    \varphi(x_k+d_k,\mu_k)-q_k(d_k)&\leq \left| f_k-f(x_k+d_k)+g_k^\top d_k \right| + \left| \frac{1}{2}d_k^\top H_k d_k \right| + \nonumber \\
    &+\mu_k(\norm{c(x_k+d_k)}_1-\norm{c_k+A_k d_k}_1).
\end{align}

By Lipschitz continuity of $g$, the first absolute value of (\ref{eq:aredpred}) is bounded by $\frac{L_g}{2}\norm{d_k}^2$. The second absolute value of the same inequality is bounded by $\frac{\kappa_H}{2}\norm{d_k}^2$ due to \hyperlink{f3}{(F3)}. Finally, by Lemmas \ref{lem:lipcon} and \ref{lem:meritparameterbound}, it follows that the last term of \eqref{eq:aredpred} is bounded by $\frac{\mu_{\max} L_A}{2} \norm{d_k}^2$. 
\end{proof}

The following technical lemmas will be crucial for the proof of our global convergence results.

\begin{lemma}\label{lem:technicallemma}
    Suppose that \hyperlink{f1}{(F1)}-\hyperlink{f3}{(F3)} hold, and let $\mathcal{I}$ be some infinite index sets of iterates. If $\sigma_k\to\infty$ as $k\to\infty$, $k\in \mathcal{I}$, and $\Delta q_k\geq\min\left\{ \delta_1,\frac{\delta_2}{\sqrt{\sigma_k}} \right\}$ for every $k\in \mathcal{I}$, for some $\delta_1,\delta_2>0$ independent of $k$, then $\rho_k>\eta_2$ for every $k\in\mathcal{I}$ sufficient large. Moreover, not all $k$ sufficiently large belong to the index set $\mathcal{I}$.
\end{lemma}
\begin{proof}

 From Lemmas \ref{lem:fullstepbound} and \ref{lem:ratiobound} we acquire
\begin{equation}\label{eq:important}
    \Delta q_k-\varphi(x_k,\mu_k)+\varphi(x_k+d_k,\mu_k)\leq \kappa_\varphi \frac{1}{\sigma_k}\quad\forall k\in\mathcal{I},
\end{equation}
for $\kappa_\varphi:=\kappa_0\kappa_d^2$. By assumption, it follows that 
\begin{equation}
   1-\rho_k= 1-\frac{\varphi(x_k,\mu_k)-\varphi(x_k+d_k,\mu_k)}{\Delta q_k}\leq \kappa_\varphi\max\left\{ \frac{1}{\delta_1 \sigma_k},\frac{1}{\delta_2\sqrt{\sigma_k}} \right\}\;\;\forall k\in\mathcal{I}.
\end{equation}
Since $\sigma_k\to\infty$ as $k\to\infty$, $k\in\mathcal{I}$, we obtain  $\max\left\{ \frac{1}{\delta_1 \sigma_k},\frac{1}{\delta_2\sqrt{\sigma_k}} \right\}< 1-\eta_2$ for every $k\in\mathcal{I}$ sufficiently large, i.e., all iterates in $\mathcal{I}$ are eventually very successful. The last fact follows directly from the construction of the algorithm, as $\sigma_{k+1}\leq \sigma_k$ under a very successful iteration $k$.
\end{proof}

\begin{lemma}\label{lem:technicallemma1}
    Suppose that \hyperlink{f1}{(F1)}-\hyperlink{f4}{(F4)} hold. Consider an infinite set of successful iterates $\mathcal{I}\subseteq\mathcal{S}$ along with a sequence $\{\alpha_k\}_{k\in\mathcal{I}}$. If $\Delta q_k\geq\alpha_k$ for every $k\in \mathcal{I}$, then $\alpha_k\to 0$ as $k\to\infty$, $k\in\mathcal{I}$. 
\end{lemma}

\begin{proof}
 Since $\mathcal{I}$ is infinite we can write without loss of generality $\mathcal{I}=\{k_i:i\geq 0\}$.  By the construction of the algorithm we have
    \begin{equation}\label{eq:sum1}
        \varphi(x_{k_i},\mu_{k_i})-\varphi(x_{k_{i+1}},\mu_{k_i})\geq \eta_1 \Delta q_{k_i}\geq \eta_1 \alpha_{k_i}\quad\forall i\geq 0.
    \end{equation}
Fix an index $j\geq 0$ such that $k_j\in\mathcal{I}$. We have
\begin{align}\label{eq:telescopic}
    \sum_{i=0}^{j}\left( \varphi(x_{k_i},\mu_{k_i})-\varphi(x_{k_{i+1}},\mu_{k_i}) \right)&=\sum_{l=k_0}^{k_j}\left( \varphi(x_l,\mu_l)-\varphi(x_{l+1},\mu_l) \right)-\sum_{\substack{k_0\leq l\leq k_j \\ l\not\in I}}\left( \varphi(x_l,\mu_l)-\varphi(x_{l+1},\mu_l) \right) \nonumber \\
    &\leq \sum_{l=k_0}^{k_j}\left( \varphi(x_l,\mu_l)-\varphi(x_{l+1},\mu_l) \right) \nonumber \\
    &=\sum_{l=k_0}^{k_j}\left( \varphi(x_l,\mu_l)-\varphi(x_{l+1},\mu_{l+1}) \right) + \sum_{l=k_0}^{k_j}\norm{c_{l+1}}_1\left(\mu_{l+1}-\mu_{l}\right) \nonumber \\
    &\leq \varphi(x_{k_0},\mu_{k_0})-\varphi(x_{k_{j+1}},\mu_{k_{j+1}})+\kappa_c(\mu_{\max}-\mu_{k_0}) \nonumber \\
    &\leq -\varphi(x_{k_{j+1}},\mu_{k_{j+1}})+ f(x_{k_0})+\kappa_c\mu_{\max},
\end{align}
where we used the relation $\varphi(x_{l+1},\mu_l)\leq \varphi(x_l,\mu_l)$ for every $l\geq 0$ by construction of the algorithm, Lemma \ref{lem:meritparameterbound} and the monotonicity of $\{\mu_k\}_{k\geq 0}$. Hence,
\begin{equation}\label{eq:sum2}
    \frac{1}{\mu_{k_{j+1}}}\sum_{i=0}^{j}\left( \varphi(x_{k_i},\mu_{k_i})-\varphi(x_{k_{i+1}},\mu_{k_i}) \right)\leq -\Phi\left(x_{k_{j+1}},\mu_{k_{j+1}}\right)+\frac{f(x_{k_0})-f_{low}+\kappa_c\mu_{\max}}{\mu_{-1}},
\end{equation}
where $\Phi(x,\mu):=\frac{\varphi(x,\mu)-f_{low}}{\mu}$. Summing \eqref{eq:sum1} up to $j$ and combining with \eqref{eq:sum2} and Lemma \ref{lem:meritparameterbound}, we deduce the bound
\begin{equation}\label{eq:sum3}
-\Phi\left(x_{k_{j+1}},\mu_{k_{j+1}}\right)+\frac{f(x_{k_0})-f_{low}+\kappa_c\mu_{\max}}{\mu_{-1}}\geq  \frac{\eta_1}{\mu_{\max}} \sum_{i=0}^{j} \alpha_{k_i}.
\end{equation}
The sequence $\{\Phi(x_{k_j},\mu_{k_j})\}_{j\geq 0}$ is non-negative and non-increasing by \cite[Lemma 5]{pei2023sequential}, thus it converges to some finite value $\Phi^*$ as $j\to\infty$. The relation \eqref{eq:sum3} in the limit $j\to\infty$ implies that the sequence $\{\alpha_{k_i}\}_{i\geq 0}$ is summable, which finally yields $\alpha_{k_i}\to 0$ as $i\to\infty$.
\end{proof}

In order to prove global convergence, a common strategy is to show that all steps are eventually (very) successful. That way, the algorithm will keep making progress towards a solution. To this end, we hope of using an argument by contradiction. Lemma \ref{lem:technicallemma} provides a sufficient criterion when the set of all successful iterations is infinite. In turn, Lemma \ref{lem:technicallemma1} gives an even weaker condition when the objective function is lower bounded:

\begin{corollary}\label{cor:technicalcorollary}
    \textit{Suppose \hyperlink{f1}{(F1)}-\hyperlink{f4}{(F4)} hold and that the set of all successful iterations $\mathcal{S}$ is infinite. If $\Delta q_k\geq\min\left\{ \delta_1,\frac{\delta_2}{\sqrt{\sigma_k}} \right\}$ for every $k\in \mathcal{S}$, for some $\delta_1,\delta_2>0$ independent of $k$, then $\sigma_k\to\infty$ as $k\to\infty$, $k\in\mathcal{S}$. Moreover, not all sufficiently large iterations are successful.} 
\end{corollary}

\begin{proof}
    The first part of the corollary follows from the previous lemma for $\mathcal{I}:=\mathcal{S}$ and $\alpha_k:=\min\left\{ \delta_1,\frac{\delta_2}{\sqrt{\sigma_k}} \right\}$. The second part follows from the construction of the algorithm, since the cubic regularization parameter increases under unsuccessful iterations.
\end{proof}

The next result handles global first-order convergence when Algorithm \ref{alg:scp} produces only a finite number of successful iterates.

\begin{lemma}\label{lem:finitecase}
    Suppose that \hyperlink{f1}{(F1)}-\hyperlink{f3}{(F3)} hold and that the set $\mathcal{S}$ of (very) successful iterates is finite. Then, there exists $k^*\geq0$ such that $\nabla_x \mathcal{L}(x_k,\lambda_k)=0$ and $c_k=0$ for every $k\geq k^*+1$. 
\end{lemma}

\begin{proof}
    Let $k^*$ be the last (very) successful iterate and set $x^*:=x_{k^*+1}=x_{k^*+i}$ for every $i\geq 2$. Suppose that $\norm{\widetilde{g}_k}>0$ for every $k\geq k^*+1$, and let $\xi:=\norm{\widetilde{g}_k}$ for all $k\geq k^*+1$. Then, we have
    \begin{equation}
        \Delta q_k\geq \Delta m_k^U(u_k)\geq \frac{3}{10}\xi^2\min\left\{\frac{1}{1+\kappa_H},\frac{1}{\sqrt{\kappa_{\widetilde{g}}}\sqrt{\sigma_k}}  \right\}\quad\forall k\geq k^*+1,
    \end{equation}
    by Lemma \ref{lem:deltagkbound}. Define $\mathcal{I}:=\{k\geq0:k\geq k^*+1\}$. By the construction of the algorithm ($\gamma_1>1$), it is deduced that $\sigma_k\to\infty$ as $k\to \infty$, $k\in\mathcal{I}$. Lemma \ref{lem:technicallemma} implies that all iterations in $\mathcal{I}$ sufficiently large are successful, which is a contradiction. Therefore, $\widetilde{g}_k=0$ for every $k\geq k^*+1$.\par
    Similarly, suppose that $\psi:=\norm{c_k}_1>0$ for every $k\geq k^*+1$. Then
    \begin{equation}
        \Delta q_k\geq \tau\mu_k\beta_k\norm{c_k}_1\geq\tau\mu_{-1}\min\left\{ \psi,\frac{\theta}{\kappa_v\sqrt{\sigma_k}} \right\}\quad\forall k\geq k^*+1.
    \end{equation}
    Once again, we reach a contradiction by invoking Lemma \ref{lem:technicallemma} for the same set of iterates $\mathcal{I}$. This implies $v_k=0$ for every $k\geq k^*+1$, which then, along with Lemma \ref{lem:dualvariables} and \eqref{eq:lambdas}, yields $\nabla_x\mathcal{L}(x_k,\lambda_k)=0$ for all $k\geq k^*+1$.
\end{proof}

Our first task is feasibility in the limit:

\begin{theorem}\label{th:constraintsconvergence}
    Suppose \hyperlink{f1}{(F1)}-\hyperlink{f4}{(F4)} hold. Then, $\lim\limits_{k\to\infty}\norm{c_k}_1= 0$.
\end{theorem}

\begin{proof}
    We first prove that $\liminf\limits_{k\to\infty}\norm{c_k}_1=0$. We will then use uniform continuity of $c$ over $\mathcal{X}$ (which is due to Lipschitz continuity) to show the stronger convergence result. Suppose that this is not the case. Then, there exists some $\varepsilon>0$ such that $\norm{c_k}_1\geq \varepsilon$ $\forall k\geq 0$. Therefore, for every $k\geq0$, from Lemma \ref{lem:normalstepbound} we have
    \begin{equation}\label{eq:good2}
        \Delta q_k\geq \tau\mu_k\beta_k\norm{c_k}_1\geq \tau\mu_{-1}\min\left\{ \varepsilon,\frac{\theta}{\kappa_v\sqrt{\sigma_k}} \right\}.
    \end{equation}
 As $\mathcal{S}$ is infinite (otherwise by Lemma \ref{lem:finitecase} we have nothing to prove) we can apply Corollary \ref{cor:technicalcorollary} to deduce that not all $k\geq0$ sufficiently large belong to $\mathcal{S}$. We then follow an argument similar to \cite[Theorem 2.5]{cartis2011adaptive}: Because $\mathcal{S}$ is infinite, there exists an infinite subsequence $\{k_j\}$ of (very) successful iterates in $\mathcal{S}$ such that $k_j-1$ is unsuccessful for every $j\geq0$. From the construction of the algorithm ($\sigma_{k_j}\leq\gamma_2\sigma_{k_j-1}$) we get $\sigma_{k_j-1}\to\infty$ as $j\to\infty$. Consider the set $\mathcal{J}:=\{k_j-1:j\geq0\}$. By \eqref{eq:good2} and Lemma \ref{lem:technicallemma} for $\mathcal{I}:=\mathcal{J}$ we deduce that the iterate $k_j-1$ is very successful for $j$ sufficiently large. This is a contradiction.\par
Hence, $\liminf\limits_{k\to\infty}\norm{c_k}_1=0$. Suppose now that $\lim\limits_{k\to\infty}\norm{c_k}_1\neq0$. Since we still assume that $\mathcal{S}$ is infinite (without loss of generality due to Lemma \ref{lem:finitecase}), there exists $\varepsilon'>0$ and an infinite subsequence $\{m_j\}\subseteq\mathcal{S}$ such that 
\begin{equation}
    \norm{c_{m_j}}_1>2\varepsilon'\;\;\forall j\geq0.
\end{equation}
Due to  $\liminf\limits_{k\to\infty}\norm{c_k}_1=0$, for that $\varepsilon'$ there exists, for every $j$, a ``first almost feasible" iterate, i.e., a successful iterate $n_j>m_j$ that satisfies $\norm{c_{n_j}}_1<\varepsilon'$, and $\norm{c_k}_1\geq\varepsilon'$ for every $m_j\leq k<n_j$. We then define the set $\mathcal{T}:=\cup_{j\geq0}\mathcal{T}_j$, where $\mathcal{T}_j:=\{k\in\mathcal{S}:m_j\leq k<n_j\}$. The relation \eqref{eq:good2} holds everywhere on $\mathcal{T}$, where we now have $\varepsilon'$ instead of $\varepsilon$ inside the minimum function. Since $\mathcal{T}$ is an infinite subset of $\mathcal{S}$, the first conclusion of Corollary \ref{cor:technicalcorollary} still applies. That is, 
\begin{equation}
    \varphi\left(x_{k},\mu_{k}\right)-\varphi\left(x_{k+1},\mu_{k}\right)\geq\eta_1\tau\mu_{-1}\frac{\theta}{\kappa_v\sqrt{\sigma_{k}}}\quad\forall k \in \mathcal{T} :\; k \text{\;is sufficiently large}.
\end{equation}

So, Lemma \ref{lem:fullstepbound} yields
\begin{equation}\label{eq:cool}
    \varphi(x_k,\mu_k)-\varphi(x_{k+1},\mu_k)\geq\frac{\eta_1\rho\mu_{-1}\theta}{\kappa_v \kappa_d}\norm{d_k} \;\;\forall k\in\mathcal{T}:\;\text{$k$ is sufficiently large}.
\end{equation}
In addition, for all $j$ sufficiently large, we have
\begin{equation}\label{eq:verycool}
     \sum_{k\in\mathcal{T}_j}\left( \varphi(x_{k},\mu_{k})-\varphi(x_{k+1},\mu_{k}) \right)\leq\varphi(x_{m_j},\mu_{m_j})-\varphi(x_{l_j+1},\mu_{l_j+1})+\sum_{k=m_j}^{l_j}\norm{c_{k+1}}(\mu_{k+1}-\mu_{k}),
\end{equation}
verbatim \eqref{eq:telescopic}, where $l_j\in\mathcal{S}$ represents the largest element of $\mathcal{T}_j$. By Lemma \ref{lem:meritparameterbound} we have $\mu_{j+i}=\mu_{j}$ for all $i\geq 0$ for every $j$ sufficiently large. Hence, we can argue that, for all $j$ large enough, \eqref{eq:verycool} is equivalent to
\begin{equation}\label{eq:reallycool}
     \frac{1}{\mu_{m_j}}\sum_{k\in\mathcal{T}_j}\left( \varphi(x_{k},\mu_{k})-\varphi(x_{k+1},\mu_{k}) \right)\leq \Phi(x_{m_j},\mu_{m_j})-\Phi(x_{l_j+1},\mu_{l_j+1}).
\end{equation}
Recall from Lemma \ref{lem:technicallemma1} that $\Phi$ is non-increasing in the space of successful iterates. Hence, since $l_j+1\leq n_j\in\mathcal{S}$, from \eqref{eq:cool}-\eqref{eq:reallycool}, Lemma \ref{lem:meritparameterbound}, and the fact $x_{k+1}=x_k$ for all $k\in\mathcal{U}$, we get
\begin{align}
    \Phi(x_{m_j},\mu_{m_j})-\Phi(x_{n_j},\mu_{n_j})&\geq \frac{\eta_1\rho\mu_{-1}\theta}{\kappa_v \kappa_d \mu_{\max}} \sum_{k\in\mathcal{T}_j} \norm{d_k} \nonumber \\
    &= \frac{\eta_1\rho\mu_{-1}\theta}{\kappa_v \kappa_d \mu_{\max}} \sum_{k=m_j}^{n_j-1}\norm{x_{k+1}-x_k} \nonumber \\
    &\geq \frac{\eta_1\rho\mu_{-1}\theta}{\kappa_v \kappa_d \mu_{\max}} \norm{x_{n_j}-x_{m_j}}.
\end{align}

Because \hyperlink{f4}{(F4)} implies that $\{\Phi(x_i,\mu_i)\}_{i\geq 0}$ is convergent over $\mathcal{S}$ (Lemma \ref{lem:technicallemma1}), we have that $\Phi(x_{m_j},\mu_{m_j})-\Phi(x_{n_j},\mu_{n_j})\to 0$ as $m_j,n_j\to \infty$ $(j\to \infty)$, thus $\norm{x_{n_j}-x_{m_j}}\to 0$ as $j\to\infty$. In addition, $c$ is uniformly continuous over $\mathcal{X}$, therefore $\norm{c_{n_j}-c_{m_j}}_1\to 0$ as $j\to\infty$. This is however impossible, as $\norm{c_{n_j}-c_{m_j}}_1\geq \norm{c_{m_j}}_1-\norm{c_{n_j}}_1\geq\varepsilon'$ for every $j\geq0$.
\end{proof}

The first part of Theorem \ref{th:firstorderconvergence} follows after feasibility in the limit and a simple bound on the distance between the true and approximate dual variables:

\begin{lemma}\label{lem:dualbounds}
    Suppose that \hyperlink{f1}{(F1)}-\hyperlink{f2}{(F2)} holds. Then, for every $k\geq 0$, we have
    \begin{enumerate}
        \item $\|\lambda_k-\lambda_k^{*}\| \leq \frac{r_{\lambda}}{\gamma_A^2}\norm{v_k}$.
        \item $\norm{\lambda_k^*}\leq \frac{\kappa_A \kappa_g}{\gamma_A^2}$ and  $\norm{\lambda_k}\leq \kappa_{\lambda}$ for some $\kappa_{\lambda}>0$. 
    \end{enumerate}
\end{lemma}

\begin{proof}
    \begin{enumerate}
        \item By \eqref{eq:lambdas}, there exists a vector $\xi_k \in \mathbb{R}^m$ such that
        \begin{equation*}
            A_k g_k + A_k A_k^{\top} \lambda_k + \xi_k = 0 \quad \text{and} \quad \|\xi_k\| \leq r_\lambda \|v_k\|.
        \end{equation*}
        Therefore, by the definition of $\lambda_k^*$,
        \begin{align*}
            \norm{\lambda_k - \lambda_k^*} &= \norm{\lambda_k + (A_k A_k^\top)^{-1} A_k g_k} \\
            &= \norm{\lambda_k - (A_k A_k^\top)^{-1} (A_k A_k^\top \lambda_k + \xi_k)} \\
            &= \norm{(A_k A_k^\top)^{-1} \xi_k} \\
            &\leq \frac{r_\lambda}{\gamma_A^2} \|v_k\|,
        \end{align*}
        where the last inequality follows by \hyperlink{f2}{(F2)}.
        \item The first bound follows directly from \hyperlink{f1}{(F1)} and \hyperlink{f2}{(F2)}. Further, from part (1) we have
    \begin{equation*}
    \norm{\lambda_k}\leq \norm{\lambda_k^*-\lambda_k}+\norm{\lambda_k^*}\leq \frac{1}{\gamma_A^2}\left( r_\lambda \kappa_v\kappa_c+\kappa_A\kappa_g\right)=:\kappa_\lambda.
\end{equation*}
    \end{enumerate}
\end{proof}

\begin{theorem}\label{th:firstorderconvergence1}
    Suppose that \hyperlink{f1}{(F1)}-\hyperlink{f4}{(F4)} hold. Then $\lim\limits_{k\to\infty}\norm{\nabla_x\mathcal{L}(x_k,\lambda_k)}=0$.
\end{theorem}
\begin{proof}
    We first show that $\liminf\limits_{k\to\infty}\norm{\widetilde{g}_k}=0$.  Theorem \ref{th:constraintsconvergence} and Lemma \ref{lem:normalstepbound} will then imply $\liminf\limits_{k\to\infty}\norm{P_k^\top g_k}=0$. Suppose that this is not true. Then, there exists some $\varepsilon>0$ such that $\norm{\widetilde{g}_k}\geq\varepsilon$ for all $k\geq0$. Hence, by Lemma \ref{lem:deltagkbound} we have
    \begin{equation*}
        \Delta q_k\geq \frac{3}{10}\varepsilon^2\min\left\{ \frac{1}{1+\kappa_H},\frac{1}{\sqrt{\kappa_{\widetilde{g}}}\sqrt{\sigma_k}} \right\}\;\;\forall k\geq0.
    \end{equation*}
 Corollary \ref{cor:technicalcorollary} indicates that not all $k$ sufficiently large are in $\mathcal{S}$. By following an argument identical to the proof of Theorem \ref{th:constraintsconvergence}, we deduce that all $k\geq0$ sufficiently large do belong to $\mathcal{S}$, which is a contradiction.\par

Therefore, $\liminf\limits_{k\to\infty}\norm{P_k^\top g_k}=0$. To show the stronger convergence result, we use uniform continuity of the function $P(\cdot)^\top g(\cdot)$ over $\mathcal{X}$, the latter of which follows from Lemma \ref{lem:dualvariables}. Suppose that $\lim\limits_{k\to\infty}\norm{P_k^\top g_k}\neq 0$. Then, similarly to Theorem 3.9, as $\mathcal{S}$ is infinite without loss of generality (due to Lemma \ref{lem:finitecase}), there exists $\varepsilon'>0$ and an infinite subsequence $\{m_j\}\subset\mathcal{S}$ such that 
\begin{equation}
    \norm{P_{m_j}^\top g_{m_j}}>2\varepsilon'\quad \forall j\geq 0.
\end{equation}
In similar fashion to the previous theorem, due to the fact that $\liminf\limits_{k\to\infty}\norm{P_k^\top g_k}=0$, for that $\varepsilon'$ there exists, for every $j$, a successful iterate $n_j>m_j$ that satisfies $\norm{P_{n_j}^\top g_{n_j}}<\varepsilon'$ and $\norm{P_k^\top g_k}\geq\varepsilon'$ for all iterates $m_j\leq k<n_j$. We then define the index sets $\mathcal{R}_j:=\{k\in\mathcal{S}:m_j\leq k<n_j\}$.
In addition, by Theorem \ref{th:constraintsconvergence}, for that $\varepsilon'>0$ there exists an index $j^*$ such that $\norm{c_{k}}\leq \frac{\varepsilon'}{2\kappa_H\kappa_v}$ for all $m_j\leq k<n_j$ for every $j\geq j^*$. Hence, by the triangle inequality and the bounds of \hyperlink{f1}{(F1)}-\hyperlink{f3}{(F3)} we obtain
\begin{equation}
    \norm{\widetilde{g}_k}\geq\varepsilon'/2\quad \forall m_j\leq k<n_j,\;\; \forall j\geq j^*.
\end{equation}
Lemma \ref{lem:deltagkbound} and the construction of the algorithm guarantee the bound
\begin{equation}
    \varphi(x_k,\mu_k)-\varphi(x_{k+1},\mu_k)\geq\frac{3}{40} \eta_1 (\varepsilon')^2 \min\left\{ \frac{1}{1+\kappa_H},\frac{1}{\sqrt{\kappa_{\widetilde{g}}}\sqrt{\sigma_k}} \right\}\quad\forall k\in\mathcal{R}_j,\;\;\forall j\geq j^*.
\end{equation}
By following exactly the same steps as we did in the proof of Theorem \ref{th:constraintsconvergence}, and by invoking Lemma \ref{lem:technicallemma1} (Corollary \ref{cor:technicalcorollary}), we deduce that $\norm{x_{m_j}-x_{n_j}}\to 0$ as $j\to\infty$. Since $P(\cdot)^\top g(\cdot)$ is uniformly continuous on the space of all iterates, it follows that $\norm{P_{m_j}^\top g_{m_j}-P_{n_j}^\top g_{n_j}}\to 0$ as $j\to\infty$. This is a contradiction, however, as $\norm{P_{m_j}^\top g_{m_j}-P_{n_j}^\top g_{n_j}}\geq \norm{P_{m_j}^\top g_{m_j}}-\norm{P_{n_j}^\top g_{n_j}}\geq \varepsilon'$ for all $j\geq 0$.\par
Hence, Lemma \ref{lem:dualvariables} yields $\lim\limits_{k\to\infty}\norm{\nabla_x\mathcal{L}(x_k,\lambda_k^*)}=0$. The error bound \eqref{eq:lambdas} and  Lemmas \ref{lem:normalstepbound} and \ref{lem:dualbounds} give
\begin{equation}\label{eq:gradientdiff}
    \norm{\nabla_x\mathcal{L}(x_k,\lambda_k^*)-\nabla_x\mathcal{L}(x_k,\lambda_k)}\leq\kappa_A\norm{\lambda_k^*-\lambda_k}\leq \frac{\kappa_A r_{\lambda}\kappa_v}{\gamma_A^2}\norm{c_k}_1.
\end{equation}
The result follows from \eqref{eq:gradientdiff} and Theorem \ref{th:constraintsconvergence}.
\end{proof}

\subsection{First-order Complexity}\label{sec:Sec3.2}

Next, we prove complexity bounds of order $\mathcal{O}\left(\max\left\{\epsilon^{-2}_g,\epsilon^{-1}_c\right\}\right)$ to first-order stationary points. This is known to be optimal in the unconstrained case \cite{nesterov2006cubic}, and matches the best bound in the equality constrained setting \cite{berahas2025sequential, curtis2024worst}, when only first-order derivative information is available.\par

We begin with a bound on the maximum number of unsuccessful iterations of the algorithm in relation to that of successful iterations. To this end, for every $k\geq 0$, we define the sets
\begin{subequations}
    \begin{align}
        \mathcal{U}_k&:=\{j\leq k:j\in\mathcal{U}\};\\
        \mathcal{S}_k&:=\{j\leq k:j\in\mathcal{S}\}. 
    \end{align}
\end{subequations}

\begin{lemma}[Theorem 2.1 \cite{cartis2011adaptive2}]\label{lem:cardinalitybound}
    If $\sigma_j\leq \widetilde{\sigma}$ $\forall j\leq k$ for some $\widetilde{\sigma}>0$, then
    \begin{equation*}
       |\mathcal{U}_k|\leq\ceil[\Bigg]{\displaystyle\frac{\log\widetilde{\sigma}-\log\sigma_{0}-|\mathcal{S}_k|\log\gamma_3}{\log\gamma_1}}. 
    \end{equation*}
\end{lemma}

We require two technical lemmas which guarantee that the sequence of cubic regularization parameters is well-controlled over certain iteration subsets of $\mathcal{X}$. These subsets will then be chosen in an appropriate manner so that they include all iterates that land away from first-order stationary points.

\begin{lemma}\label{lem:fosplemma1}
    Suppose that \hyperlink{f1}{(F1)}-\hyperlink{f3}{(F3)} hold. Let $\epsilon\in(0,1)$, $\mathcal{I}_\epsilon$ be some index set and $\tilde{\delta} := \min\left\{1, \sqrt{\frac{1-\eta_2}{10\kappa_v\kappa_0}}\right\}$, where $\kappa_v,\kappa_0$ are as in Lemmas \ref{lem:normalstepbound} and \ref{lem:ratiobound}, respectively. If $\norm{\widetilde{g_k}}>\epsilon$ and $\tilde{\delta} \epsilon \geq \sqrt{\norm{c_k}_1}$ for every $k\in\mathcal{I}_\epsilon$, then
    \begin{equation*}
        \sigma_k\leq \max\left\{ \sigma_0,\gamma_2\frac{\kappa_{cg}^2}{\epsilon} \right\}
    \end{equation*}
    for every $k\in\mathcal{I}_\epsilon$, where
    \begin{equation*}
        \kappa_{cg}:=\max\left\{\frac{45 \kappa_0}{1-\eta_2},1+\kappa_H \right\}.
    \end{equation*}
\end{lemma}

\begin{proof}
    By the construction of the algorithm, it suffices to show that if $\sqrt{\sigma_k\norm{\widetilde{g_k}}}> \kappa_{cg}$, then $k$ is very successful. Indeed, if this is true, then the maximum value that the cubic parameter may take at iteration $k\in \mathcal{I}_{\epsilon}$ (after the possibly unsuccessful iteration $k-1$) is $\gamma_2\kappa_{cg}^2/\epsilon$ due to the update formula of the algorithm and the bound $\norm{\widetilde{g}_k}>\epsilon$. To this end, suppose that $\sqrt{\sigma_k\norm{\widetilde{g_k}}}> \kappa_{cg}$ for every $k\in\mathcal{I}_\epsilon$ and consider the ratio $\rho_k$. We will show that $\rho_k>\eta_2$, or equivalently, that
    \begin{equation}\label{eq:bad1}
        \varphi(x_{k+1},\mu_k)-q_k(d_k)+(1-\eta_2)(q_k(d_k)-\varphi(x_k,\mu_k))<0
    \end{equation}
for every $k\in\mathcal{I}_\epsilon$.  First, we have $q_k(d_k)-\varphi(x_k,\mu_k)=-\Delta q_k\leq -\Delta m_k^U$. Now, the bound $\kappa_{cg}\geq 1+\kappa_H$ and Lemma \ref{lem:deltagkbound} imply that $\Delta m_k^U\geq\frac{3\norm{\widetilde{g}_k}}{10}\sqrt{\frac{\norm{\widetilde{g}_k}}{\sigma_k}}$ for every $k\in\mathcal{I}_\epsilon$, which yields
\begin{equation}\label{eq:bad2}
    q_k(d_k)-\varphi(x_k,\mu_k)\leq -\frac{3\norm{\widetilde{g_k}}}{10}\sqrt{\frac{\norm{\widetilde{g_k}}}{\sigma_k}}.
\end{equation}  
From Lemma \ref{lem:ratiobound} we have that
\begin{equation}\label{eq:good1}
    \varphi(x_{k+1},\mu_k)-q_k(d_k)\leq \kappa_0\norm{d_k}^2.
\end{equation}
Moreover, the bound $\kappa_{cg}> \kappa_H$ and Lemma \ref{lem:tangentialstepbound} give $\norm{u_k}^2\leq 9 \frac{\norm{\widetilde{g}_k}}{\sigma_k}$. In addition, we have
\begin{equation*}
    \norm{v_k}^2\leq \frac{\beta_k\norm{v_k^c}}{\sqrt{\sigma_k}}\leq \frac{\kappa_v \|c_k\|_1}{\sqrt{\sigma_k}}\leq \frac{\kappa_v \tilde{\delta}^2 \epsilon^2}{\sqrt{\sigma_k}} < \kappa_v \tilde{\delta}^2 \norm{\widetilde{g}_k} \sqrt{\frac{\norm{\widetilde{g}_k}}{\sigma_k}} \leq \frac{(1-\eta_2)\norm{\widetilde{g}_k}}{10\kappa_0} \sqrt{\frac{\norm{\widetilde{g}_k}}{\sigma_k}},
\end{equation*}
where we used \eqref{eq:normalstep1} and \eqref{eq:beta} in the first inequality, $\epsilon < 1$ in the second to last inequality and $\tilde{\delta}^2 \leq (1-\eta_2)/(10 \kappa_v\kappa_0)$, in the final inequality. Therefore, we deduce that the left-hand side of \eqref{eq:bad1} is upper bounded by
\begin{equation}
    \frac{\norm{\widetilde{g_k}}}{\sigma_k}\left[9\kappa_0-\frac{(1-\eta_2)}{5}\sqrt{\sigma_k\norm{\widetilde{g_k}}} \right],
\end{equation}
and thus, when
\begin{equation*}
    \sqrt{\sigma_k \|\widetilde{g}_k\|} \geq \frac{45 \kappa_0}{1-\eta_2},
\end{equation*}
\eqref{eq:bad1} holds. Consequently, if $0\not\in\mathcal{I}_\epsilon$, the desired bound follows. If $0\in\mathcal{I}_\epsilon$ and $\sigma_0\leq\gamma_2\frac{\kappa_{cg}^2}{\epsilon}$, then $\sigma_k\leq \max\left\{ \sigma_0,\gamma_2\frac{\kappa_{cg}^2}{\epsilon} \right\}$ for every $k\in\mathcal{I}_\epsilon$. Lastly, suppose $\sigma_0\geq\gamma_2\frac{\kappa_{cg}^2}{\epsilon}$. The result once again follows by the construction of the algorithm, as $\gamma_2>1$. 
\end{proof}

\begin{lemma}\label{lem:fosplemma2}
    Suppose that \hyperlink{f1}{(F1)}-\hyperlink{f3}{(F3)} hold. Also let $\epsilon\in(0,1)$ and $\mathcal{J}_\epsilon$ be some index set. If $\norm{c_k}_1>\epsilon$ for every $k\in\mathcal{J}_\epsilon$, then
    \begin{equation*}
        \sigma_k\leq \max\left\{ \sigma_0,\gamma_2\frac{\kappa_{cd}^2}{\epsilon^2} \right\}
    \end{equation*}
    for every $k\in\mathcal{J}_\epsilon$, where
    \begin{equation*}
        \kappa_{cd}:=\max\left\{ \frac{\theta}{\kappa_c \kappa_v}, \frac{\kappa_0\kappa_c \kappa_{v}(9\kappa_{\widetilde{g}}+1)}{\tau\mu_{-1}\theta(1-\eta_2)},\frac{\kappa_H}{\sqrt{\kappa_{\widetilde{g}}}} \right\}.
    \end{equation*}
\end{lemma}

\begin{proof}
     Suppose that  $\sqrt{\sigma_k}>\frac{\kappa_{cd}}{\epsilon}$ for some $k\in\mathcal{J}_\epsilon$. In order to show that $k\in \mathcal{J}_{\epsilon}$ is very successful, as in the proof of Lemma \ref{lem:fosplemma1}, it suffices to show that (\ref{eq:bad1}) holds. The relation (\ref{eq:good1}) still holds. From the bound on $\sigma_k$ and Lemma \ref{lem:tangentialstepbound} we get $\norm{u_k}^2\leq 9\frac{\kappa_{\widetilde{g}}}{\sigma_k}$. Further, we have $\norm{v_k}^2=\beta^2_k\norm{v_k^c}^2\leq\frac{1}{\sigma_k}$.  Moreover,
    \begin{equation}
        q_k(d_k)-\varphi(x_k,\mu_k)=-\Delta q_k\leq -\tau\mu_{-1}\min\left\{1,\frac{\theta}{\kappa_v \kappa_c \sqrt{\sigma_k}}\right\}\norm{c_k}_1.
    \end{equation}
Thereby, the left-hand-side of (\ref{eq:bad1}) is bounded above by the quantity
\begin{equation}\label{eq:badone}
    \frac{1}{\sigma_k}\left[ (9\kappa_{\widetilde{g}}+1)\kappa_0-\sigma_k(1-\eta_2)\tau\mu_{-1}\min\left\{1,\frac{\theta}{\kappa_v \kappa_c \sqrt{\sigma_k}}\right\}\norm{c_k}_1 \right].
\end{equation}
The bounds $\norm{c_k}_1>\epsilon$ and $\sigma_k>\frac{\kappa_{cd}^2}{\epsilon^2}$ imply that \eqref{eq:badone} is negative. The cases $0\in\mathcal{J}_\epsilon$ and $0\not\in\mathcal{J}_\epsilon$ are handled in similar manner to Lemma \ref{lem:fosplemma1}, leading to the desired bound on $\sigma_k$.
\end{proof}

We are now ready to present the first complexity result of this paper.

\begin{theorem}\label{th:fospcomplexity}
    Suppose that \hyperlink{f1}{(F1)}-\hyperlink{f4}{(F4)} hold. Then, for any $\epsilon_g,\epsilon_c\in(0,1)$, Algorithm \ref{alg:scp} will reach a point satisfying $\norm{\nabla_x\mathcal{L}(x_k,\lambda_k)}\leq\epsilon_g$ and $\norm{c_k}_1\leq\epsilon_c$ in at most
    \begin{equation*}
        M_\epsilon:=\ceil[\Big]{\kappa_q\max\left\{\epsilon_g^{-2},\epsilon_c^{-1}\right\}}
    \end{equation*}
    successful iterations, where
    \begin{equation*}
        \kappa_q:=\frac{f(x_0)-f_{low}+2\kappa_c\mu_{\max}}{\eta_1\min\{b_1,b_2\}}
    \end{equation*}
  for some positive constants $b_1,b_2$. Further, the algorithm will reach such a point in at most
  \begin{equation*}
        \widetilde{M}_\epsilon:=\ceil[\Bigg]{\frac{\widetilde{\kappa}_q}{\log\gamma_1}\max\left\{\epsilon_g^{-2},\epsilon_c^{-1}\right\}}
    \end{equation*}
    total iterations, where
    \begin{equation*}
    \widetilde{\kappa}_q:=(\log\gamma_1-\log\gamma_3)\kappa_q+\max\left\{ 1,\frac{\gamma_2\kappa_{cg}^2}{\sigma_{0}},\frac{\gamma_2\kappa_{cd}^2}{\sigma_{0}} \right\}.
    \end{equation*}
\end{theorem}

\begin{proof}
Define the set $\mathcal{S}_\epsilon:=\left\{ k\in\mathcal{S}:\norm{\nabla_x\mathcal{L}(x_k,\lambda_k)}>\epsilon_g\;\text{or}\;\norm{c_k}_1>\epsilon_c \right\}$. We also define the index sets
\begin{subequations}
    \begin{align}
\mathcal{I}_\epsilon&:=\left\{ k\in\mathcal{S}_{\epsilon}:\norm{\widetilde{g_k}}>\epsilon_g/2\geq \tilde{\delta}\epsilon_g/2\geq\sqrt{\norm{c_k}_1} \right\}; \\ \mathcal{J}_\epsilon&:=\Big\{ k\in\mathcal{S}_{\epsilon}:\norm{c_k}_1>\widetilde{\kappa} \min\left\{\epsilon_g^{2},\epsilon_c\right\} \Big\}, 
    \end{align}
\end{subequations}
where $\widetilde{\kappa}:=\Big[4\max\left\{\kappa_H \kappa_v,\gamma_A^{-2}\kappa_A r_{\lambda} \kappa_v,\tilde{\delta}^{-2},1\right\}\Big]^{-1}$. We consider cases about a fixed iteration $k\in\mathcal{S}_\epsilon$:\par

\underline{\textit{1st case}} $\left(k\in\mathcal{J}_\epsilon\right)$: Observe that $\widetilde{\kappa}\in (0,1)$, hence we can apply Lemma \ref{lem:fosplemma2} with $\epsilon:=\widetilde{\kappa} \min\left\{\epsilon_g^{2},\epsilon_c\right\}$ to acquire
\begin{equation}
    \Delta q_k\geq \tau \mu_k\beta_k\norm{c_k}_1 > b_1 \min\left\{\epsilon_g^{2},\epsilon_c\right\},
\end{equation}
where $b_1:=\widetilde{\kappa}\tau\mu_{-1}\min\left\{1,\frac{\theta}{\kappa_v\max\{\sqrt{\sigma_0}, \kappa_{cd}\sqrt{\gamma_2}\}}\right\}$.

\underline{\textit{2nd case}} $\left(k\not\in\mathcal{J}_\epsilon\right)$: In this case we have $\norm{\nabla_x\mathcal{L}(x_k,\lambda_k)}>\epsilon_g$, since $\widetilde{\kappa}<1$ and $k\in \mathcal{S}_{\epsilon}$. In turn, from Lemma \ref{lem:normalstepbound} and the fact $\epsilon_g\in (0,1)$,
\begin{equation}
    \norm{P_k^\top H_k v_k}\leq \kappa_H \kappa_v\norm{c_k}_1\leq \kappa_H \kappa_v \widetilde{\kappa} \min\left\{\epsilon_g^2,\epsilon_c\right\}\leq  \kappa_H\kappa_v\widetilde{\kappa}\epsilon_g^2\leq\epsilon_g/4.
\end{equation}
Furthermore, by the triangle inequality, Lemma \ref{lem:dualvariables} and \eqref{eq:gradientdiff} we have
\begin{equation}
    \norm{\nabla_x\mathcal{L}(x_k,\lambda_k)}\leq \frac{\kappa_Ar_{\lambda}\kappa_v}{\gamma_A^2}\norm{c_k}_1+\norm{P_k^\top g_k}.
\end{equation}
The last two relations yield
\begin{align}
    \norm{\widetilde{g_k}}&\geq \norm{P_k^\top g_k}-\norm{P_k^\top H_k v_k} \nonumber \\
    &> \epsilon_g-\frac{\kappa_A r_{\lambda} \kappa_v}{\gamma_A^2} \widetilde{\kappa}\epsilon_g^2-\epsilon_g/4 \nonumber \\
    &\geq \epsilon_g/2,
\end{align}
that is, iteration $k$ lies in $\mathcal{I}_\epsilon$. From Lemmas \ref{lem:deltagkbound} and \ref{lem:fosplemma1} (for $\epsilon:=\epsilon_g/2$) we then have
\begin{equation}
    \Delta q_k\geq \Delta m_k^U> b_2 \epsilon_g^2\geq b_2 \min\left\{\epsilon_g^{2},\epsilon_c\right\},
\end{equation}
where $b_2:=3\left[ 40\max\{1+\kappa_H,\max\{\sqrt{\sigma_0},\kappa_{cg}\sqrt{\gamma_2}\}\}\right]^{-1}$.\par
As a result,
\begin{equation}\label{eq:summing}
    \varphi(x_k,\mu_k)-\varphi(x_{k+1},\mu_k) > \eta_1 \min\{b_1,b_2\}\min\left\{\epsilon_g^{2},\epsilon_c\right\}\quad \forall k\in\mathcal{S}_\epsilon.
\end{equation}
Let us write $\mathcal{S}_{\epsilon}=\{k_j:j=0,...,j_s\}$, where $k_{j_{s}}+1\leq \infty$ is the first iterate satisfying \eqref{eq:fosp}. We have
\begin{align}
    \sum_{j=0}^{j_s}\left( \varphi(x_{k_j},\mu_{k_j})-\varphi(x_{k_{j+1}},\mu_{k_j}) \right)&\leq \varphi(x_{k_0},\mu_{k_0})-\varphi\left(x_{k_{j_s}+1},\mu_{k_{j_s}+1}\right)+\kappa_c(\mu_{\max}-\mu_{k_0}) \nonumber \\
    &\leq f(x_0)-f_{low}+2\kappa_c\mu_{\max},    \label{eq:must}
\end{align}
where we invoked \eqref{eq:telescopic}, Lemma \ref{lem:meritparameterbound}, and the fact that $\varphi(x_{k_0},\mu_{k_0})\leq \varphi(x_0,\mu_{k_0})+\mu_{k_0}\norm{c_0}_1$. To verify the latter bound, suppose $0\not\in\mathcal{S}$ (otherwise $k_0=0\in\mathcal{S}_{\epsilon}$, in which case there is nothing to show). Then, by construction, $x_0=x_{k_0-1}$. Since $k_0\in\mathcal{S}$, we have $\varphi(x_0,\mu_{k_0-1})\geq \varphi(x_{k_0},\mu_{k_0-1})$, or equivalently,
\begin{equation}
    \varphi(x_{k_0},\mu_{k_0})\leq \varphi(x_0,\mu_{k_0})+\left(\norm{c_0}_1-\norm{c_{k_0}}_1\right)\left(\mu_{k_0}-\mu_{k_0-1}\right),
\end{equation}
and the bound follows by monotonicity of the sequence of merit parameters.

Thereof, by summing (\ref{eq:summing}) over all $k\in\mathcal{S}_{\epsilon}$, we deduce that $j_s<\infty$ and
\begin{equation}
    f(x_0)-f_{low}+2\kappa_c\mu_{\max}>|\mathcal{S}_\epsilon|\eta_1 \min\{b_1,b_2\}\min\left\{\epsilon_g^{2},\epsilon_c\right\}.
\end{equation}
Consequently,
\begin{equation}
    |\mathcal{S}_{\epsilon}|\leq \ceil[\Bigg]{\frac{f(x_0)-f_{low}+2\kappa_c\mu_{\max}}{\eta_1\min\{b_1,b_2\}}\max\left\{\epsilon_g^{-2},\epsilon_c^{-1}\right\}},
\end{equation}
since $|\mathcal{S}_\epsilon|$ must be an integer.\par
For the total iteration complexity bound, notice that $\norm{\nabla_x\mathcal{L}(x_{k_{j_s}+1},\lambda_{k_{j_s}+1})}\leq\epsilon_g$ and $\norm{c_{k_{j_s}+1}}_1\leq \epsilon_c$. Since $\epsilon_g,\epsilon_c\in(0,1)$, by taking cases about whether an iteration $k\leq k_{j_s}$ is in $\mathcal{J}_\epsilon$ as above, and invoking Lemmas \ref{lem:fosplemma1} and \ref{lem:fosplemma2}, we find
\begin{equation}
    \sigma_k\leq \frac{\max\left\{ \sigma_0,\gamma_2\kappa_{cg}^2,\gamma_2\kappa_{cd}^2 \right\}}{\min\left\{\epsilon_g^{2},\epsilon_c\right\}}\quad\forall k\leq k_{j_s}.
\end{equation}
Therefore, we have 
\begin{equation}
    \log\left(\frac{\sigma_k}{\sigma_0}\right)\leq \log\left(\frac{\bar{\sigma}}{\min\left\{\epsilon_g^{2},\epsilon_c\right\}}\right)\;\;\forall k\leq k_{j_s},
\end{equation}
where $\bar{\sigma}:=\max\left\{ 1,\frac{\gamma_2\kappa_{cg}^2}{\sigma_0},\frac{\gamma_2\kappa_{cd}^2}{\sigma_0} \right\}$. Since $k_{j_s}=|\mathcal{U}_{k_{j_s}}|+|\mathcal{S}_{k_{j_s}}|$ and $|\mathcal{S}_{k_{j_s}}|=|\mathcal{S}_\epsilon|$, the final iteration complexity bound follows from Lemma \ref{lem:cardinalitybound} and the inequality $\log x\leq x$ for every $x>1$.
\end{proof}

\section{Second-order stationary points}\label{sec:sec4}

In this section we prove global convergence and worst-case complexity bounds to second-order stationary points. Besides the standard \hyperlink{f1}{(F1)}-\hyperlink{f4}{(F4)}, we make an additional assumption:\hypertarget{f5}{}

\noindent\textbf{\hyperlink{f5}{(F5)}}: The functions $\nabla^2 f$ and $\nabla^2 c_i$, for all $i \in \{1,\dots,m\}$, are bounded and Lipschitz continuous over $\mathcal{X}$ with constants $\kappa_{fh},\;L_{fh}$ and $\kappa_{ch_i},\;L_{ch_i}$, respectively.

Similarly to Assumption \hyperlink{f1}{(F1)}, we denote $\kappa_{ch} := \sum_{i=1}^m \kappa_{ch_i}$ and $L_{ch} := \sum_{i=1}^m L_{ch_i}$. This assumption simply translates to second-order information of the objective and the constraints, the former of which is needed for optimal complexity guarantees in the unconstrained non-convex optimization setting \cite{nesterov2006cubic}. A significant observation is that \hyperlink{f3}{(F3)} becomes redundant under \hyperlink{f5}{(F5)}, due to the definition of the Hessian (i.e., \eqref{eq:Hessian} and \eqref{eq:lambdas}), and assumptions \hyperlink{f1}{(F1)} and \hyperlink{f2}{(F2)} (Lemma \ref{lem:dualbounds}).

\subsection{Second-order Convergence}\label{sec:sec4.1}

We first state the main global convergence result of this paper:

\begin{theorem}\label{th:secondorderconvergence}
    Suppose that \hyperlink{f1}{(F1)}-\hyperlink{f5}{(F5)} hold. Then $\lim\limits_{k\to\infty}\norm{\nabla_x\mathcal{L}(x_k,\lambda_k)}=0$, $\lim\limits_{k\to\infty}\norm{c_k}_1=0$ and $\liminf\limits_{k\to\infty}\lambda_{\min}(Z_k^\top H_k Z_k)\geq0$.
\end{theorem}

In comparison to the previous section, for strong second-order theoretical guarantees, all approximation criteria for the tangential step provided by the \textit{SCP Oracle} are needed. These can be established in practice via the classical Lanczos method (Proposition \ref{prop:lanczos}). The following bounds on the decrease of the model $m_k^U$ are standard and have appeared previously in the fundamental works of Cartis et al. \cite{cartis2011adaptive, cartis2011adaptive2}, over an unconstrained optimization setting and slightly different termination criteria.

\begin{lemma}\label{lem:deltanullstepbounds}
    If $u_k\in\mathcal{R}(P_k)$ satisfies \hyperlink{or2}{(OR2)} and \hyperlink{or3}{(OR3)}, then $\Delta m_k^U(u_k)\geq(\frac{1}{6}-\delta)\sigma_k \norm{u_k}^3$ for every $k\geq 0$.
\end{lemma}

\begin{proof}
    By definition we have $\Delta m_k^U(u_k)=-\widetilde{g}_k^\top u_k-\frac{1}{2}u_k^\top \widetilde{H}_k u_k-\frac{\sigma_k}{3}\norm{u_k}^3$. In addition,
    \begin{equation}\label{eq:4.1}
        \nabla m_k^U(u_k)^\top u_k=\widetilde{g_k}^\top u_k+u_k^\top \widetilde{H}_k u_k+\sigma_k\norm{u_k}^3.
    \end{equation}
Condition \hyperlink{or2}{(OR2)} then implies
\begin{equation}\label{eq:4.2}
    \nabla m_k^U(u_k)^\top u_k\leq \norm{\nabla m_k^U(u_k)}\norm{u_k}\leq \delta \sigma_k \norm{u_k}^3,
\end{equation}
thus,
\begin{equation}\label{eq:lemmaOR23}
        \Delta m_k^U(u_k)\geq \frac{1}{2}u_k^\top \widetilde{H}_k u_k +\left(\frac{2}{3}-\delta\right)\sigma_k\norm{u_k}^3.
    \end{equation}
Thus, by \hyperlink{or3}{(OR3)}, we have $\Delta m_k^U(u_k)\geq (\frac{1}{6}-\delta)\sigma_k\norm{u_k}^3$.
\end{proof}

The next technical result along with Theorem \ref{th:constraintsconvergence} imply that the sequence of accepted trial steps $\{d_k\}_{k\in\mathcal{S}}$ eventually vanishes. 

\begin{lemma}\label{lem:nullstepconvergence}
    Suppose that \hyperlink{f1}{(F1)}-\hyperlink{f4}{(F4)} hold, and that the set $\mathcal{S}$ of all successful iterations is infinite. Then, $\lim\limits_{\substack{k\to\infty \\ k\in\mathcal{S}}}\norm{u_k}=0$.
\end{lemma}

\begin{proof}
    Lemma \ref{lem:reduction},  \ref{lem:deltanullstepbounds} and the bound $\sigma_k\geq \sigma_{\min}$ give
    \begin{equation}
        \Delta q_k\geq\Delta m_k^U(u_k)\geq\sigma_{\min}\left(\frac{1}{6}-\delta\right)\norm{u_k}^3\;\;\forall k\geq 0.
    \end{equation}
Since $\mathcal{S}$ is infinite, we can apply Lemma \ref{lem:technicallemma1} with $\mathcal{I}:=\mathcal{S}$ and $\alpha_k:=\sigma_{\min}\left(\frac{1}{6}-\delta\right)\norm{u_k}^3$ to get $\norm{u_k}\to 0$ as $k\to\infty$, $k\in\mathcal{S}$.
\end{proof}

 One crucial property of the algorithm proposed is that sufficient decrease in the model of the merit function is always achieved and the ratios $\rho_k,\;\rho_k^{corr}$ are always well-defined. This property is now rigorously verified: It is shown that sufficient predicted reduction is always achieved, i.e., $\Delta q_k>0$ holds at each iteration, unless we are at a second-order stationary point. Recall that positive predicted reduction was guaranteed in Remark \ref{rem:remark1} when first-order points were the points of interest.\par
 Fix an iteration $k\geq 0$ and consider the following cases:
\begin{itemize}
    \item If $\norm{c_k}_1\neq 0$, then $\Delta q_k>0$ due to Lemma \ref{lem:reduction}.
    \item If $\norm{c_k}_1=0$ and $\norm{\nabla_x\mathcal{L}(x_k,\lambda_k)}\neq0$, then $\norm{\nabla_x\mathcal{L}(x_k,\lambda_k)}=\norm{\widetilde{g}_k}\neq 0$ by Lemma \ref{lem:dualvariables}, since $\lambda_k^*=\lambda_k$ due to Lemma \ref{lem:dualbounds}. Therefore, $\Delta q_k>0$ by Lemma \ref{lem:deltagkbound}.
    \item  If $\norm{c_k}_1=0$, $\norm{\nabla_x\mathcal{L}(x_k,\lambda_k)}=0$ and $\lambda_{\min}(Z_k^\top H_k Z_k)<0$, then $\norm{u_k}\neq0$ due to \hyperlink{or3}{(OR3)} and Lemma \ref{lem:dualvariables}, thus $\Delta q_k>0$ by Lemma \ref{lem:deltanullstepbounds}.
    \item If $\norm{c_k}_1=0$, $\norm{\nabla_x\mathcal{L}(x_k,\lambda_k)}=0$ and $\lambda_{\min}(Z_k^\top H_k Z_k)\geq 0$, then $x_k$ is a second-order stationary point, so the SCP algorithm terminates.
\end{itemize}

Next, we relate the residuals of the approximations for the correction steps with the true least-squares estimators:

\begin{lemma}\label{lem:correctionbounds}
    If \hyperlink{f2}{(F2)} holds, then $\norm{w_k-w_k^*}\leq \frac{r_w}{\gamma_A}\norm{d_k}^3$.
\end{lemma}

\begin{proof}
     The LICQ assumption yields
\begin{align*}
    \gamma_A\norm{w_k-w_k^*}&\leq \norm{A_k(w_k-w_k^*)} \\ &=\norm{A_k w_k+c(x_k+d_k)-A_k w_k^*-c(x_k+d_k)} \nonumber \\  &=\norm{A_k w_k+c(x_k+d_k)}\\ &\leq r_w\norm{d_k}^3,
\end{align*}
where we've used \eqref{eq:correction} and equivalence of norms along with $w_k-w_k^*\in\mathcal{R}(A_k^\top)$.
\end{proof}

As second-order corrections are incorporated to ensure global convergence, we provide a useful upper bound.

\begin{lemma}\label{lem:wbound}
    Suppose that \hyperlink{f1}{(F1)}-\hyperlink{f2}{(F2)} and \hyperlink{f5}{(F5)} hold. If $k\in\mathcal{C}$, where $\mathcal{C}$ is defined in \eqref{eq:cset}, then $\norm{w_k}\leq\kappa_w\norm{d_k}^2$, for some $\kappa_w>0$.
\end{lemma}

\begin{proof}
    Since $k\in\mathcal{C}$, we must have $k\in\mathcal{K}$. The inequality $\norm{v_k^c}\leq\frac{\zeta}{\sqrt{\sigma_k}}$ yields $\beta_k=1$ by definition \eqref{eq:beta}. Therefore, from (\ref{eq:normalstep2}) and $u_k \in \mathcal{N}(A_k)$, we have $A_k d_k = A_k v_k^c$. Thus, by \eqref{eq:normalstep3}, (\ref{eq:correction}), \eqref{eq:w} and Lemmas \ref{lem:lipcon}, \ref{lem:fullstepbound}, and \ref{lem:correctionbounds}, we get
\begin{align}
    \norm{w_k}&\leq \norm{w_k^*-w_k}+\norm{w_k^*} \nonumber \\
    &\leq \frac{r_w}{\gamma_A}\norm{d_k}^3+\frac{1}{\gamma_A}\norm{c(x_k+d_k)} \nonumber \\
    &\leq \frac{r_w \kappa_d}{\gamma_A \sigma_{\min}}\norm{d_k}^2+\frac{r_v \kappa_d}{\gamma_A \sigma_{\min}}\norm{d_k}^2+\frac{L_A}{2\gamma_A}\norm{d_k}^2.
\end{align}
The result follows for $\kappa_w:= \frac{\kappa_d\left(r_w+r_v\right)}{\gamma_A\sigma_{\min}}+\frac{L_A}{2\gamma_A}$.
\end{proof}

Our next task is to show that Lipschitz continuity of the Hessian of the Lagrangian leads to an upper bound on the sequence of cubic regularization parameters. Second-order correction steps will always force the SCP algorithm to accept a trial step after only a finite number of ``failed attempts". In other words, the set $\mathcal{S}$ of all successful iterations has to be infinite. It is important to note that this result - and thus assumption \hyperlink{f5}{(F5)} - simplifies the proof of Theorems \ref{th:firstorderconvergence} and \ref{th:fospcomplexity} significantly.

\begin{lemma}\label{lem:sigmabound}
Suppose that \hyperlink{f1}{(F1)}-\hyperlink{f2}{(F2)} and \hyperlink{f5}{(F5)} hold. Then, $\sigma_k\leq\max\{\sigma_0,\gamma_2\sigma^*\}=:\sigma_{\max}$ for every $k\geq0$, for some $\sigma^*>0$.    
\end{lemma}

\begin{proof}
    We shall identify two positive quantities $C_1$, $C_2$,  such that
    \begin{equation}\label{eq:review}
        \sigma_k\geq \max\{C_1,C_2\}\;\Longrightarrow\;k\in\mathcal{S}.
    \end{equation}
The reason why this implication is sufficient for bounding the cubic parameter uniformly is very similar to that of Lemma \ref{lem:fosplemma1} (also see \cite[Lemma 5.2]{cartis2011adaptive}): Suppose we set $\sigma^*:=\max\{C_1,C_2\}$. Then, by construction of the Algorithm, the factor $\gamma_2>1$ will take care of the case when $\sigma_k$ is slightly below $\sigma^*$ and $k$ is not (very) successful, while $\sigma_0$ amounts to the initialization choice.

To show \eqref{eq:review}, we first assume that $x_k$ lies away from the feasible region. More specifically, suppose $k\not\in\mathcal{K}$. By Lemmas \ref{lem:fullstepbound} and \ref{lem:ratiobound} we obtain (\ref{eq:important}) (ignoring the index set $\mathcal{I}$). From Lemma \ref{lem:normalstepbound} we get
\begin{equation}
    \Delta q_k\geq \tau\mu_k\beta_k\norm{c_k}_1\geq\frac{\tau\mu_{-1}}{\kappa_v}\min\left\{\frac{\zeta}{\sqrt{\sigma_k}},\frac{\theta}{\sqrt{\sigma_k}}\right\}=\frac{\tau\mu_{-1}\zeta}{\kappa_v\sqrt{\sigma_k}}.
\end{equation}
Therefore,
\begin{equation}\label{eq:ratiobound1}
    1-\rho_k\leq \frac{\kappa_{\varphi}\kappa_v}{\tau\mu_{-1}\zeta}\frac{1}{\sqrt{\sigma_k}}.
\end{equation}
Thus, if $\sigma_k\geq C_1:=\left( \frac{\nu\kappa_{\varphi}\kappa_v}{\tau\mu_{-1}\zeta(1-\eta_1)} \right)^2$, then $k$ is successful.

Next, suppose that $k\in\mathcal{K}$ and $\sigma_k\geq C_1$. Without loss of generality, we may assume that the first ratio test fails, thus we need a correction step, i.e., we only need to explore the case $k\in\mathcal{C}$ (if for the iterate $k\in\mathcal{K}$ the lower bound $\sigma_k\geq C_1$ implies $\rho_k\geq \eta_1$, then we have nothing to show). The final result boils down to the existence of a constant $C_2$ such that, if $\sigma_k\geq C_2$, then 
    \begin{equation}
        \varphi(x_k+d_k+w_k,\mu_k)-q_k(d_k)\leq0.
    \end{equation}
This will imply
    \begin{equation}
        \rho_k^{corr}=\frac{\varphi(x_k,\mu_k)-(\varphi(x_k+d_k+w_k,\mu_k)-q_k(d_k))-q_k(d_k)}{\Delta q_k}\geq 1>\eta_2,
    \end{equation}
since $\varphi(x_k,\mu_k)-q_k(d_k)=q_k(0)-q_k(d_k)=\Delta q_k$. We have 
\begin{align}\label{eq:4.10}
    \varphi(x_k+d_k+w_k,\mu_k)-q_k(d_k) &= \left[f(x_k+d_k+w_k)-f_k-g_k^\top d_k-\frac{1}{2}d_k^\top H_k d_k -\frac{\sigma_k}{3}\norm{d_k}^3 \right] +\nonumber \\
    &+ \left[ \mu_k\norm{c(x_k+d_k+w_k)}_1-\mu_k\norm{c_k+A_k d_k}_1 \right].
\end{align}
By Taylor's theorem around $f(x_k+d_k)$ and $f(x_k+d_k+w_k)$ (applied twice by adding and subtracting $f(x_k+d_k)$ and $g(x_k+d_k)^\top w_k$) we can bound the first bracket term of (\ref{eq:4.10}) above by the quantity
\begin{equation}\label{eq:4.11}
    \frac{1}{2}\left|d_k^\top \left(\nabla^2f_k-\nabla^2f(\xi_k)\right)d_k\right|+\frac{1}{2}\left|w_k^\top\nabla^2f(\hat{\xi}_k) w_k\right|+\left| 
g(x_k+d_k)^\top w_k - \frac{1}{2}\sum_{i=1}^{m}\lambda_k^i d_k\nabla^2 c_k^i d_k \right|-\frac{\sigma_k}{3}\norm{d_k}^3
\end{equation}
where $\xi_k\in[x_k,x_k+d_k]$ and $\hat{\xi}_k\in[x_k+d_k,x_k+d_k+w_k]$. In turn, the first quantity of (\ref{eq:4.11}) is bounded above by $\frac{1}{2}L_{fh}\norm{d_k}^3$ due to Lipschitz continuity of the Hessian of the objective, and the second by $\frac{1}{2}\kappa_{fh}\norm{w_k}^2$. Further, the last absolute value of (\ref{eq:4.11}) is bounded by
\begin{equation}\label{eq:4.12}
    \left|g(x_k+d_k)^\top w_k-g_k^\top w_k\right| + \left| 
g_k^\top w_k - \frac{1}{2}\sum_{i=1}^{m}\lambda_k^i d_k\nabla^2 c_k^i d_k\right|.
\end{equation}
By Lipschitz continuity of the gradient of the objective, the first term of \eqref{eq:4.12} is bounded above by $L_g\norm{d_k}\cdot\norm{w_k}$. It remains to bound the last absolute value of \eqref{eq:4.12}. Recalling \eqref{eq:lambdas} and \eqref{eq:w}, we have that there exist vectors $\xi_k^{\lambda} \in \mathbb{R}^m$ and $\xi_k^w \in \mathbb{R}^m$ such that
\begin{equation*}
    A_k g_k + A_k A_k^{\top} \lambda_k + \xi_k^{\lambda} = 0 \quad \text{and} \quad \norm{\xi_k^{\lambda}} \leq r_\lambda \norm{v_k},
\end{equation*}
and
\begin{equation*}
    A_k w_k + c(x_k + d_k) + \xi_k^{w} = 0 \quad \text{and} \quad \norm{\xi_k^{w}} \leq r_w \norm{d_k}^3.
\end{equation*}
Since $w_k\in \mathcal{R}(A_k^\top)$, we have $w_k=A_k^\top (A_k A_k^\top)^{-1}A_k w_k$ by Lemma \ref{lem:dualvariables}. Therefore, we get
\begin{align}
    g_k^\top w_k&=g_k^\top A_k^\top (A_k A_k^\top)^{-1}A_k w_k \nonumber \\
    &=(A_k g_k)^\top (A_k A_k^\top)^{-1}A_k w_k \nonumber \\
    &=-\lambda_k^\top A_k w_k-(\xi_k^\lambda)^\top (A_k A_k^\top)^{-1}A_k w_k \nonumber \\
    &=\lambda_k^\top c(x_k+d_k)+\lambda_k^\top \xi_k^w-(\xi_k^\lambda)^\top (A_k A_k^\top)^{-1}A_k w_k. \label{eq:gw}
\end{align}
The last two terms in \eqref{eq:gw} are bounded by $r_w\kappa_{\lambda} \norm{d_k}^3$ and $\kappa_A r_\lambda \|v_k\| \|w_k\|/\gamma_A^2$, respectively. In addition, since $\beta_k = 1$, by Taylor's theorem we have
\begin{equation}
    \lambda_k^{\top} c(x_k + d_k) = \lambda_k^{\top} (c_k + A_k d_k) + \frac12 \sum_{i=1}^m \lambda_k^i d_k^{\top} \nabla^2 c^i(\tilde{\xi}_k^i) d_k,
\end{equation}
where each $\tilde{\xi}_k^i$ lies on the line segment between $x_k$ and $x_k + d_k$. Thus, the final term in \eqref{eq:4.12} is bounded by
\begin{align*}
    \left| g_k^\top w_k - \frac{1}{2}\sum_{i=1}^{m}\lambda_k^i d_k\nabla^2 c_k^i d_k\right| &\leq r_w\kappa_{\lambda} \norm{d_k}^3 + \frac{\kappa_A r_\lambda}{\gamma_A^2} \|v_k\| \|w_k\| + \|\lambda_k\| \|c_k + A_k d_k\| \\
    &\quad+ \frac12 \sum_{i=1}^m \left|\lambda_k^i\right| \left|d_k^{\top} (\nabla^2 c^i(\tilde{\xi}_k^i)- \nabla^2 c^i_k)d_k\right| \\
    &\leq r_w\kappa_{\lambda} \norm{d_k}^3 + \frac{\kappa_A r_\lambda}{\gamma_A^2} \|v_k\| \|w_k\| + \kappa_\lambda r_v \|v_k^c\|^3 \\
    &\quad+ \frac12 \|\lambda_k\| \sum_{i=1}^m \norm{\nabla^2 c^i(\tilde{\xi}_k^i)- \nabla^2 c^i_k} \norm{d_k}^2 \\
    &\leq r_w\kappa_{\lambda} \norm{d_k}^3+ \frac{\kappa_A r_\lambda}{\gamma_A^2} \|d_k\| \|w_k\| + \kappa_\lambda r_v \|d_k\|^3 + \frac{L_{ch} \kappa_\lambda}{2} \|d_k\|^3.
\end{align*}

Regarding the last bracket term of (\ref{eq:4.10}), we use Taylor's theorem to obtain
\begin{equation}\label{eq:4.15}
c(x_k+d_k+w_k)=c(x_k+d_k)+A(\overline{\xi_k}) w_k= c(x_k+d_k) + A_kw_k+\left( A(\overline{\xi_k}) - A_k \right) w_k,
\end{equation}
where $\overline{\xi_k}$ lies in the segment between $x_k+d_k$ and $x_k+d_k+w_k$. Therefore, by Lemmas \ref{lem:meritparameterbound}, \ref{lem:correctionbounds}, relations \eqref{eq:w}, (\ref{eq:4.15}) and the triangle inequality, we find
\begin{equation}\label{eq:4.16}
    \mu_k\norm{c(x_k+d_k+w_k)}_1-   \mu_k\norm{c_k+A_k d_k}_1\leq \mu_{\max} r_w\norm{d_k}^3+\mu_{\max} L_A\norm{w_k}^2.
\end{equation}
Recall now that $\norm{w_k}\leq\kappa_w\norm{d_k}^2$ from Lemma \ref{lem:wbound}. Also recall from Lemma \ref{lem:fullstepbound} that $\norm{d_k}\leq\frac{\kappa_d}{\sqrt{\sigma_{\min}}}$. Therefore, the quantities $\norm{w_k}^2$ and $\norm{d_k}\cdot\norm{w_k}$ can be bounded above by $\gamma_w\norm{d_k}^3$ for some fixed $\gamma_w>0$. Using this fact, we combine all the above upper bounds of (\ref{eq:4.10}) to obtain a positive constant $C_2$ such that
\begin{equation}
    \varphi(x_{k}+d_k+w_k,\mu_k)-q_k(d_k)\leq \left( C_2-\sigma_k \right)\frac{\norm{d_k}^3}{3}.
\end{equation}
Therefore, for every $k\geq0$, there exist positive constants $C_1,C_2$ such that $\sigma_k\geq\max\{C_1,C_2\}$ implies  that $k$ is successful, and the result follows.
\end{proof}

The proof of Theorem \ref{th:secondorderconvergence} follows from Theorem \ref{th:firstorderconvergence} and the sequent non-negative curvature result:

\begin{theorem}\label{th:curvatureconvergence}
    Suppose that \hyperlink{f1}{(F1)}-\hyperlink{f5}{(F5)} hold. Then $\liminf\limits_{k\in\mathcal{S},\;k\to\infty}\lambda_{\min}(Z_k^\top H_k Z_k)\geq0$.
\end{theorem}

\begin{proof}
    The tangential step $u_k$ satisfies \hyperlink{or3}{(OR3)}, namely $\lambda_{\min}(P_k^\top H_k P_k)\geq -\sigma_k \norm{u_k}$.  Lemmas \ref{lem:dualvariables} and \ref{lem:sigmabound} yield 
    \begin{equation*}
        \sigma_{\max}\norm{u_k}\geq \sigma_k\norm{u_k}\geq -\lambda_{\min}(Z_k^\top H_k Z_k)=\left|\lambda_{\min}(Z_k^\top H_k Z_k)\right|
    \end{equation*}
for all iterates with $\lambda_{\min}(Z_k^\top H_k Z_k)<0$. As $\mathcal{S}$ is infinite (otherwise $\sigma_k\to\infty$; a contradiction due to Lemma \ref{lem:sigmabound}), the result is straightforward from Lemma \ref{lem:nullstepconvergence}.
\end{proof}

Note that the stronger relation $\lambda_{\min}(P_k^\top H_k P_k)=\min\left\{\lambda_{\min}(Z_k^\top H_k Z_k),0\right\}$ given by Lemma  \ref{lem:dualvariables} further yields $\liminf\limits_{k\in\mathcal{S},\;k\to\infty}\lambda_{\min}(P_k^\top H_k P_k)=0$.

\subsection{Second-order Complexity}\label{sec:sec4.2}

Given second-order derivative information we are able to prove optimal (with respect to the gradient and Hessian of the Lagrangian) complexity bounds to second-order critical points. We begin with a crucial technical result.

\begin{lemma}\label{lem:sosplemma1}
    Suppose that \hyperlink{f1}{(F1)}-\hyperlink{f2}{(F2)} and \hyperlink{f5}{(F5)} hold. Then, there exist constants $\psi_d>0$ and $\psi_u>0$ such that $\norm{P_{k+1}^\top g_{k+1}}\leq\psi_d\norm{d_k}^2+\psi_u\norm{u_k}^2$ for every $k\in\mathcal{S}$.
\end{lemma}

\begin{proof}
    Let $u_k$ be an (in)exact solution returned by the \textit{SCP Oracle}, for some $k\in\mathcal{S}$. Then, $q_k:=P_k u_k$ is an (in)exact solution of the constrained cubic subproblem
\vspace{-0.7cm}
\begin{center}
\begin{equation}\label{eq:cubicmodel}
\begin{array}{ll}
    \min\limits_{q\in\mathbb{R}^n}\displaystyle m_k^Q(q):= f_k+(g_k+H_k v_k)^\top q+\frac{1}{2}q^\top H_k q +\frac{ \sigma_k}{3}\norm{q}^3\\
    \hspace{0.08cm}\text{s.t.}\;\; A_k q=0
    \end{array}\vspace{0.1cm}
\end{equation}
\end{center}   
Consider the optimal dual variables for \eqref{eq:cubicmodel}
\begin{equation}\label{eq:leastsquares}
    \hat{\lambda}_k:=-(A_k A_k^\top)^{-1}A_k(g_k + H_k d_k + \sigma_k \norm{q_k} q_k),
\end{equation}
and let the residuals $r_k:=\nabla_q m_k^Q(q_k)+A_k^\top \hat{\lambda}_k$. If $u_k$ is an exact solution of (\ref{eq:mku}), then $q_k$ is also an exact solution of (\ref{eq:cubicmodel}), and the KKT conditions are satisfied exactly, meaning that $r_k=0$ by definition of $\hat{\lambda}_k$. Note that, as the \textit{SCP Oracle} might return an inexact solution of the cubic subproblem, it may be the case that $\norm{r_k}>0$.\par 
We have that
\begin{equation}\label{eq:left}
    \norm{P_{k+1}^\top g_{k+1}}=\norm{P_{k+1}^\top \left(g_{k+1}+A_{k+1}^\top\hat{\lambda}_k\right)}\leq \norm{g_{k+1}+A_{k+1}^\top\hat{\lambda}_k}.
\end{equation}
In turn,
\begin{equation}\label{eq:mainlemmabound1}
\norm{g_{k+1}+A_{k+1}^\top\hat{\lambda}_k}\leq \norm{g_{k+1}+A_{k+1}^\top\hat{\lambda}_k-\nabla m_k(d_k)-A_k^\top \hat{\lambda}_k}+\norm{\nabla m_k(d_k)+A_k^\top\hat{\lambda}_k-r_k}+\norm{r_k}.
\end{equation}
We bound each norm of the right-hand-side of (\ref{eq:mainlemmabound1}), individually. From the triangle inequality, the first norm is bounded above by the sum
\begin{equation}\label{eq:mainlemmabound2}
     \norm{g_{k+1}-g_k-\nabla^2 f_k d_k}+\sum_{i=1}^{m}\left|\hat{\lambda}_k^i\right|\norm{a^{i}_{k+1}-a^{i}_k-\nabla^2 c_{k}^i d_k}+\norm{\lambda_k-\hat{\lambda}_k}\norm{\nabla^2 c_k d_k}+\sigma_k\norm{d_k}^2.
\end{equation}

From a Taylor's expansion and \hyperlink{f5}{(F5)} (and adding and subtracting the term $\nabla^2 f_k w_k$, and then invoking Lemma \ref{lem:fullstepbound} in the case $k\in\mathcal{C}$), we have 
\begin{equation}
    \norm{g_{k+1}-g_k-\nabla^2 f_k d_k}\leq \left(L_{fh}\left(1+\frac{\kappa_w^2 \kappa_d^2}{\sigma_{\min}}\right)+\kappa_{fh} \kappa_w\right)\norm{d_k}^2.
\end{equation}
Similarly, by observing that $\hat{\lambda}_k$ is bounded above in norm by some $\kappa_{\hat{\lambda}}>0$ (due to \hyperlink{f1}{(F1)}-\hyperlink{f2}{(F2)} and Lemmas \ref{lem:tangentialstepbound} and \ref{lem:normalstepbound}), and invoking Lemma \ref{lem:fullstepbound} (adding and subtracting the terms $\nabla^2 c_k^i w_k$, if necessary), we find
\begin{equation}
\sum_{i=1}^{m}\left|\hat{\lambda}_k^i\right|\norm{a^{i}_{k+1}-a^{i}_k-\nabla^2 c_{k}^i d_k}\leq \kappa_{\hat{\lambda}}\left(L_{ch}\left(1+\frac{\kappa_w^2 \kappa_d^2}{\sigma_{\min}}\right) + \kappa_{ch} \kappa_w\right)\norm{d_k}^2.
\end{equation}
For the third term of (\ref{eq:mainlemmabound2}), Lemmas \ref{lem:dualbounds} and \ref{lem:sigmabound} yield
\begin{align}
    \norm{\lambda_k-\hat{\lambda}_k}\norm{\nabla^2 c_k d_k}&\leq \kappa_{ch}\norm{\lambda_k^*-\lambda_k}\norm{d_k}+\kappa_{ch}\norm{\lambda_k^*-\hat{\lambda}_k}\norm{d_k} \nonumber \\
    &\leq \frac{\kappa_{ch} r_\lambda}{\gamma_A^2}\norm{d_k}^2+\frac{\kappa_{ch}\kappa_A}{\gamma_A^2} \left(\kappa_H+\sqrt{\sigma_{\max}}\kappa_d\right) \norm{d_k}^2.    
\end{align}
The last term of \eqref{eq:mainlemmabound2} is bounded by $\sigma_{\max}\norm{d_k}^2$ due to Lemma \ref{lem:sigmabound}.\par

By definition of the residuals $r_k$ and Lemma \ref{lem:sigmabound}, for the second norm of (\ref{eq:mainlemmabound1}) we acquire
\begin{equation}
    \norm{\nabla m_k(d_k)+A_k^\top \hat{\lambda}_k-r_k}\leq \sigma_{\max}\norm{d_k}^2 +\sigma_{\max}\norm{u_k}^2.
\end{equation}

Lastly, Lemma \ref{lem:dualvariables} gives the important relation $r_k=P_k^\top \nabla_q m_k^Q(q_k)$. Since $q_k=P_k u_k$, where $u_k$ is an output of the \textit{SCP Oracle}, it follows that $P_k^\top \nabla_q m_k^Q(q_k)=\nabla_u m_k^U(u_k)$. As a consequence, the \hyperlink{or2}{(OR2)} criterion and Lemma \ref{lem:sigmabound} give the final bound
\begin{equation}\label{eq:right}
    \norm{r_k}=\norm{\nabla m_k^U(u_k)}\leq \delta \sigma_{\max}\norm{u_k}^2.
\end{equation}
From \eqref{eq:left}-\eqref{eq:right} we finally deduce the existence of constants $\psi_d,\psi_u>0$ such that $\norm{P_{k+1}^\top g_{k+1}}\leq\psi_d\norm{d_k}^2+\psi_u\norm{u_k}^2$.
\end{proof}

The next result shows that when the normal step dominates the tangential step, then the constraint violation at the next iteration cannot be much larger than its current value.

\begin{lemma}\label{lem:sosplemma3}
    Suppose that \hyperlink{f1}{(F1)}-\hyperlink{f5}{(F5)} hold and let an iterate $k\in\mathcal{S}$. If $\norm{u_k}<\norm{v_k}$, then $\norm{c_{k+1}}_1<\kappa_{cc}\norm{c_k}_1$ for some $\kappa_{cc}>1$.
\end{lemma}

\begin{proof}
    Suppose $k\in\mathcal{C}$. Then, (\ref{eq:4.16}) holds, hence one has (ignoring the penalty parameter terms) $\norm{c_{k+1}}_1\leq \frac{\kappa_A r_w}{\gamma_A}\norm{d_k}^3+ L_A\norm{w_k}^2$. Lemma \ref{lem:wbound} guarantees the existence of some constant $\gamma_c>0$ such that $\norm{c_{k+1}}_1\leq\gamma_c\norm{d_k}^3$. Therefore,
\begin{equation}    
\norm{c_{k+1}}_1\leq\gamma_c\frac{\kappa_d^2}{\sigma_{\min}}\norm{d_k}<\sqrt{2}\gamma_c\frac{\kappa_d^2}{\sigma_{\min}}\norm{v_k}\leq \sqrt{2}\gamma_c\frac{\kappa_d^2}{\sigma_{\min}}\kappa_v\norm{c_k}_1.
\end{equation}

If $k\not\in\mathcal{C}$, then no correction step is taken, so by Lemma \ref{lem:lipcon} and \eqref{eq:normalstep3},
\begin{equation*}
    \|c(x_{k+1})\|_1 \leq \|c_k + A_k d_k\|_1 + \frac{L_A}{2}\|d_k\|^2 <  \|c_k\|_1 + L_A \kappa_{vs} \|c_k\|_1 = (1 + L_A \kappa_{vs}) \|c_k\|_1.
\end{equation*}
The result follows for $\kappa_{cc}:=\max\left\{ \sqrt{2}\frac{\gamma_c\kappa_d^2}{\sigma_{\min}\kappa_v},1 + L_A \kappa_{vs} \right\}$.
\end{proof}

We are finally ready to present the main iteration complexity result of this paper:

\begin{theorem}\label{th:sospcomplexity}
    Suppose that \hyperlink{f1}{(F1)}-\hyperlink{f5}{(F5)} hold. Then, for any $\epsilon_g,\epsilon_c,\epsilon_H\in(0,1)$, Algorithm \ref{alg:scp} will reach a point satisfying $\|\nabla_x\mathcal{L}(x_k,\lambda_k)\| \leq \epsilon_g$, $\|c_k\|_1 \leq \epsilon_c$ and $\lambda_{\min}(Z_k^\top H_k Z_k)\geq -\epsilon_H$ in at most
    \begin{equation*}
        K_\epsilon:=\ceil[\Big]{\kappa_s\max\left\{\epsilon_g^{-3/2},\epsilon_c^{-1},\epsilon_H^{-3}\right\}}+1
    \end{equation*}
    successful iterations, where
    \begin{equation*}
        \kappa_s:=\frac{f(x_0)-f_{low}+2\kappa_c\mu_{\max}}{\eta_1\min\limits_{i=1,2,3,4}\alpha_i},
    \end{equation*}
for some positive constants $\alpha_1,\alpha_2,\alpha_3,\alpha_4$.  Further, Algorithm \ref{alg:scp} will reach such a point in at most
\begin{equation*}
    \widetilde{K}_\epsilon:=\ceil[\Bigg]{\displaystyle\frac{ \widetilde{\kappa}_s}{\log\gamma_1}\max\left\{\epsilon_g^{-3/2},\epsilon_c^{-1},\epsilon_H^{-3}\right\}}+1
\end{equation*}
total iterations, where
\begin{equation*}
    \widetilde{\kappa}_s:=(\log\gamma_1-\log\gamma_3)\kappa_s+\frac{\sigma_{\max}}{\sigma_0}.
\end{equation*}
\end{theorem}

\begin{proof}
    To show the first bound, it suffices to bound the number of successful iterations until we find a point satisfying
    \begin{equation*}
        \min\left\{\norm{\nabla_x\mathcal{L}(x_{k+1},\lambda_{k+1})},\norm{\nabla_x\mathcal{L}(x_k,\lambda_k)}\right\}\leq\epsilon_g , \quad \max\left\{\norm{c_k}_1,\norm{c_{k+1}}_1\right\}\leq\epsilon_c,
    \end{equation*}
    and
    \begin{equation*}
        \max\left\{\lambda_{\min}(Z_k^\top H_k Z_k),\lambda_{\min}(Z_{k+1}^\top H_{k+1} Z_{k+1})\right\}\geq -\epsilon_H.
    \end{equation*}
    To this end, recall from Lemma \ref{lem:sigmabound} that the set of all successful iterations $\mathcal{S}$ has to be infinite. Let us thus write, without any loss of generality, $\mathcal{S}=\{k_j:j\geq 0\}$. We then define the index subsets of $\mathcal{S}$,
    \begin{align*}
      \mathcal{F}_\epsilon&:=\Big\{ k_j\geq 0:\min\left\{\norm{\nabla_x\mathcal{L}(x_{k_{j+1}},\lambda_{k_{j+1}})},\norm{\nabla_x\mathcal{L}(x_{k_j},\lambda_{k_j})}\right\}>\epsilon_g\Big\}; \\
      \mathcal{H}_{\epsilon}&:=\Big\{ k_j\geq 0: \norm{c_{k_j}}_1>\hat{\kappa}\epsilon_c\text{\;or\;}\norm{c_{k_{j+1}}}_1>\bar{\kappa}\min\left\{\epsilon_g^{3/2},\epsilon_c\right\} \Big\}; \\
      \mathcal{G}_{\epsilon}&:=\left\{k_j\geq 0:\lambda_{\min}(Z_{k_j}^\top H_{k_j} Z_{k_j})<-\epsilon_H\;\text{or}\;\lambda_{\min}(Z_{k_{j+1}}^\top H_{k_{j+1}} Z_{k_{j+1}})<-\epsilon_H\right\}.
    \end{align*}
    for  $\hat{\kappa}:=\Big[4\max\left\{1,(\psi_d+\psi_u)\kappa_{vs},\kappa_{cc},\kappa_{cc}\kappa_A r_{\lambda}\kappa_v\gamma_A^{-2},\kappa_A\kappa_c r_v\right\}\Big]^{-1}$, $\bar{\kappa}:=[2\max\{1,\kappa_A r_{\lambda}\kappa_v \gamma_A^{-2}\}]^{-1}$.
    
    Further, we set $\mathcal{C}_{\epsilon}:=\mathcal{F}_{\epsilon}\cup\mathcal{H}_{\epsilon}\cup \mathcal{G}_{\epsilon}$. Fix some $k_j\in\mathcal{C}_\epsilon$. We consider three cases for that successful iteration.\par
    \underline{\textit{1st case}} $\left(k_j\not\in\mathcal{G}_{\epsilon}\;\text{and}\;\norm{c_{k_j}}_1>\hat{\kappa}\min\left\{\epsilon_g^{3/2},\epsilon_c\right\}\right)$: Then,
    \begin{equation}\label{eq:complexitybound1}
        \Delta q_{k_j}\geq\tau\mu_{k_j}\norm{c_{k_j}}_1\min\left\{1,\frac{\theta}{\norm{v_{k_j}^c}\sqrt{\sigma_{k_j}}}\right\}>\widetilde{\alpha}_1\min\left\{\epsilon_g^{3/2},\epsilon_c,\epsilon_H^3\right\},
\end{equation}
where $\widetilde{\alpha}_1:=\tau\mu_{-1}\hat{\kappa}\min\left\{1,\frac{\theta}{\kappa_v\kappa_c\sqrt{\sigma_{\max}}}\right\}$ by Lemma \ref{lem:sigmabound}.\par

\underline{\textit{2nd case}} $\left(k_j\not\in\mathcal{G}_{\epsilon}\;\text{and}\;\norm{c_{k_j}}_1\leq\hat{\kappa}\min\left\{\epsilon_g^{3/2},\epsilon_c\right\}\right)$: We first claim that we must have $\norm{u_{k_j}}\geq\norm{v_{k_j}}$. Suppose otherwise. We invoke Lemma \ref{lem:sosplemma3} and the definitions of the constants $\bar{\kappa},\hat{\kappa}$ to obtain $\norm{c_{k_{j+1}}}_1<\kappa_{cc}\norm{c_{k_j}}_1\leq\bar{\kappa}\min\left\{\epsilon_g^{3/2},\epsilon_c\right\}$ (recall that the constraint violation remains fixed within two successful iterations, by construction, i.e., $\norm{c_{{k_j}+1}}_1=\norm{c_{k_j+2}}_1=...=\norm{c_{k_{j+1}}}_1$; a similar property holds for the norms of the projected gradient and gradient of the Lagrangian). Since $k_j\in\mathcal{C}_{\epsilon}$ and $k_j\not\in\mathcal{G}_{\epsilon}\cup \mathcal{H}_{\epsilon}$, we must have $k_j\in\mathcal{F}_{\epsilon}$. Observe that Lemmas \ref{lem:dualvariables}, \ref{lem:dualbounds} yield
\begin{align}
    \norm{\nabla_x\mathcal{L}(x_{k_{j+1}},\lambda_{k_{j+1}})}&\leq \norm{\nabla_x\mathcal{L}(x_{k_{j+1}},\lambda_{k_{j+1}}^*)-\nabla_x\mathcal{L}(x_{k_{j+1}},\lambda_{k_{j+1}})}+\norm{\nabla_x\mathcal{L}(x_{k_{j+1}},\lambda_{k_{j+1}}^*)} \nonumber \\
    &\leq \kappa_A\norm{\lambda_{k_{j+1}}^*-\lambda_{k_{j+1}}}+\norm{P_{k_{j+1}}^\top g_{k_{j+1}}} \nonumber \\
    &\leq \frac{\kappa_A r_{\lambda}\kappa_v\kappa_{cc}}{\gamma_A^2}\norm{c_{k_j}}_1+\norm{P_{k_{j+1}}^\top g_{k_{j+1}}}. \label{eq:cooll}
\end{align}
Rearranging \eqref{eq:cooll} and using $\epsilon_g \in (0,1)$ implies
\begin{equation}\label{eq:cool2}
    \norm{P_{k_{j+1}}^\top g_{k_{j+1}}}>\epsilon_g-\frac{\kappa_A r_{\lambda}\kappa_v\kappa_{cc}\hat{\kappa}}{\gamma_A^2}\epsilon_g^2>\epsilon_g/2.
\end{equation}
However, from Lemmas \ref{lem:normalstepbound} and \ref{lem:sosplemma1} we find
\begin{align}
    \norm{P_{k_{j+1}}^\top g_{k_{j+1}}}&\leq \psi_d\norm{d_{k_j}}^2+\psi_u\norm{u_{k_j}}^2 \nonumber \\
    &\leq 2(\psi_d+\psi_u) \norm{v_{k_j}}^2 \nonumber \\
    & \leq 2(\psi_d+\psi_u)\kappa_{vs}\norm{c_{k_j}}_1\nonumber \\
    &\leq \epsilon_g^{3/2}/2. \label{eq:cool3}
\end{align}
The relations \eqref{eq:cool2} and \eqref{eq:cool3} contradict one another, since $\epsilon_g\in(0,1)$.\par
Hence, we have $\norm{u_{k_j}}\geq\norm{v_{k_j}}$ whenever $\norm{c_{k_j}}_1\leq\hat{\kappa}\min\left\{\epsilon_g^{3/2},\epsilon_c\right\}$. We proceed by considering two sub-cases about the next iterate:\par

\textit{\underline{1st subcase}} $\left(\norm{c_{k_{j+1}}}_1\leq\bar{\kappa}\min\left\{\epsilon_g^{3/2},\epsilon_c\right\}\right)$:
 In this case, we must have $k_j\in \mathcal{F}_{\epsilon}$.  Coming back to the result of Lemma \ref{lem:sosplemma1}, we use $\norm{u_{k_j}}\geq\norm{v_{k_j}}$ to get
\begin{equation}\label{eq:superbound}
    \norm{P_{k_{j+1}}^\top g_{k_{j+1}}}\leq 2(\psi_d+\psi_u)\norm{u_{k_j}}^2 .
\end{equation}
Combining \eqref{eq:superbound} with the second inequality of \eqref{eq:cooll} and Lemma \ref{lem:dualbounds}, we obtain
\begin{align}
    \norm{\nabla_x\mathcal{L}(x_{k_{j+1}},\lambda_{k_{j+1}})}&\leq \kappa_A \norm{\lambda^*_{k_{j+1}}-\lambda_{k_{j+1}}}+\norm{P_{k_{j+1}}^\top g_{k_{j+1}}} \nonumber \\
    &\leq \frac{\kappa_A r_{\lambda}\kappa_v}{\gamma_A^2}\norm{c_{k_{j+1}}}_1+\norm{P_{k_{j+1}}^\top g_{k_{j+1}}} \nonumber \\
    &\leq \epsilon_g/2+2(\psi_d+\psi_u)\norm{u_{k_j}}^2,
\end{align}
Therefore, since $k_j\in \mathcal{F}_{\epsilon}$, we deduce
\begin{equation*}
    2(\psi_d+\psi_u)\norm{u_{k_j}}^2\geq \norm{\nabla_x\mathcal{L}(x_{k_{j+1}},\lambda_{k_{j+1}})}-\epsilon_g/2>\epsilon_g/2.
\end{equation*}
By invoking Lemma \ref{lem:deltanullstepbounds}, the above relation gives
\begin{equation}\label{eq:complexitybound2}
    \Delta q_{k_j}\geq \Delta m_{k_j}^U(u_{k_j})\geq \left(\frac{1}{6}-\delta\right)\sigma_{k_j}\norm{u_{k_j}}^3\geq\widetilde{\alpha}_2\epsilon_g^{3/2}\geq\widetilde{\alpha}_2 \min\left\{\epsilon_g^{3/2},\epsilon_c,\epsilon_H^3\right\},
\end{equation}
where $\widetilde{\alpha}_2:=\left(\frac{1}{6}-\delta\right)      \frac{\sigma_{\min}}{8\left(\psi_d+\psi_u\right)^{3/2}}$. 

\textit{\underline{2nd subcase}} $\left(\norm{c_{k_{j+1}}}_1>\bar{\kappa}\min\left\{\epsilon_g^{3/2},\epsilon_c\right\}\right)$: This subcase is treated analogously to the first case, so we deduce
\begin{equation}\label{eq:complexitybound4}
        \Delta q_{k_{j+1}}\geq\tau\mu_{k_{j+1}}\norm{c_{k_{j+1}}}_1\min\left\{1,\frac{\theta}{\norm{v_{k_{j+1}}^c}\sqrt{\sigma_{k_{j+1}}}}\right\}>\widetilde{\alpha}_1\min\left\{\epsilon_g^{3/2},\epsilon_c,\epsilon_H^3\right\}
\end{equation}
with $\widetilde{\alpha}_1$ as in (\ref{eq:complexitybound1}), because $\hat{\kappa}<\bar{\kappa}$ due to $\kappa_{cc}>1$.\par
We now consider the final possible case for an iterate $k_j\in\mathcal{C}_{\epsilon}$.

\underline{\textit{3rd case}} $\left(k_j\in \mathcal{G}_{\epsilon}\right)$: Suppose that the first disjunction holds. By the computational properties of the tangential step $u_{k_1}$, the bound $\lambda_{\min}(Z_{k_j}^\top H_{k_j} Z_{k_j})\geq \lambda_{\min}(P_{k_j}^\top H_{k_j} P_{k_j})$ (Lemma \ref{lem:dualvariables}), and Lemmas \ref{lem:deltanullstepbounds} and \ref{lem:sigmabound}, we have
    \begin{align}\label{eq:complexitybound5}
         \Delta q_{k_j}&\geq\Delta m_{k_j}^U(u_{k_j})\geq \left(\frac{1}{6}-\delta\right)\sigma_{k_j}\norm{u_{k_j}}^3 \nonumber \\
         &\geq\frac{\left(\frac{1}{6}-\delta\right)\sigma_{\min}}{\sigma_{\max}^3}(-\lambda_{\min}(Z_{k_j}^\top H_{k_j} Z_{k_j}))^3 \nonumber \\
    &\geq\frac{\left(\frac{1}{6}-\delta\right)\sigma_{\min}}{\sigma_{\max}^3}\epsilon_H^3 \nonumber \\
    &\geq \widetilde{\alpha}_4\min\left\{\epsilon_g^{3/2},\epsilon_c,\epsilon_H^3\right\},
    \end{align}
where $\widetilde{\alpha}_4:=\left(\frac{1}{6}-\delta\right)\sigma_{\min}/\sigma_{\max}^3$. Analogously, if the second disjunction holds, we deduce
\begin{equation}\label{eq:complexitybound6}
    \Delta q_{k_{j+1}}\geq \widetilde{\alpha}_4\min\left\{\epsilon_g^{3/2},\epsilon_c,\epsilon_H^3\right\}.
\end{equation}
Recall that $\varphi(x_k,\mu_k)-\varphi(x_{k+1},\mu_k)\geq \eta_1 \Delta q_k$ whenever $k\in \mathcal{S}$, by construction of the algorithm. Combining \eqref{eq:complexitybound1} and \eqref{eq:complexitybound2}-\eqref{eq:complexitybound6}, we thus conclude that for every $k_j\in\mathcal{C}_\epsilon$ at least one of the following holds:
\begin{subequations}
    \begin{align}
        \varphi(x_{k_j},\mu_{k_j})-\varphi(x_{k_{j+1}},\mu_{k_j})&\geq \eta_1\min\limits_{i=1,2,3,4}\{\widetilde{\alpha}_i\}\min\left\{\epsilon_g^{3/2},\epsilon_c,\epsilon_H^3\right\}; \label{eq:summable} \\
        \varphi(x_{k_{j+1}},\mu_{k_{j+1}})-\varphi(x_{k_{j+2}},\mu_{k_{j+1}})&\geq \eta_1\min\limits_{i=1,2,3,4}\{\widetilde{\alpha}_i\}\min\left\{\epsilon_g^{3/2},\epsilon_c,\epsilon_H^3\right\},
    \end{align}
\end{subequations}
In other words, at least half iterates in $\mathcal{C}_\epsilon$ will give a strictly positive reduction of order $\min\left\{\epsilon_g^{3/2},\epsilon_c,\epsilon_H^3\right\}$. By this dichotomy and the bound $\varphi(x_k,\mu_k)\geq \varphi(x_{k+1},\mu_k)$ for every $k\in\mathcal{S}$, summing \eqref{eq:summable} over $\mathcal{C}_\epsilon$  gives
\begin{equation}
    f(x_0)-f_{low}+2\kappa_c\mu_{\max}\geq \frac{1}{2}|\mathcal{C}_\epsilon|\eta_1\min\limits_{i=1,2,3,4}\{\widetilde{\alpha}_i\}\min\left\{\epsilon_g^{3/2},\epsilon_c,\epsilon_H^3\right\},
\end{equation}
in similar manner to \eqref{eq:must}. We finally get
\begin{equation}\label{eq:fospcomplexitybound}
    |\mathcal{C}_\epsilon|\leq \ceil[\Bigg]{\frac{f(x_0)-f_{low}+2\kappa_c\mu_{\max}}{\eta_1\min\limits_{i=1,2,3,4}\{\alpha_i\}}\max\left\{\epsilon_g^{-3/2},\epsilon_c^{-1},\epsilon_H^{-3}\right\}},
\end{equation}
where $\alpha_i:=\widetilde{\alpha}_i/2$ for every $i\in\{1,2,3,4\}$.

Now, suppose $k_{j_C}<\infty$ is the last iterate in $\mathcal{C}_{\epsilon}$. Since $x_{k_{j_C}}$ might not be an approximate SOSP, we infer that the algorithm will reach such a point in at most $K_\epsilon+1$ successful iterations, due to $\hat{\kappa},\bar{\kappa}<1$.\par

Lastly, we deduce the second, total iteration complexity bound verbatim Theorem \ref{th:fospcomplexity}, via the use of (\ref{eq:fospcomplexitybound}), Lemma \ref{lem:cardinalitybound}, the bound $\sigma_k\leq \sigma_{\max}$ and the inequality $\log x<x$ for $x>0$. The final total iteration complexity bound directly follows.
\end{proof}

\begin{remark}\label{rem:criterion}
    Going back to the discussion of Section \ref{sec:sec2}, one can replace the termination criterion \hyperlink{or2}{(OR2)} with \hyperlink{or2prime}{(OR2')}, the latter of which commonly appears in related works (e.g., \cite{cartis2013evaluation}). The bound given by Lemma \ref{lem:sosplemma1} will then change to ``$\norm{P_{k+1}^\top g_{k+1}}\leq\psi_d\norm{d_k}^2+\psi_u\norm{u_k}^2+\kappa_{\theta}\min\{1,\norm{u_k}\}\norm{\widetilde{g_k}}$". In turn, one can utilize this to get an iteration complexity threshold of the same order, with only a few suitable changes in the proof of the above Theorem (in fact, it suffices to compare the values $\norm{u_k}$ and $\norm{\widetilde{g_k}}$). Note that Lemma \ref{lem:deltanullstepbounds} is a crucial tool that directly makes use of the termination criterion \hyperlink{or2}{(OR2)}. Upon the substitute of \hyperlink{or2prime}{(OR2')}, this Lemma boils down to the model reduction guaranteed by \cite[Lemma 3.3]{cartis2011adaptive}, which is sufficient for the proof of the main results in this section.
\end{remark}

\section{Local quadratic convergence}\label{sec:sec5}

In this section we show that the SCP algorithm enjoys local quadratic convergence properties. We conclude that our work is universal within the class of second-order methods that offer both iteration complexity and local quadratic convergence warranties for equality constrained optimization problems\footnote{To the best of our knowledge, the only work we are aware of that claims both complexity bounds and local quadratic convergence is that of Goyens et al. \cite{goyens2024computing}. However, this type of fast local convergence is claimed in a remark in the latter work, while a formal statement and proof are omitted.}. Related papers that guarantee both fast local convergence properties and iterations complexity bounds in equality constrained optimization are those of Bai and Mei \cite{bai2018analysis} and Bourkhissi and Necoara \cite{bourkhissi2025complexity}, where linear and (sub)linear rates are given, respectively.

\begin{theorem}\label{th:quadconv}
    Suppose that \hyperlink{f1}{(F1)}-\hyperlink{f5}{(F5)} hold and the sequence $\{x_k\}$ converges to $x^*$ for some $(x^*, \lambda^*)$ satisfying second order sufficient conditions, i.e., $\nabla_x \mathcal{L}(x^*, \lambda^*) = 0$, $c(x^*) = 0$, and $Z_*^{\top} \nabla^2_{xx} \mathcal{L}(x^*, \lambda^*) Z_*$ is positive definite, where $Z_*$ is an orthonormal basis of $\mathcal{N}(A(x^*))$. Then, $\{x_k\}$ converges to $x^*$ Q-quadratically.
\end{theorem}

\begin{proof}
First, under the assumption that $x_k \rightarrow x^*$, we establish a number of properties of the iterates of Algorithm \ref{alg:scp} for all $k$ sufficiently large. Note that since $c(x^*) = 0$, by Lipschitz continuity of $c$, it follows that $\|c_k\| = \mathcal{O}(\|x_k - x^*\|)$. By \eqref{eq:normalstep3}, \eqref{eq:beta}, Lemma \ref{lem:normalstepbound}, and Lemma \ref{lem:sigmabound}, we have that, for all $k$ sufficiently large, $\beta_k = 1$ and thus $v_k = v_k^c$ for all $k$ sufficently large.
    
Now, since $Z_*^\top \nabla^2_{xx} \mathcal{L}(x^*, \lambda^*) Z_*$ is positive definite, we have that $u^\top P_*^\top \nabla^2_{xx} \mathcal{L}(x^*, \lambda^*) P_* u > 0$ for all $u \neq 0, u \in \mathcal{N}(A(x^*))$, where $P_*$ is the orthogonal projection matrix on $\mathcal{N}(A(x^*))$. Then, by \hyperlink{f2}{(F2)} and \hyperlink{f5}{(F5)}, for all $k$ sufficiently large, $u^\top P_k^\top \nabla^2_{xx} \mathcal{L}(x_k, \lambda_k^*) P_k u > 0$ for all $u \neq 0, u\in \mathcal{N}(A_k)$. Further, we have $\|\nabla^2_{xx} \mathcal{L}(x_k, \lambda_k^*) - H_k\| \leq \mathcal{O}(\|x_k - x^*\|)$ by \eqref{eq:gradientdiff}. It follows that for all $k$ sufficiently large, there exists $c_{H_{\min}} > 0$ such that $u^\top P_k^\top H_k P_k u \geq c_{H_{\min}}$ for all $u \neq 0, u\in \mathcal{N}(A_k)$. Therefore, from \eqref{eq:4.1}-\eqref{eq:4.2} and $u_k \in \mathcal{N}(A_k)$ we find
    \begin{align*}
        c_{H_{\min}}\norm{u_k}^2&\leq -\widetilde{g_k}^\top u_k - (1-\delta)\sigma_k\norm{u_k}^3 \nonumber \\
        &\leq -\widetilde{g_k}^\top u_k \nonumber \\
        &\leq \norm{P_k^\top g_k}\norm{u_k}+\kappa_H\norm{v_k}\norm{u_k} \nonumber \\
        &\leq \norm{u_k}\left(\norm{P_k^\top g_k}+\kappa_H\kappa_v \|c_k\|_1\right),
    \end{align*}
    and thus
    \begin{equation*}
        \|u_k\| \leq c_{H_{\min}}^{-1}\left(\norm{P_k^\top g_k}+\kappa_H\kappa_v \|c_k\|_1\right).
    \end{equation*}
    Therefore, it follows that there exists $K$ such that
    \begin{equation} \label{eq:dkquadconverge}
        \|d_k\| = \mathcal{O}(\|P_k^\top g_k\| + \|c_k\|) = \mathcal{O}(\|Z_k^\top g_k\| + \|c_k\|),
    \end{equation}
    for all $k \geq K$, where the last equality follows by Lemma \ref{lem:dualvariables}. For the rest of the proof, we assume that $k \geq K$.

Next, consider the system of nonlinear equations
    \begin{equation*}
        \left[
        \begin{matrix}
            \nabla_x \mathcal{L}(x, \lambda(x)) \\
            c(x)
        \end{matrix}
        \right] = 0,
    \end{equation*}
 where $\lambda(x) = -(A(x)A(x)^T)^{-1} A(x) \nabla f(x)$ and the Newton iteration on this system,\footnote{Modulo an error term of $\mathcal{O}(\|\nabla_x \mathcal{L}(x,\lambda(x))\|\|\bar{d}\|)$, which does not impact quadratic convergence.}
    \begin{equation} \label{eq:system1}
        \left[\begin{matrix} Z(\bar{x}_k)^{\top} \nabla_{xx}^2 \mathcal{L}(\bar{x}_k, \lambda(\bar{x}_k)) \\ A(\bar{x}_k) \end{matrix} \right] \bar{d}_k = -\left[\begin{matrix} Z(\bar{x}_k)^{\top} \nabla f(\bar{x}_k) \\ c(\bar{x}_k) \end{matrix} \right], \quad \quad \bar{x}_{k+1} = \bar{x}_k + \bar{d}_k,
    \end{equation}
where we utilized the equivalence between $\nabla_x \mathcal{L}(\bar{x}_k, \lambda(\bar{x}_k))$ and $Z(\bar{x}_k)^T \nabla f(\bar{x}_k)$ from Lemma \ref{lem:dualvariables}. Since $x_k \rightarrow x^*$, $A_k$ is nonsingular for all $k$ by Assumption \hyperlink{f2}{(F2)}, and $A(x)$ is Lipschitz continuous over $\mathcal{X}$ by Assumption \hyperlink{f1}{(F1)}, it follows that $A(x^*)$ has full rank. In addition, since $Z_*^{\top} \nabla^2_{xx} \mathcal{L}(x^*, \lambda^*) Z_*$ is positive definite, the Jacobian of this system is nonsingular at $x^*$, and thus there exists a ball around $x^*$ of nonzero radius such that the sequence $\{\bar{x}_k\}$ converges quadratically to $x^*$. Given this, we wish to show that the step $d_k$ generated by Algorithm \ref{alg:scp} satisfies
    \begin{equation} \label{eq:desiredsystem}
        \left[\begin{matrix} Z_k^{\top} \nabla_{xx}^2 \mathcal{L}(x_k, \lambda_k^*) \\ A_k \end{matrix} \right] d_k = -\left[\begin{matrix} Z_k^{\top} g_k \\ c_k \end{matrix} \right] + \left[\begin{matrix} \xi_g \\ \xi_c \end{matrix} \right],
    \end{equation}
where $\|\xi_g\|, \|\xi_c\| = \mathcal{O}(\|Z_k^{\top} g_k\|^2 + \|c_k\|^2)$. Under the assumption that such a relation holds, the claim follows directly by a simple modification of the standard inexact Newton argument (see, for example, \cite[Theorem 11.3]{nocedal2006numerical}).

Now, recall that $u_k \in \mathcal{N}(A_k)$. Therefore, there exists $p_k \in \mathbb{R}^{n-m}$ such that $u_k = Z_k p_k$ and
    \begin{equation}\label{eq:reducedmodel}
        p_k \approx \underset{p \in \mathbb{R}^{n-m}}{\argmin} \ f_k + (g_k + H_k v_k)^\top Z_k p + \frac12 p^\top Z_k^\top H_k Z_k p + \frac{\sigma_k}{3} \|p\|^3.
    \end{equation}
    Thus, it follows that the step $d_k = Z_k p_k + v_k$ satisfies the nonlinear system of equations
    \begin{equation} \label{eq:system2}
        \left[\begin{matrix} Z_k^{\top} H_k \\ A_k \end{matrix} \right] (Z_k p_k + v_k) = -\left[\begin{matrix} Z_k^{\top} g_k \\ c_k \end{matrix} \right] + \left[\begin{matrix}
            -\sigma_k \|p_k\|p_k + \hat{\xi}_{g} \\ \hat{\xi}_c
        \end{matrix}\right],
    \end{equation}
    where $\hat{\xi}_g$ and $\hat{\xi}_c$ are the residuals of these system satisfying $\|\hat{\xi}_g\| = \mathcal{O}(\|p_k\|^2)$ by \hyperlink{or2}{(OR2)} and $\|\hat{\xi}_c\| \leq r_v \|v_k\|^3$ by \eqref{eq:normalstep3}. Returning our attention to \eqref{eq:desiredsystem}, we see that $\xi_c = \hat{\xi}_c$, and thus by \eqref{eq:normalstep3}, Lemma \ref{lem:normalstepbound}, \hyperlink{f1}{(F1)}, and $\beta_k = 1$ for all $k$ sufficiently large, we get
    \begin{equation}
        \|\xi_c\| \leq r_v \|v_k^c\|^3 \leq r_v \kappa_{vs} \kappa_v \kappa_c \|c_k\|^2 = \mathcal{O}(\|c_k\|^2),
    \end{equation}
    so the second equation satisfies the desired bound.

    Focusing on the first equation of \eqref{eq:desiredsystem}, by \eqref{eq:system2}, we have that
    \begin{align*}
        Z_k^\top \nabla^2_{xx} \mathcal{L}(x_k,\lambda_k^*) (Z_k p_k + v_k) &= Z_k^\top (\nabla^2_{xx} \mathcal{L}(x_k,\lambda_k^*) - H_k) (Z_k p_k + v_k) + Z_k^\top H_k (Z_k p_k + v_k) \\
        &= - Z_k^\top g_k + Z_k^\top (\nabla^2_{xx} \mathcal{L}(x_k,\lambda_k^*) - H_k) (Z_k p_k + v_k) - \sigma_k \|p_k\| p_k + \hat{\xi}_g.
    \end{align*}
    Therefore, we simply need to bound the final three terms in the above equation. For the first term, recalling that $d_k = Z_k p_k + v_k$,
    \begin{align*}
        \norm{Z_k^\top (\nabla^2_{xx} \mathcal{L}(x_k,\lambda_k^*) - H_k) d_k} &= \left\|Z_k^\top \sum_{i=1}^m (\lambda_k^{*,i} - \lambda_k^i) \nabla^2 c_k^i d_k\right\| \\
        &\leq \norm{Z_k^\top} \|\lambda_k^* - \lambda_k\|_{\infty} \left(\sum_{i=1}^m \norm{\nabla^2 c_k^i}\right) \|d_k\| \\
        &\leq r_\lambda \kappa_{ch} \|v_k\| \|d_k\| \\
        &\leq r_\lambda \kappa_{ch} \|d_k\|^2 \\
        &= \mathcal{O}(\|Z_k^\top g_k\|^2 + \|c_k\|^2),
    \end{align*}
    where the final equation follows for all $k \geq K$ by \eqref{eq:dkquadconverge}. Next, by Lemma \ref{lem:sigmabound},
    \begin{equation*}
        \sigma_k \norm{p_k}^2 \leq \sigma_{\max} \|p_k\|^2 \leq \sigma_{\max}\|d_k\|^2 = \mathcal{O}(\|Z_k^\top g_k\|^2 + \|c_k\|^2).
    \end{equation*}
Finally, we recall that
    \begin{equation*}
        \|\hat{\xi}_g\| = \mathcal{O}(\|p_k\|^2) = \mathcal{O}(\|d_k\|^2) = \mathcal{O}(\|Z_k^\top g_k\|^2 + \|c_k\|^2).
    \end{equation*}
    Therefore, it follows that $d_k$ satisfies \eqref{eq:desiredsystem} with
    \begin{equation*}
        \|\xi_c\| = \mathcal{O}(\|Z_k^\top g_k\|^2 + \|c_k\|^2) \ \text{and} \ \|\xi_g\| = \mathcal{O}(\|Z_k^\top g_k\|^2 + \|c_k\|^2).
    \end{equation*}
    Finally, in the case where a second-order correction step is taken, by Lemma \ref{lem:wbound}, we have that
    \begin{equation*}
        \|w_k\| \leq \kappa_w \|d_k\|^2 = \mathcal{O}(\|Z_k^\top g_k\|^2 + \|c_k\|^2),
    \end{equation*}
    so there still exists $\xi_g = \mathcal{O}(\|Z_k^\top g_k\|^2 + \|c_k\|^2)$ and $\xi_c = \mathcal{O}(\|Z_k^\top g_k\|^2 + \|c_k\|^2)$ for which \eqref{eq:desiredsystem} is satisfied with $d_k + w_k$ in place of $d_k$. Hence, $\{x_k\}$ converges to $x^*$ Q-quadratically.
\end{proof}

\section{Discussion}\label{sec:sec6}

In this work, we introduced a cubic sequential programming method with the best known worst-case complexity guarantees for smooth, equality constrained optimization. Perhaps surprisingly, we showed that, under standard assumptions including LICQ, the optimal second-order theoretical guarantees for the unconstrained optimization setting unaffectedly carry over to the equality constrained one: An optimizer interested in minimizing a smooth, non-convex function should not worry about an unprecedented increase of the worst-case number of iterations needed for convergence to a local minimum when (smooth, non-convex) equality constraints are also considered, unless feasibility is its primary goal. Specifically, the algorithm developed converges to an approximate second-order critical point in at most $\mathcal{O}\left(\max\left\{\epsilon_g^{-3/2},\epsilon_c^{-1},\epsilon_H^{-3}\right\}\right)$ iterations. The bounds with respect to the gradient and Hessian of the Lagrangian are known to be tight, as they match those of the unconstrained setting, whereas the bound that corresponds to the constraint violation matches the best reported baseline for this class of problems. Further, we showed that our algorithm enjoys local quadratic convergence properties, a property that has not been established for other second-order methods with complexity guarantees for equality constrained optimization. \par 

Our work raises a number of questions and open problems. Now that we have established the fact that second-order theoretical guarantees for the unconstrained and equality constrained settings are the same, it is natural to ask whether similar guarantees hold for inequality constrained problems. Second, although the complexity bound $\mathcal{O}\left( \max\left\{\epsilon_g^{-3/2},\epsilon_H^{-3}\right\} \right)$ is sharp \cite{cartis2012complexity}, the same is not currently known for the constraint violation bound $\mathcal{O}(\epsilon^{-1}_c)$. Although this is identical to the best-known bound for this class of problems \cite{berahas2025sequential, curtis2024worst}, there are no obvious indications that one cannot do better when it comes to ensuring (approximate) feasibility. Indeed, given the everywhere LICQ assumption, if one ignores optimality all together and focuses entirely on feasibility, it is easy to obtain linear convergence in the constraint violation. Thus, if the bound for Algorithm \ref{alg:scp} is tight for this class of problems, it is due to the conflicting nature of achieving optimality and feasibility simultaneously. Therefore, a natural open question is whether one can generate matching lower bounds, or if there are algorithms that can achieve faster convergence with respect to the constraint violation while maintaining the complexity with respect to the gradient and Hessian of the Lagrangian.\par

Despite the fact that our algorithm uses tools from standard SQP theory, it still requires solving an unconstrained cubic subproblem, which can be a computationally expensive task \cite{carmon2018analysis}. Incorporating the insights developed here into more standard SQP frameworks, such as line search or trust region based methods, may be a fruitful direction of future work. Lastly, for the main results we required a ``linearly independence" constraint qualification, which is oftentimes characterized as a ``strong" assumption in the relevant literature of non-convex optimization: Most two-phase methods developed for equality constrained optimization problems make no such assumption (albeit, with respect to a different constraint violation measure; see Appendix \ref{sec:AppA}). However, they fall short of ensuring satisfactory constraint violation bounds. Indeed, the best known iteration threshold with respect to approximate feasibility is equal to $\mathcal{O}(\epsilon_c^{-3/2})$ \cite{cartis2013evaluation, curtis2018complexity}. It is therefore natural to examine whether the bounds established in this work - especially the ones with respect to the constraints - can be supported under weaker assumptions.

\bibliography{references}

\begin{appendices}

\section{Extended Related work}\label{sec:AppA}

In equality constrained optimization, the landscape of theoretical results is shaped as much by the adopted optimality criteria as by the algorithm itself. The convergence guarantees that appear in the literature are always stated relative to a particular class of critical points, chosen according to the prevailing constraint qualifications and regularity assumptions. As of now, there is no universal notion of optimality. For well-posed smooth problems satisfying constraint qualifications, such as LICQ, the natural target points are framed with respect to the Lagrangian function, and are usually represented by \eqref{eq:fosp}-\eqref{eq:sosp}, the latter of which have been adopted here. The state-of-the-art SQP methods as well as a number of augmented Lagrangian methods developed for problems of the form \eqref{eq:1} are also designed around these optimality criteria; see Table \ref{tab:table}. However, a number of ``two-phase" works are concerned with different notions of stationary minimizers, which can be understood as scaled KKT points. In addition, some SQP and augmented Lagrangian methods for equality constrained optimization have been developed over manifolds due to their attractive properties in modern applications; see, for instance, \cite{bai2018analysis, goyens2024computing}. In such spaces the notion of criticality is more generic than that of \eqref{eq:fosp}-\eqref{eq:sosp}. In this section we expand on single-phase methods over manifolds and on two-phase methods, by briefly studying prior works and comparing their optimality termination criteria with the ones considered in this paper.

\subsection{Two-phase methods}

One of the most standard approaches in non-convex constrained optimization are two-phase methods. A number of variations of such methods have been developed by Cartis, Gould and Toint in a series of sequent works \cite{cartis2011evaluation, cartis2012complexity, cartis2013evaluation, cartis2014complexity, cartis2019optimality} for various constrained settings, including the equality one \eqref{eq:1}. In its first phase, a two-phase algorithm is searching for an approximate feasible solution. Namely, given a primal tolerance $\epsilon_P>0$, the algorithm is seeking for an initial point $x_1$ such that $\norm{c(x_1)}\leq \epsilon_P$. This process is usually achieved by solving an unconstrained least-squares problem. Unless approximate infeasibility is deduced, in its second phase, the algorithm is decreasing the objective value while maintaining approximate feasibility in every iteration. The termination criteria proposed in two-phase methods are different to \eqref{eq:fosp}-\eqref{eq:sosp}. The latter target points only make sense when the Jacobian of the constraints is full rank over the space of iterates and all of its singular values are bounded away from zero, i.e., when LICQ holds. Our assumption \hyperlink{f2}{(F2)} is, however, not considered in most two-phase methods.\par
The termination criterion of Cartis, Gould and Toint \cite{cartis2013evaluation} differentiates between cases of a zero/non-zero residual of criticality. Rigorously speaking, given primal-dual tolerances $\epsilon_P,\epsilon_D\in(0,1)$, their two-phase algorithm terminates when
\begin{equation}\label{eq:a1}
    \norm{r(x,t)}\leq \epsilon_P\quad\text{or}\quad \norm{g_r(x,t)}\leq \epsilon_D,
\end{equation}
where
\begin{equation}
    r(x,t)=\begin{pmatrix}
c(x) \\
f(x)-t
\end{pmatrix}
\quad\text{and}\quad g_r(x,t)=\frac{\nabla_x r(x,t)^\top r(x,t)}{\norm{r(x,t)}}\mathbf{1}_{\{r(x,t)\neq 0\}}.
\end{equation} 
Here, $t$ can be understood as the ``target" value of the objective. In fact, the two-phase methods of Table \ref{tab:table} are mainly concerned with just the second bound of \eqref{eq:a1}. This is because the second phase - in combination with the initialization property $\norm{c(x_1)}\leq \epsilon_P$ of the first phase - implicitly guarantees that all iterates are approximately feasible, namely $\norm{c(x_k)}\leq \epsilon_P$ for every $k\geq 1$. The criticality condition $\norm{g_r(x,t)}\leq \epsilon_D$ essentially corresponds to a scaled KKT point; the scaling is considered so that it takes the size of the Lagrangian multipliers into account.\par
Both the works of Cartis, Gound and Toint \cite{cartis2013evaluation} and Curtis, Robinson and Samadi \cite{curtis2018complexity} (which also considers the same stationarity criterion) achieve complexity bounds of order $\mathcal{O}\left(\epsilon_P^{-1/2}\epsilon_D^{-3/2}\right)$, where $\epsilon_D\leq \epsilon_P^{1/3}$.  When $\epsilon_P=\epsilon_D$, this bound boils down to $\mathcal{O}\left(\epsilon^{-2}\right)$, which matches that of gradient descent for unconstrained optimization. On the other hand, if one takes $\epsilon_D=\epsilon_P^{2/3}$, then the worst case number of iterations reduces to  $\mathcal{O}\left(\epsilon^{-3/2}\right)$, which is known to be optimal. Considering different primal-dual tolerances $\epsilon_P,\epsilon_D$ is reasonable if one takes into account the different scalings of the criticality residuals and the dual gradients, an event plausible when LICQ fails.\par
In the first phase of the algorithm in \cite{cartis2013evaluation}, the least squares problem $\frac{1}{2}\norm{r(x,t)}^2$ is solved. Both phases of the algorithm suggested incorporate a cubic regularization method. The motivation behind this choice lies on the success of cubic methods in proving optimal complexity bounds for the unconstrained setting \cite{cartis2011adaptive, nesterov2006cubic}. Instead, the work of Curtis et al. \cite{curtis2018complexity} resorts to a trust funnel method, inspired by previous trust-region methods \cite{curtis2017trust}. Importantly, while it achieves similar theoretical guarantees, it improves upon a practical disadvantage of \cite{cartis2013evaluation}: The first phase of the trust-funnel algorithm, besides effectively seeking an initial approximate feasible solution, also reduces the objective value.\par

Aiming to extend the theoretical guarantees of their classical two-phase approach to second and third-order critical points, Cartis et al. \cite{cartis2019optimality} considered the equality constrained problem \eqref{eq:1} under the additional constraint ``$x\in\mathcal{F}$", where $\mathcal{F}$ is a non-empty closed convex set. The termination criterion considered resembles \eqref{eq:a1} but is appropriately modified so that it targets stationary points of higher order: An approximate q-order critical point is defined as satisfying
\begin{equation}
    \phi_{\mu,j}^{\Delta}(x,t)\leq \epsilon_D \Delta^{j}\norm{r(x,t)}\quad\text{for every\;}j\in\{1,2,...,q\},
\end{equation}
where
\begin{equation}\label{eq:a4}
    \phi_{\mu,j}^{\Delta}(x,t) := \frac{1}{2}\norm{r(x,t)}^2 - 
\min_{\substack{d \in \mathcal{F} \\ \|d\| \le \Delta}} T_{\mu,j}(x,d)
\end{equation}
is the largest feasible decrease of the jth order Taylor model achieved over a trust-region in $\mathcal{F}$.
For $q=1$, when second-order derivative information is available, the complexity bounds match those of \cite{cartis2013evaluation, curtis2018complexity}. Whereas, for $q=2$ and the special case $\epsilon_P=\epsilon_D=\epsilon$, the bound of Table \ref{tab:table} boils down to $\mathcal{O}(\epsilon^{-5})$ with the use of a trust-region inner solver. One downside of this approach is that each subproblem must respect the constraint $x \in \mathcal{F}$ exactly, which, even for simple sets such as $x \geq 0$, may result in NP-hard subproblems.\par

Further, Cartis et al. \cite{cartis2014complexity} treated the more generic inequality constrained optimization problem. With the introduction of a first-order short-step homotopy algorithm, they managed to show a complexity bound no worse than the one provided by steepest descent. The termination criterion implemented in their two-phase method is different to the previous ones. In particular, their algorithm terminates when $\chi(x,t)\leq \epsilon$, where
\begin{equation}\label{eq:a5}
    \chi(x,t)=l_{\varphi}(x,t,0)-\min_{\norm{d}\leq 1}l_{\varphi}(x,t,d)
\end{equation}
for $l_{\varphi}(x,t,d):=\norm{c(x)+A(x) d}+\left| f(x)+g(x)^\top d-t \right|$. As before, $t$ here can be thought of as a ``target value" for the objective function. As a justification for this criterion, the equivalence ``$\chi(x,t)=0$ if and only if $x$ is a first-order stationary point" is shown by the authors. \par
Birgin et al. \cite{birgin2016evaluation} also considered the inequality constrained setting. Their two-phase feasibility and target-following algorithm consists of three stopping criteria in vein similar to the aforementioned \eqref{eq:a1}-\eqref{eq:a4}, but without a scaling characterization. More specifically, if the algorithm stops at the first criterion, then approximate infeasibility is deduced. Their second criterion gives an (unscaled) $\epsilon_D$-KKT point. The last termination rule can be interpreted as reaching an $\epsilon_P$-feasible point where some standard regularity condition fails, under mild tolerance conditions. The authors conclude that their suggested method finds an unscaled approximate KKT point under suitable ``nondegeneracy" assumptions, in contrast to the scaled KKT points considered in the preceding works. When information of the first-order derivatives of the problem's functions is available, they can achieve an iteration bound of $\mathcal{O}(\epsilon^{-3})$ to first order points, where $\epsilon:=\epsilon_P=\epsilon_D$. For second order derivatives, the bound that appears in Table \ref{tab:table} equals $\mathcal{O}(\epsilon^{-2})$, matching that of the previous works for the case $\epsilon_P=\epsilon_D$.\par
Finally, Cartis et al. \cite{cartis2011evaluation} targeted the unconstrained optimization problem $\min f(x)+h(c(x))$, where $h$ is a composite function. The equality constrained optimization problem \eqref{eq:1} then becomes a special case for $h(c(x)):=\rho \norm{c(x)}$ for some penalty parameter $\rho>0$. Although their algorithm consists of just one phase, their termination criterion is inspired by the ones of the above two-phase methods (mainly \eqref{eq:a5}). More rigorously, their exact penalty function algorithm terminates when a point $x$ with $\Psi(x)\leq \epsilon$ is encountered, where
\begin{equation}
    \Psi(x)=l(x,0)-\min_{\norm{d}\leq 1}l(x,d)
\end{equation}
for $l(x,d):=f(x)+g(x)^\top d+\rho \norm{c(x)+A(x)d}$. Their worst-case iteration complexity guarantees range from $\mathcal{O}(\epsilon^{-2})$ to $\mathcal{O}(\epsilon^{-5})$, depending on the boundedness of the penalty parameter.

\subsection{Augmented Lagrangian and SQP methods over manifolds}

Some of the most attractive single-phase approaches for the problem \eqref{eq:1} are those making use of the augmented Lagrangian function
\begin{equation}\label{eq:auglag}
    \mathcal{L}_{\beta}(x,\lambda):=f(x)+\lambda^\top c(x) +\frac{\beta}{2}\norm{c(x)}^2,
\end{equation}
for some parameter $\beta\geq 0$. Augmented Lagrangian methods have been a popular choice over the last few years for tackling equality constrained optimization problems \cite{grapiglia2021complexity, he2023newton, xie2021complexity}. Despite their practical performance, they lack good complexity guarantees in comparison to SQP and two-phase methods, as one can readily verify via inspecting Table \ref{tab:table}.\par

The work of Goyens, Eftekhari and Boumal \cite{goyens2024computing} represents an exception to the aforementioned works. They consider \eqref{eq:1} with the generic constraint ``$x\in\mathcal{M}$", where $\mathcal{M}$ represents a manifold. Their method combines gradient and eigensteps applied to Fletcher's augmented Lagrangian function, the latter of which is known for being computationally expensive to compute. Although their complexity bound $\mathcal{O}\big(\max\left\{\epsilon_g^{-2},\epsilon_c^{-2},\epsilon_H^{-3}\right\}\big)$ to second-order stationarity is weaker than the one established here, their target points considered are more general than \eqref{eq:fosp}-\eqref{eq:sosp}, since the gradients and Hessians are defined with respect to the tangent space of a layer manifold $\mathcal{M}_x$. Rigorously speaking, a point $x$ is a $(\epsilon_c,\epsilon_g,\epsilon_H)$-approximate second-order point if
\begin{equation}\label{eq:manifolds}
    \norm{c(x)}\leq\epsilon_c,\quad \norm{\text{grad}_{\mathcal{M}_x}f(x)}\leq \epsilon_g,\quad\text{and}\quad \text{Hess}_{\mathcal{M}_x}f(x)\succeq- \epsilon_H\text{Id}.
\end{equation}
Their criticality conditions have a natural geometric interpretation and can be thought of as extensions of Riemannian optimality conditions to points outside the feasible manifold $\mathcal{M}$. Importantly, they show that \eqref{eq:manifolds} always implies \eqref{eq:fosp}-\eqref{eq:sosp} when all tolerances are the same and equal to $\epsilon$, but the converse is not true, in general. In similar fashion to other related works \cite{grapiglia2021complexity, he2023newton, xie2021complexity}, the method of \cite{goyens2024computing} strongly relies on an initialization condition.\par

Aside from augmented Lagrangian approaches, SQP methods have also been adopted in this manifold setting due to their remarkable practical proficiency in constrained optimization. Bai and Mei \cite{bai2018analysis} designed a first-order sequential quadratic programming method that is more suitable than Riemannian optimization methods in the undesirable scenario where the retraction onto the constraint set is an intractable problem. Although Riemannian gradient schemes are usually preferred for this class of problems, these gradient steps are nearly identical to those of the proposed SQP algorithm. It is shown that the suggested algorithm converges to an approximate first-order stationary point in at most $\mathcal{O}(\epsilon^{-4})$ steps, under similar assumptions to the ones considered here. It should be mentioned that the target points considered in \cite{bai2018analysis} match those of \cite{goyens2024computing}, i.e., \eqref{eq:manifolds}, and are thus different to the ones of the current work. Lastly, the authors state explicit local linear rates of order $(1-1/\kappa_R)^k$, where $\kappa_R$ represents the condition number of the Riemannian Hessian at the optimal solution.

\section{Technical Lemmas}\label{sec:AppB}

We begin this section with some standard bounds on $c$ under varying assumptions on its derivatives.

\begin{lemma} \label{lem:lipcon}
    Suppose \hyperlink{f1}{(F1)} holds. Then,
    \begin{equation} \label{eq:lipcongrad}
        \norm{c(x_k + d_k)}_1 \leq \norm{c_k + A_k d_k}_1 + \frac{L_A}{2} \norm{d_k}^2.
    \end{equation}

    If, in addition, assumption \hyperlink{f5}{(F5)} holds, then
    \begin{equation}    \label{eq:lipconhess}
        \norm{c(x_k + d_k)}_1 \leq \sum_{i=1}^m |c^i_k + \nabla c^i_k d_k + \frac{1}{2} d_k^{\top} \nabla^2 c^i_k d_k| + \frac{L_{ch}}{6} \norm{d_k}^3.
    \end{equation}
\end{lemma}

\begin{proof}
    By Lipschitz continuity of the gradient of each $\nabla c^i$,
    \begin{equation*}
        |c^i(x_k+d_k)| - |c^i_k + \nabla c^i_k d_k| \leq |c^i(x_k+d_k) - c^i_k - \nabla c^i_k d_k| \leq \frac{L_{A_i}}{2}\norm{d_k}^2_2.
    \end{equation*}
    Using the definition of $L_A$, rearranging and summing this inequality for all $i=1,\dots,m$ gives the first result.

    The second result follows by a similar argument when Lipschitz continuity of the Hessian holds for each $c^i$. 
\end{proof}

Next, we provide a proof for Lemma \ref{lem:dualvariables}.

\begin{proof}
    Due to LICQ, $P_k=Z_k Z_k^\top$ is the orthogonal projection matrix onto $\mathcal{N}(A_k)$, where $Z_k$ is an orthonormal basis of the latter. Moreover, let $D_k:=A_k^\top (A_k A_k^\top)^{-1} A_k$ be the orthogonal projection matrix onto $\mathcal{R}(A_k^\top)$. We then have $I=D_k+P_k$. By definition of the Lagrangian, for the true least-squares estimators we find
    \begin{equation}\label{eq:b3}
        \nabla_x\mathcal{L}(x_k,\lambda_k^*)=g_k+A_k^\top \lambda_k^*=(I-D_k)g_k=Z_k Z_k^\top g_k=P_k^\top g_k.
    \end{equation}
Part (1) follows directly from the orthonormal properties of $Z_k$.\par

Let $Q_k:=[Z_k\;Y_k]\in\mathbb{R}^{n\times n}$ represent an orthogonal extension of $Z_k$ into a basis, where $Y_k\in\mathbb{R}^{n\times m}$ is an orthonormal basis for $\mathcal{N}(A_k)^{\perp}$. Then, we have
\begin{equation}\label{eq:b4}
    Q_k^\top P_k Q_k=\begin{bmatrix}
I_{n-m} & 0 \\
0 & 0
\end{bmatrix}.
\end{equation}
The orthogonality of $Q_k$, the property $P_k=P_k^\top$ and \eqref{eq:b4} yield
\begin{equation}
    Q_k^\top (P_k^\top H_k P_k) Q_k=(Q_k^\top P_k Q_k)(Q_k^\top H_k Q_k)(Q_k^\top P_k Q_k)=\begin{bmatrix}
Z_k^\top H_k Z_k & 0 \\
0 & 0
\end{bmatrix}. 
\end{equation}
The fact that any two orthogonally similar matrices have the same eigenvalues gives (2).

For the third part of the auxiliary lemma, we first show that the orthogonal projection matrix $D(\cdot)$ of $\mathcal{R}(A(\cdot)^\top)$ is Lipschitz continuous. To do this, we first show that the operator $B(\cdot):=(A(\cdot)A(\cdot)^\top )^{-1}$ is Lipschitz continuous.  To this end, let points $x,y$. By a standard inverse identity %(\nd{special case of the Woodbury formula})
and the Cauchy-Schwarz inequality we get
\begin{equation}\label{eq:b6}
    \norm{B(x)-B(y)}\leq \norm{B(x)}\norm{B(x)^{-1}-B(y)^{-1}}\norm{B(y)}.
\end{equation}
Note that $B$ is uniformly bounded due to \hyperlink{f2}{(F2)}. By adding and subtracting $A(x)A(y)^\top$ inside the middle norm of (\ref{eq:b6}) and using the Cauchy-Schwarz inequality as well as \hyperlink{f2}{(F2)}, we deduce the existence of some constant $C_A$ such that
\begin{equation}
    \norm{B(x)-B(y)}\leq C_A \norm{A(x)-A(y)}.
\end{equation}
Lipschitz continuity of the Jacobian yields Lipschitz continuity of $B$. Now, obserse that, for points $x,y$, we have $D(x)-D(y)=A(x)^\top B(x)A(x)-A(y)^\top B(y) A(y)$, which can be equivalently written as
\begin{equation}\label{eq:b8}
    (A(x)^\top - A(y)^\top) B(x) A(x)
 + A(y)^\top (B(x) - B(y)) A(x)
 + A(y)^\top B(y) (A(x) - A(y)).
\end{equation}
By Lipschitz continuity of $A$ and $B$, and the triangle and Cauchy-Schwarz inequalities applied to (\ref{eq:b8}), we can finally show that $D(\cdot)$ is Lipschitz continuous over $\mathcal{X}$. Hence, $P(\cdot)$ is also Lipschitz continuous. Let $L_P$ be its Lipschitz constant. Then, for $x,y\in \mathcal{X}$,
\begin{align*}
    \norm{P(x)^\top g(x)-P(y)^\top g(y)}&\leq \norm{P(x)}\norm{g(x)-g(y)}+\norm{g(y)}\norm{P(x)-P(y)}\\
    &\leq (L_g+\kappa_g L_P)\norm{x-y}.
\end{align*}
Lipschitz continuity of $P(\cdot)^\top g(\cdot)$ follows.
\end{proof}

\section{Inexact Steps}\label{sec:AppC}

In this section we utilize the standard LSQR and Lanczos iterative methods to show that inexactness of both the normal and tangential steps can be achieved in finitely many iterations, and we explicitly establish the number of repetitions needed for termination.

\subsection{Normal Step Computation}

To approximately solve the least squares problem of \eqref{eq:normalstep3}, we apply the classical LSQR method developed by Paige and Saunders \cite{paige1982lsqr}, which is particularly fitting for large-scale optimization problems. In exact arithmetic, this iterative process is equivalent to the standard Conjugate Gradient method applied to the linear system $A_k^\top A_k v_k=-A_k^\top c_k$. This allows us to use classical results to determine the minimum number of iterations needed to satisfy the residual bound of \ref{eq:normalstep3}, as the next lemma shows.

\begin{lemma}\label{lem:lsqr}
   Suppose (F2) holds. Fix an iteration $k\geq 0$ and suppose that $c_k\neq 0$ and $\kappa_{A_k}>1$, where $\kappa_{A_k}$ represents the condition number of the Jacobian $A_k$. Further, suppose we run the classical LSQR algorithm in the linear system $A_k v^c+c_k=0$ with starting iterate $v_{k,0}^c:=0$. Then, LSQR will terminate with an iterate $v^c_{k,j}$ satisfying $\norm{A_k v_{k,j}^c+c_k}_1\leq r_v \norm{c_k}_1$ in at most
    \begin{equation*}
        M_{k}:=\min\left\{ m,\ceil[\Bigg]{\frac{\log \left(\frac{2\sqrt{m}}{r_v}\right)}{\log \left(\frac{\kappa_{A_k}+1}{\kappa_{A_k}-1}\right)}} \right\}
    \end{equation*}
iterations. In addition, LSQR will terminate with an iterate $v^c_{k,j}$ satisfying \eqref{eq:normalstep3} in at most
\begin{equation*}
        M_{k}^*:=\min\left\{ m,\ceil[\Bigg]{\frac{\log \left(\frac{2\norm{c_k}}{\chi_k}\right)}{\log \left(\frac{\kappa_{A_k}+1}{\kappa_{A_k}-1}\right)}} \right\}
    \end{equation*}
    iterations, where $\chi_k:=\min\left\{\frac{\norm{c_k}}{2},\frac{r_v \norm{c_k}_1}{\sqrt{m}},\frac{r_v}{\sqrt{m}}\left(\frac{\norm{c_k}}{2\norm{A_k}}\right)^3\right\}$.
\end{lemma}

\begin{proof}

Let $v_{k,j}^{c}\in \mathcal{R}(A_k^\top)$ be the $j$-th inner iterate of the LSQR algorithm, with $v_{k,0}^c:=0$. As already mentioned above, in exact arithmetic, LSQR is precisely equivalent to applying the classical CG method to the system $A_k^\top A_k v=-A_k^\top c_k$ \cite{paige1982lsqr}. After restricting the operator $A_k^\top A_k$ to the range space of $A_k^\top$, a standard result \cite[p. 117]{nocedal2006numerical} gives
\begin{equation}\label{eq:c1}
    \norm{v_{k,j}^c-v_k^*}_{A_k^\top A_k}\leq 2\left(\frac{\kappa_{A_k}-1}{\kappa_{A_k}+1}\right)^j\norm{v_{k,0}^c-v_k^*}_{A_k^\top A_k},
\end{equation}
where $v_k^*$ is given by \eqref{eq:normalstep2} due to (F2). By recalling that $\norm{v}_{A_k^\top A_k}=\norm{A_k v}$ for every $v\in \mathcal{R}(A_k^\top)$ and using \eqref{eq:normalstep2}, the relation \eqref{eq:c1} becomes
\begin{equation}\label{eq:c2}
    \norm{A_k v_{k,j}^c+c_k}\leq 2\left(\frac{\kappa_{A_k}-1}{\kappa_{A_k}+1}\right)^j \norm{c_k}.
\end{equation}
By equivalence of norms, \eqref{eq:c2} implies
\begin{equation}\label{eq:c3}
    \norm{A_k v_{k,j}^c+c_k}_1\leq 2\sqrt{m}\left(\frac{\kappa_{A_k}-1}{\kappa_{A_k}+1}\right)^j \norm{c_k}_1.
\end{equation}
Therefore, the iterate $v_{k,j}^c$ will always satisfy the first bound for $j\geq \log \left(\frac{2\sqrt{m}}{r_v}\right)/\log \left(\frac{\kappa_{A_k}+1}{\kappa_{A_k}-1}\right)$. Since LSQR terminates in at most $r(A_k)=m$ iterations in the worst-case scenario, the first part of the Lemma follows.\par 
Moving onto the second part, let $v_{k,j}^c$ be an iterative that satisfies $\norm{A_k v_{k,j}^c+c_k}\leq \chi_k$. The reverse triangle inequality and the Cauchy-Schwarz inequality yield
\begin{equation}\label{eq:c4}
    \norm{v_{k,j}^c}\norm{A_k}\geq \norm{A_k v_{k,j}^c}=\norm{A_k v_{k,j}^c-c_k +c_k}\geq \norm{c_k}-\chi_k\geq \norm{c_k}/2.
\end{equation}
Further, norm equivalence combined with the definition of $\chi_k$ and \eqref{eq:c4} give
\begin{equation}\label{eq:c5}
    \norm{A_k v_{k,j}^c+c_k}_1\leq \sqrt{m}\chi_k\leq r_v \left(\frac{\norm{c_k}}{2\norm{A_k}}\right)^3 \leq r_v \norm{v_{k,j}^c}^3.
\end{equation}
The second term of definition of $\chi_k$ and \eqref{eq:c5} finally give
\begin{equation}
    \norm{A_k v_{k,j}^c+c_k}_1\leq \sqrt{m}\chi_k\leq r_v\min\left\{\norm{c_k}_1,\norm{v_{k,j}^c}^3 \right\}.
\end{equation}
Standard CG will take at most $m$ iterations to terminate with an optimal solution, so the second bound directly follows. 
\end{proof}

\subsection{Tangential Step Computation}

We now show that the standard Lanczos method, which is commonly used to solve cubic regularization problems \cite{bellavia2025regularized, carmon2018analysis, cartis2011adaptive, jia2022solving}, can generate a step $u_k\in \mathcal{N}(A_k)$ satisfying (OR1)-(OR3). Suppose $\widetilde{g}_k\neq 0$ and consider the $j$-th Krylov space $\mathcal{K}_{k,j}:=\mathcal{K}_{j}(\widetilde{H}_k,\widetilde{g}_k)=\text{span}(\widetilde{g}_k,\widetilde{H}_k \widetilde{g}_k,...,\widetilde{H}_k^{j-1}\widetilde{g}_k)$ for $j\geq 1$, assuming that the outer iteration $k\geq 0$ is fixed. Consider a Lanczos orthogonal basis $Q_{k,j}:=(q_{k,1},...,q_{k,j})$, where $q_{k,1}:=\widetilde{g}_k/\norm{\widetilde{g}_k}$. Further, we define the tridiagonal matrices $B_{k,j}:=Q_{k,j}^\top \widetilde{H}_k Q_{k,j}$. At every iteration $j$, the Lanczos method considers the cubic subproblem
\begin{equation}\label{eq:c7}
    \min_{y\in \mathbb{R}^j}\norm{\widetilde{g}_k} e_1^\top y+\frac{1}{2}y^\top B_{k,j} y+\frac{\sigma_k}{3}\norm{y}^3
\end{equation}
and then sets $u_{k,j}:=Q_{k,j}y_{k,j}^*$, where $y_{k,j}$ represents the optimal solution of \eqref{eq:c7} (see \cite{jia2022solving} relations (4.1)-(4.4)). This can equivalently be written and efficiently solved as an eigenvalue subproblem \cite[Theorem 3.2]{jia2022solving}. Define the ratios
\begin{equation*}
    \nu_{k,j}:=\frac{\lambda_{\max}(B_{k,j})+\sigma_k\norm{u_{k,j}}}{\lambda_{\min}(B_{k,j})+\sigma_k\norm{u_{k,j}}}.
\end{equation*}
These are well-defined by \cite[Lemma 4.1]{jia2022solving}, under the assumption that $\widetilde{g}_k\neq 0$ and $\widetilde{g}_k \not \perp \mathcal{E}_1^{(k)}$, where $\mathcal{E}_1^{(k)}$ is defined as the eigenspace of $\widetilde{H}_k$ associated with $\lambda_{\min}(\widetilde{H}_k)$ (see also \cite[Lemma 4]{cartis2020monotonic}). In turn, let
\begin{equation*}
    \nu_k:=\max_{j} \nu_{k,j}\quad \text{and}\quad \xi_k:=\frac{\sqrt{\nu_k}+1}{\sqrt{\nu_k}-1}.
\end{equation*}
Given the global minimizer of \eqref{eq:mku} $u_k^*\in \mathcal{R}(P_k)$, we also define the quantities
\begin{equation*}
    \nu_k^*:=\frac{\lambda_{\max}(\widetilde{H}_k)+\sigma_k \norm{u_k^*}}{\lambda_{\min}(\widetilde{H}_k)+\sigma_k \norm{u_k^*}}\quad \text{and}\quad \xi_k^*:=\frac{\sqrt{\nu_k^*}+1}{\sqrt{\nu_k^*}-1}.
\end{equation*}
The first ratio is also well-defined under the above eigenspace ``easy case" assumption \cite[Theorem 4.4]{jia2022solving}. Further, we assume without loss of generality that both $\xi_k$ and $\xi_k^*$ are well-defined; otherwise the Lanczos method needs only one iteration to output an optimal solution.

Now that the required notation has been introduced, we are in a position to establish the minimum number of iterations needed to output an inexact tangential step, under the previous geometric condition that avoids the ``hard case", the latter of which can be present in large-scale problems and ill-conditioned settings.

\begin{proposition}\label{prop:lanczos}
Suppose (F2) holds and fix an iteration $k\geq 0$. Moreover, suppose that $\widetilde{g}_k\neq 0$ and either $\lambda_{\min}(\widetilde{H}_k)\geq 0$ or $\widetilde{g}_k \not \perp \mathcal{E}_1^{(k)}$. Then, the Lanczos method will output an iterate $u_{k,l}\in \mathcal{R}(P_k)$ satisfying (OR1)-(OR3), in at most
    \begin{equation*}
        N_k^{OR}:= \min \left\{ n-m,\max\left\{N_k^{OR2},N_k^{OR3}\right\} \right\}
    \end{equation*}
total iterations, where
\begin{equation*}
    N_k^{OR2}:=1+\ceil[\Bigg]{\max\left\{0,\log\left(\frac{2\norm{H_k}\nu_k}{\delta \sigma_k \norm{u_k^C}}\right)\right\}/\log(\xi_k)}
\end{equation*}
and
\begin{equation*}
    N_k^{OR3}:=1+\ceil[\Bigg]{\max\left\{0,\log\left(\frac{\sigma_k \norm{\widetilde{g}_k}(\nu_k^*-1)}{\sqrt{\nu_k^*}(\lambda_{\min}(\widetilde{H}_k)+\sigma_k\norm{u_k^*})^2}\right)\right\}/\log(\xi_k^*)}.
\end{equation*}
\end{proposition}

\begin{proof}
    First, it is easy to see that every inner iterate $u_{k,j}$ of the Lanczos method satisfies (OR1): Since $\mathcal{K}_{k,1}\subseteq \mathcal{K}_{k,j}$ and $u_{k,j}$ represents the global minimizer of $m_k$ over $\mathcal{K}_{k,j}$ for every $j\geq 1$, we have $m_k^U(u_{k,j})\leq m_k^U(u_{k,1})=m_k^U(u_k^C)$. The latter equality follows from $\mathcal{K}_{k,1}=\text{span}\{\widetilde{g}_k\}$.\par
    For (OR2), fix an inner iteration $j\geq 1$, and let $u_{k,j}$ be the optimal solution of $\min_{u\in \mathcal{K}_{k,j}} m_k^U(u)$. By invoking \cite[Corollary 3.7]{gould2020error} for $p=3$ and the inequality $\norm{\widetilde{H}_k}\leq \norm{H_k}$, we get
    \begin{align}
        \norm{\nabla m_k^U(u_{k,j})}&\leq 2\nu_{k,j}\norm{H_k}\norm{u_{k,j}}\left(\frac{\sqrt{\nu_{k,j}}-1}{\sqrt{\nu_{k,j}}+1}\right)^{j-1} \nonumber \\
        &\leq 2\nu_{k}\norm{H_k}\norm{u_{k,j}}\xi_k^{-(j-1)}. \label{eq:c8}
    \end{align}
Suppose that $j\geq N_k^{OR2}$. Then, by definition of $N_k^{OR2}$, we obtain
\begin{equation}\label{eq:c9}
    2\norm{H_k}\nu_k\xi_k^{-(j-1)}\leq \delta \sigma_k \norm{u_k^C}.
\end{equation}
In addition, assume that $j\leq \tau_k$, where $\tau_k:=\text{grade}(\widetilde{H}_k,\widetilde{g}_k)$ is the grade of $\widetilde{g}_k$ with respsect to $\widetilde{H}_k$. From the monotonicity result \cite[Theorem 1]{cartis2020monotonic} we get $\norm{u_k^C}=\norm{u_{k,1}}\leq \norm{u_{k,j}}$. Combining this with \eqref{eq:c8}-\eqref{eq:c9}, we acquire (OR2).\par
For the third property of the Oracle, we may assume, without loss of generality, that the minimum eigenvalue of $\widetilde{H}_k$ is negative (otherwise (OR3) holds automatically). Thus, we make use of the assumption $\widetilde{g}_k \not \perp \mathcal{E}_1^{(k)}$. For $1\leq j \leq \tau_k$, \cite[Theorem 4.4]{jia2022solving} yields the bound
\begin{equation}\label{eq:c10}
    \norm{u_{k,j}-u_k^*}\leq 4\norm{\widetilde{g}_k}\sqrt{\nu_k^*}\frac{\left(\zeta_k+\sqrt{\zeta_k^2-1}\right)^{-j}}{(\lambda_{\max}(\widetilde{H}_k)-\lambda_{\min}(\widetilde{H}_k))(\zeta_k^2-1)},
\end{equation}
where $\zeta_k:=(\nu_k^*+1)/(\nu_k^*-1)$. The relations $\zeta_k+\sqrt{\zeta_k^2-1}=\xi_k^*$ and $\zeta_k^2-1=4\nu_k^*/(\nu_k^*-1)^2$ imply the following equivalent form of \eqref{eq:c10}:
\begin{equation}\label{eq:c11}
    \norm{u_{k,j}-u_k^*}\leq \frac{ \norm{\widetilde{g}_k}(\nu_k^*-1)(\xi_k^*)^{-(j-1)}}{\sqrt{\nu_k^*}(\lambda_{\min}(\widetilde{H}_k)+\sigma_k\norm{u_k^*})}.
\end{equation}
Also, suppose that $ j \geq N_{k}^{OR3}$. Then, by definition of the latter integer,
\begin{equation}\label{eq:c12}
    \frac{\sigma_k \norm{\widetilde{g}_k}(\nu_k^*-1)(\xi_k^*)^{-(j-1)}}{\sqrt{\nu_k^*}(\lambda_{\min}(\widetilde{H}_k)+\sigma_k\norm{u_k^*})}\leq \lambda_{\min}(\widetilde{H}_k)+\sigma_k\norm{u_k^*}.
\end{equation}
A triangle inequality gives
\begin{equation}\label{eq:c13}
    \sigma_k \norm{u_{k,j}}\geq \sigma_k \norm{u_k^*}-\sigma_k \norm{u_{k,j}-u_k^*}.
\end{equation}
By combining \eqref{eq:c11}-\eqref{eq:c13} we deduce that $\lambda_{\min}(\widetilde{H}_k)\geq -\sigma_k \norm{u_{k,j}}$.

Therefore, for any inner iteration $\tau_k \geq j \geq \max \{N_k^{OR2}, N_k^{OR3}\}$ of the Lanczos method, the step $u_{k,j}$ will always satisfy (OR1)-(OR3). If $\max\{N_k^{OR2},N_k^{OR3}\}>\tau_k$, then the Lanczos algorithm will terminate in at most $\tau_k$ iterations. Since $\tau_k\leq n-m$, the final bound follows.
\end{proof}

\begin{remark}
    In order to invoke results from \cite{cartis2020monotonic, gould2020error, jia2022solving}, in the proof above, we implicitly assumed that the subproblem \eqref{eq:mku} has a cubic term independent of the projection matrix. Observe that this reduction can be safely made since every candidate solution $u_{k,j}$ generated within a Krylov subspace always lies in $\mathcal{N}(A_k)$ by construction, and therefore satisfies $\norm{P_k^\top u_{k,j}}=\norm{u_{k,j}}$.
\end{remark}

\end{appendices}

\end{document}